\documentclass[12pt]{article}

\usepackage{amsmath}
\usepackage{amssymb}
\usepackage{amsthm}
\usepackage{graphicx}
\usepackage{caption}
\usepackage{fourier} 
\usepackage{soul} 

\usepackage{color}

\usepackage[margin=3cm]{geometry}

\usepackage[T1]{fontenc} 

\usepackage{nicefrac}

\usepackage{cancel} 
\usepackage{hyperref}

\newtheorem{theorem}{Theorem}
\newtheorem{proposition}{Proposition}
\newtheorem{lemma}{Lemma}

\theoremstyle{definition}
\newtheorem{definition}{Definition}

\theoremstyle{remark}
\newtheorem{remark}{Remark}

\def\RR{\mathbb{R}}

\def\11{\mathds{1}}

\def\E{\mathbb{E}}
\def\P{\mathbb{P}}
\def\R{\mathbb{R}}

\def\1{\mathbf{1}}
\def\N{\mathbb{N}}

\def\d{\partial}
\def\Z{\mathbb{Z}}

\def\cF{{\cal F}}

\def\cP{{\cal P}}

\newcommand{\sss}{\scriptscriptstyle}

\newcommand{\e}{\mathrm{e}}

\begin{document}

\title{Stochastic approximation on non-compact measure spaces and
application to measure-valued P\'olya processes.}
\author{C\'ecile Mailler$^1$ 
and Denis Villemonais$^{2}$}

\footnotetext[1]{University of Bath, Department of Mathematical Sciences, Claverton Down, BA2 7AY Bath, UK.\\ Email: \texttt{c.mailler@bath.ac.uk}}
\footnotetext[2]{Universit\'e de Lorraine, CNRS, Inria, IECL, UMR 7502, F-54000 Nancy, France.\\
Email: \texttt{denis.villemonais@univ-lorraine.fr}}

\maketitle

\begin{abstract}
  Our main result is to prove almost-sure convergence of a stochastic-approximation
  algorithm defined on the space of measures on a non-compact
  space. Our motivation is to apply this result to measure-valued P\'olya processes (MVPPs, also known as infinitely-many P\'olya urns). Our main idea is to use Foster-Lyapunov type criteria in a
  novel way to generalize stochastic-approximation methods to
  measure-valued Markov processes with a non-compact
  underlying space, overcoming in a fairly
  general context one of the major difficulties of existing studies
  on this subject.

  From the MVPPs point of view, our result implies almost-sure convergence of a large
  class of MVPPs; this convergence was only obtained until now for specific examples, with only convergence in probability established for general classes.
  Furthermore, our approach allows us to extend the definition of MVPPs by
  adding ``weights'' to the different colors of the
  infinitely-many-color urn.
  We also exhibit a link between non-``balanced'' MVPPs and quasi-stationary
  distributions of Markovian processes, which allows us to treat, for
  the first time in the literature, the non-balanced case. 

Finally, we show how our result can be applied to designing stochastic-approximation algorithms for the approximation of quasi-stationary distributions of discrete- and continuous-time Markov processes on non-compact spaces.
\end{abstract}
%
%
%
%

\section{Introduction}

Measure-valued P\'olya processes (MVPPs) are a generalization of P\'olya urns to
the infinitely-many-color case. 
{P\'olya urns date back to P\'olya \& Eggenberger~\cite{EP23}, and have been thoroughly studied since then; highlights include, e.g., the seminal works of Athreya \& Karlin~\cite{AK} and Janson~\cite{Janson04}.
Although the question of generalizing P\'olya urns to infinitely-many colors was posed in 2004 in~\cite{Janson04}, MVPPs were only introduced recently by 
Bandyopadhyay \& Thacker~\cite{BT++} and Mailler \& Marckert~\cite{MaillerMarckert2016}.}
In both papers, MVPPs are coupled with branching Markov chains on the random recursive tree.

The main idea of this article is to use stochastic-approximation methods (in the spirit of Duflo~\cite{Duflo} and Bena\"im~\cite{Benaim1999}) to prove almost-sure convergence of a class of MVPPs; the main difficulty comes from the fact that the stochastic-approximation algorithm that we consider is defined on the space of measures on a \emph{non-compact} space.

The stochastic-approximation approach is a classical method for the
study of P\'olya urn processes when the color-set is
finite.  For instance, in Section~2.2 of Bena\"im~\cite{Benaim1999}, the author
introduces the reformulation of the classical P\'olya urn model in
terms of stochastic approximations and provide some ideas for
generalizations; in Laruelle \& Pag\`es~\cite{LaruellePages2013}, 
the authors reformulate the study of several urn models in the setting of stochastic
approximations, with applications to clinical trials based on
randomized urn models (see also Laruelle \& Pag\`es~\cite{LaruellePages2013a} with
applications to optimal asset allocation in finance
and Zhang~\cite{Zhang2016} with applications to adaptive designs); we also
refer the reader to Pemantle~\cite{Pemantle2007}, which provides a survey of
random processes with reinforcement using stochastic-approximation
methods. 
Since stochastic approximation naturally applies to processes in general state spaces, 
it is natural to extend the above methods to the case of MVPPs. 

\bigskip
Our main contribution from the stochastic-approximation point of view
is to prove convergence of a stochastic-approximation
algorithm defined on a non-compact space, namely the set of
probability measures on the color-space (being an arbitrary Polish
space).  To our knowledge, very little is known for
  measure valued stochastic-approximation algorithm on
  non-compact spaces, with some exceptions such as~\cite{Janson03} 
  and~\cite{MP16}. In the first
  reference, Janson deals with the compactness issue by proving that
  the considered model can be restricted to finite subspaces; in the
  second one, Maillard \& Paquette prove that a specific stochastic
  approximation on the set of measures on $[0,\infty)$ converges
  almost surely, using an ad hoc coupling with the Kakutani and
    the uniform process.  Our generalization of
  measure-valued stochastic-approximation methods to
  non-compact state spaces is made by using abstract
  Foster-Lyapunov type criteria in an original way, 
  yielding the tightness of the stochastic-approximation algorithm. 

\bigskip
Our main contribution to the theory of MVPPs is 
to prove almost-sure convergence for a large class of MVPPs 
(instead of the convergence in probability shown by Mailler \& Marckert~\cite{MaillerMarckert2016}).
Furthermore, we generalize the 
definition of measure-valued P\'olya processes 
to allow different colors to have different ``weights'',
and to allow the so-called ``replacement rule'' to be random
(two features that are classical in the context of P\'olya urns).
We are also able to treat the ``non-balanced'' case, 
which was not treated at all by
Bandyopadhyay \& Thacker~\cite{BT++} or Mailler \& Marckert~\cite{MaillerMarckert2016}.

\bigskip
We believe that the applications of our results go beyond the field of
MVPPs: in particular, we detail an application to the approximation of
quasi-stationary distributions. Consider a Markov process that gets
absorbed when it reaches a state~$\partial$. A quasi-stationary
distribution (QSD), if it exists, is the limiting distribution of this
Markov process conditioned on not reaching $\partial$ (we refer
  the reader
  to~\cite{MeleardVillemonais2012,DoornPollett2013,ColletMartinezEtAl2013}
  for general introductions to quasi-stationary distributions).
Given an absorbed Markov process, it is in general a hard question to
prove existence and uniqueness of a QSD; an even harder question is to
find an explicit formula for it.  With many applications,
  including the study of interacting particle
  systems~\cite{OliveiraDickman2006,DeshayesRoll2017}, of population
  dynamics~\cite{VerboomLankesterEtAl,CattiauxColletEtAl2009}, of the
  simulation of metastable systems~\cite{DiLelievreEtAl2016} and of
  Monte-Carlo methods~\cite{WangKolbEtAl2017}, numerical approximation
  methods for quasi-stationary distributions have attracted a lot of
  interest during the last decades (see for instance~\cite{GrigorescuKang2004,FerrariMaric2007,GroismanJonckheere2013,Maric2015,OcafrainVillemonais2017}).
A recent method
introduced independently by Bena\"im \& Cloez~\cite{BC15} and by
Blanchet, Glynn \& Zheng~\cite{BlanchetGlynnEtAl2016} makes use of a
stochastic-approximation algorithm for computing quasi-stationary
distributions on finite state spaces. This method has been recently
extended to compact state space cases by Bena\"im, Cloez \&
Panloup~\cite{BCP++} and Wang, Roberts \&
Steinsaltz~\cite{WangRobertsEtAl2018}.
We show (see Section~\ref{sub:approx_QSD}) that our result can be
applied to prove almost-sure convergence of such QSD-approximation
algorithms for absorbed Markov processes taking values on a
non-compact space.

\bigskip

\subsection{Definition of the model and main result}
\label{sec:defmodel}
Throughout the article, $E$ is a Polish space endowed with its Borel
sigma-field. A measure-valued P\'olya process (MVPP) is a Markov chain
$(m_n)_{n\geq 0}$ taking values in the set of measures on a Polish
space $E$.  It depends on three parameters: its \textit{initial
  composition} $m_0$ a non-zero non-negative measure on
  $E$, a sequence of i.i.d.\ \textit{replacement}
kernels\footnote{A kernel (resp. a non-negative kernel) on $E$ is, by definition, a function from
  $E$ into the set of measures (resp. non-negative measures)  on $E$. In particular, for all
  $x\in E$, $R^{\sss (n)}_x$ is a measure on $E$ almost
  surely.}~$(R^{\sss (n)})_{n\geq 1}$ on~$E$, and a non-negative
\textit{weight} kernel $P$ on~$E$. We assume that
\begin{itemize}
  \item[(T${}_{>0}$)] almost surely, for all $x\in E$, $R^{\sss (n)}_x$ is a non-negative measure.
\end{itemize}

Given $m_n$, we define $m_{n+1}$ as
follows: pick a random element $Y_{n+1}$ of $E$ according to the
probability distribution proportional to $m_nP$, i.e., for all Borel
set $A$ of $E$,
    \begin{equation}
      \label{eq:law_Ynp1} \mathbb
      P(Y_{n+1}\in A \,|\, m_n) = \frac{\int_E P_x(A) \,\mathrm
        dm_n(x)}{\int_E P_x(E) \,\mathrm dm_n(x)};
    \end{equation}
    and then set
    \[
      m_{n+1} = m_n + R_{Y_{n+1}}^{\sss (n+1)}.
    \]

  Measure-valued P\'olya processes were originally introduced
    by~\cite{BT++} and~\cite{MaillerMarckert2016}, as a
    generalization of $d$-color P\'olya urns, although they did not consider ``weighted'' MVPPs
        (they always had $P_x = \delta_x$ for all
        $x\in E$). Let us recall
    the definition of a P\'olya urn and show why MVPPs generalize this model:
    A $d$-color P\'olya urn is a Markov process $(U(n))_{n\geq 0}$ on
    $\mathbb N^d$ that depends on three parameters: the initial
    composition vector $U(0)$, the replacement matrix $M$, and weights
    $w_1, \ldots, w_d\in(0,\infty)$.  The vector $U(n)$ represents the
    content of an urn that contains balls of $d$ different colors;
    balls of color $i$ all have weight $w_i$.  Given $U(n)$, one
    defines $U(n+1)$ by picking a ball at random in the urn with
    probability proportional to its weight, denoting the color of
    this random ball $\xi_{n+1}$, and setting
    $U(n+1) = U(n)+M_{\xi_n}$, where $M_1, \ldots, M_d$ are the lines
    of~$M$.

    If we let $E = \{1, \ldots, d\}$ and
    $m_n = \sum_{i=1}^d U_i(n)\delta_i$ for all $n\geq 0$, then $m_n$
    is a measure-valued P\'olya process with replacement kernel
    \[R^{\sss (n)}_x = \sum_{i=1}^d M_{x,i} \delta_i \quad(\text{
        almost surely for all } n\geq 0, 1\leq x\leq d),\] and weight
    kernel $P_x = w_x\delta_x$ for all $1\leq x\leq d$.

    Therefore, the MVPP process $(m_n)_{n\geq 0}$ can be thought of as
    a composition measure on a set $E$ of colors, and the random
    variable $Y_{n+1}$ can be seen as the color of the ``ball'' drawn
    at time $n+1$. The main advantage of this wider model is that one
    can consider P\'olya urns defined on an infinite, and even
    uncountable, set.

    \bigskip Our main result is to prove almost-sure convergence of
    the sequence $(\nicefrac{m_n}{m_n(E)})_{n\geq 0}$ to a
    deterministic measure under the following assumptions: We denote
    by~$R$ the common expectation of the $R^{\sss (n)}$'s and set
    $Q^{\sss (n)} = R^{\sss (n)}P$ for all $n\geq 1$, and $Q= RP$,
    meaning that, for all $x\in E$ and all Borel set $A\subseteq E$,
    \[Q_x^{\sss (n)}(A) = \int_E P_y(A) \,\mathrm dR^{\sss (n)}_x(y)
      \quad \text{ and }\quad
      Q_x(A) = \int_E P_y(A) \,\mathrm dR_x(y).\]
    We assume that
    \begin{itemize}
    \item[(A1)] for all $x\in E$, $Q_x(E)\leq 1$, and there exists a
      probability measure $\mu$ on $\R$ {\color{black}with positive mean}
      such that, for all $x\in E$, the law of $Q_x^{\sss (i)}(E)$
      stochastically dominates $\mu$. In particular, setting
      $c_1=\int_0^\infty x\,\mathrm d\mu(x)$,
      \[
        {\color{black}0<}c_1 \leq \inf_{x\in E} Q_x(E)\leq \sup_{x\in E} Q_x(E)\leq 1;
      \]
    \item[(A2)] there exists a locally bounded function $V\,:\,E \to [1,+\infty)$ such that, 
      \begin{itemize}
\item[(i)] for all $N\geq 0$, 
  the set $\{x\in E \colon V(x)\leq N\}$ is relatively compact;
\item[(ii)] there exist two constants $\theta\in(0, c_1)$ and $K\geq 0$ such that
\[Q_x \cdot V\leq \theta V(x) + K \quad (\forall x\in E),\]
\item[(iii)] and that there exist three constants  $r>1$, $p>\frac{\ln\theta}{\ln(\nicefrac{\theta}{c_1})}\vee 2$, 
$A>0$ such that
\[\E\left[R^{\sss (1)}_x(E)^r\right]\vee \E\left[Q^{\sss (1)}_x(E)^p\right]  \leq  A V(x) \quad (\forall x\in E).\]
\end{itemize}
\end{itemize}

Under Assumption (A1), $Q$ is a non-negative kernel such that
$\sup_x Q_x(E)\leq 1$, so that $Q-I$ is the jump kernel (or
infinitesimal generator) of a unique sub-Markovian transition kernel
$(P_t)_{t\geq 0}$ on $E$. We consider the continuous-time
pure-jump Markov process $(X_t)_{t\geq 0}$ on $E \cup \{\partial\}$,
where $\partial\notin E$ is an absorbing state, with Markovian
transition kernel $P_t +(1-P_t(E))\delta_\d$.
A probability distribution $\nu$ is a {\it quasi-stationary
distribution} of $(X_t)_{t\geq 0}$ if, and only if, 
there exists a probability measure $\alpha$ on $E$ such that, for all Borel
sets $A\subseteq E$,
\[
\mathbb P_\alpha(X_t\in A \,|\, X_t \neq \partial) \xrightarrow[t\rightarrow+\infty]{} \nu(A),
\]
where $\mathbb P_\alpha$ is the law of $X$ with initial distribution~$\alpha$.

\begin{itemize}
\item[(A3)] the continuous-time pure jump Markov process~$X$ with
  sub-Markovian jump kernel $Q-I$ admits a \textit{quasi-stationary
    distribution} $\nu\in\cP(E)$.  We further assume that the
  convergence of $\mathbb P_{\alpha}(X_t\in \cdot \,|\, X_t \neq \partial)$
  holds uniformly with respect to the total variation norm on
  $\{\alpha\in\cP(E)\mid \alpha \cdot V^{1/q}\leq C\}$, for each
  $C>0$, where $q = p/(p-1)$.
\end{itemize}
Finally, we need the following technical assumption:
\begin{itemize}
\item[(A4)] for all bounded continuous functions $f:E\rightarrow\R$,
  $x\in E \mapsto R_x f$ and $x\in E\mapsto Q_x f$ are continuous.
\end{itemize}

Under these assumptions, we are able to prove almost-sure convergence of the renormalized MVPP 
$\tilde m_n:= m_n/m_n(E)$:
\begin{theorem}
\label{thm:unbal-with-weights}
Under Assumptions (T${}_{>0}$) and (A1-- 4), if $m_0\cdot V<\infty$ and
$m_0P\cdot V<\infty$, then the sequence of random measures
  $(m_n/n)_{n\geq 0}$ converges almost surely to $\nu R$ with respect
  to the topology of weak convergence.  Moreover,
  $\sup_{n} \{{m_n P \cdot V^{\nicefrac1q}}/{n}\} <+\infty$ almost surely, where  $q=p/(p-1)$.

  Furthermore, if $\nu R(E)>0$, then $(\tilde{m}_n)_{n\in\N}$
  converges almost surely to $\nu R/\nu R(E)$ with respect to the
  topology of weak convergence.
\end{theorem}

\begin{remark}
  If $R=Q$, then the quasi-stationary distribution $\nu$ is a left
  eigenfunction for $R$, with associated eigenvalue
  $\theta_0\in(0,1]$. In particular,
  Theorem~\ref{thm:unbal-with-weights} implies that the average mass
  of $m_n$, i.e. $m_n(E)/n$, converges almost surely to $\theta_0$.
\end{remark}

\begin{remark}
  The main result holds under a weaker versions of
  Assumption~3: namely, the total variation distance can be replaced
  by any metric inducing the topology of weak convergence (or a stronger one).
\end{remark}

\begin{remark}
  \label{remNfinite}
  To illustrate how this theorem applies, let us first consider the
  simple case of a classical $d$-color P\'olya urn of 
    random replacement matrix $M^{{\color{black} (n)}}$ with no
  weights, {\color{black} where $(M^{\sss (n)})_n$ is a sequence of i.i.d.\ 
    random matrices with non-negative entries and mean $M$}.  We
  assume that $\sum_{i=1}^d M_{x,i} > 0$ for all $1\leq x\leq d$
  and that $M$ is irreducible.  Let
  $S=\max_{x=1}^d \sum_{i=1}^d M_{x,i}$, and let
  $m_n = \frac{1}{S}\sum_{i=1}^d U_i(n)\delta_i$, where $U_i(n)$ is the number of
  balls of color~$i$ in the urn at time~$n$. One can check that
  $(m_n)_{n\geq 0}$ is an MVPP on $E=\{1, \ldots, d\}$ with replacement kernel
  $R^{\sss (n)}_x=\frac1S\sum_{i=1}^d
  M^{\sss (n)}_{x,i}\delta_i$, for all $n\geq 0$ and
  $1\leq x\leq d${\color{black}, such that $R=\nicefrac{M}{S}$}.

  Note that, since we have no weights, $R=Q$. Let $\mu$ be the distribution of
  $\min_{x\in \{1,\ldots,d\}} X_x$, where $X_1,\ldots,X_d$ are
  independent random variables respectively distributed as
  $Q_1^{\sss (1)}(E), \ldots,$ $Q_d^{\sss (1)}(E)$.
  Assumption (A1) is satisfied since $\mu$ has positive mean $c_1\leq Q_x(E)\leq 1$
  for all $1\leq x\leq d$. Assumption (A2) is automatically satisfied since the color
  space $E$ is compact. Consider the process $X$ on
  $E\cup\{\partial\}$ absorbed at $\d$ and whose jump matrix
  restricted to $E$ is given by $\nicefrac MS-I$.  Then, since $\nicefrac MS$ is
  irreducible, the process $X$ conditioned on not hitting $\partial$
  has a unique quasi-stationary distribution
  $\nu = \sum_{i=1}^n v_i\delta_i$, which is given by the unique
  non-negative left eigenvector~$v$ of~$\nicefrac MS-I$ and hence of~$M$.
  { It is also known (see e.g.\ Darroch \&
    Seneta\cite{DarrochSeneta1967}) that there exists $C, \delta>0$ such
    that
    $\|\mathbb P_{\alpha}(X_t\in\cdot|X_t\notin \partial)-\nu\|_{\sss
      TV}\leq C \mathtt e^{-\delta t}$ for all $\alpha\in\mathcal P(E)$,
    which thus implies (A3).  } Finally, Assumption (A4) is trivially
  satisfied since $E$ is discrete.

  Thus, Theorem~\ref{thm:unbal-with-weights} applies, and we get that,
  almost surely when $n$ tends to infinity,
  $\tilde m_n \to \nu R/\nu R(E) = \nu$ (with respect to the topology
  of weak convergence), and thus, $U(n)/n\to v$, a result that dates
  back to Athreya \& Karlin's work on generalized P\'olya
  urns~\cite{AK}.
\end{remark}

\begin{remark}
  In the original P\'olya urn model, the replacement matrix is the
  identity and is not irreducible. In this case, there are several
  quasi-stationary distributions and thus Assumption~(A3) fails. We
  may thus say that the equivalent of the irreducible assumption in
  Athreya \& Karlin's result is our Assumption~(A3).
\end{remark}

In Section~\ref{sec:examples} we apply our result to many more
examples, and, in particular, to examples where the color space $E$ is
infinite, and even non-compact. Before that, in the rest of this
introduction, we discuss our result and its assumptions.

\subsection{Discussion of the result in view of the existing literature on MVPPs}
Our definition of a measure-valued P\'olya process is more general than the definition of Bandyopadhyay \& Thacker~\cite{BT++} and Mailler \& Marckert~\cite{MaillerMarckert2016}; indeed, their model can be obtained from ours by taking $R^{\sss (i)} = R$ almost surely for all $i\geq 1$ (deterministic replacement rule), and $P_x = \delta_x$ for all $x\in E$ (no weights). \cite{BT++} and~\cite{MaillerMarckert2016} also make the following assumptions:
\begin{itemize}
\item[(I)] $0<m_0(E)<+\infty$;
\item[(B)] for all $x\in E$, $R_x(E) = 1$;
\item[(E)] there exist two sequences $(a_n)_{n\geq 0}$ and
  $(b_n)_{n\geq 0}$ such that the Markov chain $(W_n)_{n\geq 0}$ on
  $E$ of transition kernel $(R_x)_{x\in E}$ satisfies
\[\frac{W_n - b_n}{a_n} \Rightarrow \nu,\]
in distribution when $n$ goes to infinity, independently from the initial distribution of~$W_0$.
\item[(R)] the sequences $(a_n)_{n\geq 0}$ and $(b_n)_{n\geq 0}$ are such that, for all $\varepsilon_n = o(\sqrt n)$, for all $x\in \mathbb R$,
\[\lim_{n\to\infty} \frac{b_{n+x\sqrt n + \varepsilon_n}-b_n}{a_n} = f(x) \quad \text{ and }\quad
\text{ and }
\lim_{n\to\infty} \frac{a_{n+x\sqrt n + \varepsilon_n}}{a_n} = g(x),\]
where $f$ and $g$ are two measurable functions,
\end{itemize}
The names of the assumptions are (I) for {\it initial composition}, (B) for {\it balance}, 
(E) for {\it ergodicity} and (R) for {\it regularity}.
Under these assumptions Mailler \& Marckert~\cite{MaillerMarckert2016} prove that (a slightly weaker version of this result is proved by~\cite{BT++}):
\begin{theorem}[Mailler \& Marckert~\cite{MaillerMarckert2016}]\label{th:MM}
If $(m_n)_{n\geq 0}$ is a MVPP that satisfies assumptions (I), (B), (E) and (R), then
\begin{equation}\label{eq:cv_MVPP}
n^{-1}m_n (a_{\log n} \,\cdot\, + b_{\log n}) \to \mu,
\end{equation}
in probability when $n$ goes to infinity,
for the topology of weak convergence,
where $\mu$ is the distribution of $f(\Lambda) + g(\Lambda)\Phi$, where $\Lambda\sim \mathcal N(0,1)$ and 
$\Phi\sim \nu$ are independent.
\end{theorem}

Note that Theorem~\ref{thm:unbal-with-weights} applies under (I), (B), (E) and (R) if we assume additionally that $a_n\equiv 1$ and $b_n\equiv 0$, and it gives that
\[\frac{m_n}{n} \to \nu \quad \text{ almost surely},\]
which improves the convergence in probability of Theorem~\ref{th:MM}.
Our theorem though does not cover the cases of more general renormalization 
sequences $(a_n)_{n\geq 1}$ and $(b_n)_{n\geq 1}$.

\medskip
In summary, our main contributions to the theory of MVPPs are to
\begin{itemize}
\item[($\alpha$)] remove the balance hypothesis (B) and replace it by the weaker (A1);
\item[($\beta$)] prove convergence almost sure in Equation~\eqref{eq:cv_MVPP} when $a_n \equiv 1$ and $b_n \equiv 0$;
\item[($\gamma$)] allow the weighting of the different elements of $E$, and to
\item[($\delta$)] {allow the} re-sampling {of} the replacement measures at each time-step in an i.i.d.\ way.
\end{itemize}
Our result was motivated by the classical P\'olya urn theory (see e.g.~\cite{Janson04}), in which all these features are standard. Since this paper was submitted, Janson~\cite{Janson19} generalised Theorem~\ref{th:MM} to the random replacement case, thus treating ($\gamma$) in that case. Also, Bandyopadyhay, Janson \& Thacker~\cite{BJT} prove almost sure convergence of a class of balanced MVPPs for which the set of colours is countable and under a condition of strong ergodicity for the underlying Markov chain, thus treating ($\beta$) in that case.

\begin{remark}
A standard generalization of finitely-many-color P\'olya urns is indeed to add weights (or activities): each color $x$ is given a weight $w(x)$, and, at every time-step, one picks a ball in the urn with probability proportional to the weights (vs.\ uniformly at random in the non-weighted model) and then applies the replacement rule associated to this color (see, e.g.~\cite{Janson04}). 
In our model, if $P_x = w(x)\delta_x$, where $w(x)$ is non-negative, then
\[\mathbb P(Y_{n+1}\in A \,|\,  m_n) = \frac{\int_A w(x) \,\mathrm dm_n(x)}{\int_E w(x) \,\mathrm dm_n(x)},\]
which corresponds to weighting the color~$x$ by a weight~$w(x)$.
The introduction of a weight kernel is a generalization of the weight concept: one can for example see~$P$ as a noise on the color drawn at random.
\end{remark}

\begin{remark}
  Our model, assumptions and result can be easily adapted to the
  situation where $R^{\sss (1)}$ is a kernel from $E$ to an other
  Polish state space $F$ and $P$ is a non-negative kernel from~$F$ to~$E$. 
  The main point of this extension is to check that the proof
  of Theorem~\ref{thm:unbal-with-weights}
  mainly makes use of the properties of
  the composed kernel $Q^{\sss (1)}$. For instance, in the $d$-color
  P\'olya urn model (see the end of Subsection~\ref{sec:defmodel}), if
  $\sum_{j=1}^d M_{i,j}>0$ for all $i\in\{1,\ldots,d-1\}$ and if
  $\sum_{j=1}^d M_{d,j}=0$, then one can choose $E=\{1,\ldots,d-1\}$
  and $F=\{1,\ldots,d\}$ together with the kernels
  $R^{\sss (i)}_{i,j}=R_{i,j}=M_{i,j}/S$ for all $(i,j)\in E\times F$
  and $P_{ij}=\1_{i\neq d}\delta_i$ for all $i\in F$. In this case, we
  thus have $Q^{\sss (1)}_{i,j}=Q_{i,j}=M_{i,j}/S$ for all
  $(i,j)\in E\times E$. If $M$ restricted to $E\times E$ is
  irreducible, we get that there exists a unique quasi-stationary
  distribution~$\nu$ on~$E$ for the continuous time Markov process~$X$
  with infinitesimal generator~$Q-I$ (see~\cite{DarrochSeneta1967}). 
  Hence, using our approach to MVPPs
  in this slightly more general context, we get that the
  $d$-color P\'olya urn converges almost surely, when
  $n\rightarrow+\infty$, to $\nu R/\nu R(E)$ 
  (which is a probability measure on $F$),
  a result that can be found, e.g., in~\cite{Janson04}.
\end{remark}

\begin{remark}
  The main idea in~\cite{BT++} and~\cite{MaillerMarckert2016} is to
  show a link between the MVPP of replacement kernel $R$ and the
  Markov chain of kernel $R$. This relationship breaks down if the
  balance assumption is not satisfied since $R$ is no longer a
  probability kernel but a sub-Markovian kernel (we can assume without
  loss of generality that the upper bound of $\sup_x R_x(E)$
  is~$1$). Our main idea to relax the balance assumption is to add an
  absorbing state $\partial$ that ``makes'' the transition kernel Markovian; 
  note that this idea is similar to adding ``dummy'' balls in the
  finitely-many-color case (see \cite{Janson04}).  The ergodicity
  assumption (E) then naturally becomes Assumption (A3) that the Markov
  chain has a quasi-stationary distribution.

  The link between P\'olya urns and quasi-stationary distributions
  already exists in the literature; for example, Aldous, Flannery and
  Palacios~\cite{AldousFlanneryEtAl1988} apply the convergence results
  of Athreya and Karlin~\cite{AK} to approximating quasi-stationary
  distributions on a finite state space.  Our main result generalizes
  this work to the case of measure-valued P\'olya processes.
\end{remark}

\begin{remark}
Another difference with~\cite{BT++} and~\cite{MaillerMarckert2016}
is that Theorem~\ref{thm:unbal-with-weights} naturally covers periodic
transition kernels since we consider the continuous time process
associated to it, which is never periodic.
\end{remark}

\subsection{Discussion of the assumptions}
\label{sub:disc}

In Assumption~(A1), we assume that $Q_x(E)$ is uniformly bounded from
above by~$1$. If the supremum $\kappa=\sup_{x\in E} Q_x(E)$ is finite
(but larger than~$1$), one can consider the process defined by
$\hat m_n := m_n/\kappa$ for all $n\geq 0$. One can easily check that
$\hat m_n$ is an MVPP with parameters
$\hat{R}^{\sss (i)}=R^{\sss (i)}/\kappa$, $\hat P=P$, and
$\hat Q=\hat R \hat P$, and such that $\hat m_0=m_0/\kappa$. Also, it
satisfies $\hat Q_x(E)\leq 1$ as in Assumption~(A1).

For the lower bound, we assume that 
the random value $Q_x^{\sss (i)}(E)$ stochastically dominates 
an integrable probability measure~$\mu$ on~$\R$ with mean $c_1>0$. 
This is used to prove that, for any fixed $c'\in(\theta,c_1)$
\[
  \liminf_{n\rightarrow+\infty} \,\frac{m_nP(E)}{n}\geq c',
\]
almost surely; this is done by a coupling argument (see
Lemma~\ref{lem:cv_sigma_k}).  An alternative assumption, which may be
particularly useful when $Q_x^{\sss (i)}(E)$ can take negative values
as in Subsection~\ref{sec:remball} below, is that there exist $c_1>0$
and $\beta>1$ such that
\begin{align}
  \label{eq:alt-A1}
 c_1 \leq \inf_{x\in E} Q_x(E)\leq \sup_{x\in E} Q_x(E)\leq 1\quad\text{ and }\quad  \sup_{x\in E}\, \E\left|Q_x^{\sss (i)}(E)-Q_x(E)\right|^{\beta}<+\infty.
\end{align}
For instance, in the example developed in Remark~\ref{remNfinite},
take $E=\{1,2\}$ and
\[
  M^{\sss (n)}=
  \varepsilon_n
  \begin{pmatrix}
      -1& 0\\
      0& 1
    \end{pmatrix}
    + (1-\varepsilon_n)
    \begin{pmatrix}
      1& 2\\
      1& 0
    \end{pmatrix},
  \]
  where $(\varepsilon_n)_{n\geq 1}$ is a sequence of i.i.d.\ Bernoulli
  random variables with parameter $\nicefrac12$. Then any probability
  measure $\mu$ on $E$ as in Assumption~(A1) has non-positive mean, so
  that this assumption is not satisfied. However,
  Assumption~\eqref{eq:alt-A1} is satisfied with $c_1=1$.

\medskip
Assumption (A2) is a Lyapunov assumption and is standard in the study
of the ergodicity of Markov processes. {\color{black} In
  Section~\ref{sec:examples}, we show how to apply our main result to
  examples, and therefore give examples of such Lyapunov
  functions. There is no general method to find Lyapunov functions,
  except testing functions from classical families (polynomials,
  exponentials, etc). For instance, for processes in $\Z$, $\R$ or
  $\R^d$ with a drift towards~$0$, exponential or power functionals of
  the distance to~$0$ often prove to be useful. Sometimes,
  probabilistic arguments can help find a Lyapunov function; indeed,
  if, for some $\theta\in(0,1)$, $\E_x \left[\theta^{\tau_K}\right]$
  is finite for all $x\in E$ (where $\tau_K$ denotes here the first
  entry time in a set $K$ of a discrete-time Markov chain with
  transition probability given by $Q$), then
  $V : x\mapsto \E_x\left[\theta^{\tau_K}\right]$ satisfies
  $Q_x\cdot V\leq \theta V(x)$ for all $x\in E\setminus K$.}

When
$Q_x(E)=1$ for all $x\in E$, the existence of a Lyapunov function for
$Q$ can be used to prove the ergodicity of the Markov process~$X$.
More precisely, if compact subsets of $E$ are \textit{petite sets} for
$X$, then the existence of a Lyapunov function entails the ergodicity
of~$X$ (see Meyn \& Tweedie~\cite{MeynTweedie1993}, for the definition
of a petite set and for the deduction that~$X$ is ergodic) and hence
Assumption~(A3). Note that our proof does not seem to generalize to
the case of a weaker form of Lyapunov function (satisfying, for
instance, $Q_x(V)\leq V(x)-V^{1/2}(x)+C$ for all $x\in E$), although
those weaker forms are generally sufficient to prove the ergodicity of
the process.

When $Q$ is a sub-Markovian kernel, it has been recently proved
in Champagnat \& Villemonais~\cite{ChampagnatVillemonais2017} that the Lyapunov condition
(A2-ii), with additional suitable assumptions, can be used to prove
the existence of a quasi-stationary distribution $\nu$ and to prove
that the domain of attraction of $\nu$ contains
$\{\alpha\in\cP(E)\mid \alpha \cdot V^{\nicefrac1q}<\infty\}$. These criteria will be used extensively
in our examples. Note that this result, when applicable, entails the
existence of a quasi-stationary distribution $\nu$ and the uniform convergence
of Assumption~(A3) in total variation norm.

For conditions implying Assumption~(A3), we also refer the reader to Villemonais~\cite{Villemonais2015}
where the case of birth and death processes is considered,
to Gosselin~\cite{Gosselin2001}, and Ferrari, Kesten \& Mart\'inez~\cite{FerrariKestenEtAl1996} for population processes
and the utility of the theory of $R$-positive matrices in this
matter. This is also implied by the general results provided in Champagnat \& Villemonais~\cite{ChampagnatVillemonais2016}.

\subsection{Removing balls from the urn}
\label{sec:remball}

In the finitely-many-color case, it is often allowed to remove balls
from the urn, i.e.\ the coefficients of the replacement matrix can be
negative.  In Theorem~\ref{thm:unbal-with-weights}, we have assumed
that the measures~$(R_x)_{x\in E}$ are positive, but we can in fact
consider situations where $(R_x)_{x\in E}$ are signed kernels as soon as they
satisfy additional assumptions (which are already implied by
conditions (A1-4) when $(R_x)_{x\in E}$ are positive measures). In
Section~\ref{sec:examples}, we give examples that fall into this
special framework.

In this section, we assume that $(R^{\sss (i)}_x)_{x\in E}$ is almost surely a signed kernel such that, for all $x\in E$, $Q_x$ restricted to $E\setminus \{x\}$ is a positive measure and $Q_x(\{x\})\in \R$. We assume that
\begin{itemize}
\item[(T)] for all $n\geq 0$, $m_n$ is almost surely a positive measure.
\end{itemize}
In the finitely-many-color case, this assumption is called
tenability. It is clearly satisfied when Assumption~(T${}_{>0}$) holds
true. We refer the reader to~\cite[Definition~1.1-(iii)]{Pouyanne2008}
for a sufficient condition for tenability in the finite state space
case.  As will appear in the examples section, tenability is often naturally satisfied.

In the case when $(R^{\sss (i)}_x)_{x\in E}$ is allowed to be a signed kernel, we need
to replace Assumption~(A2) by:
\begin{itemize}
\item[(A'2)] there exist a locally bounded function
  $V\,:\,E \to [1,+\infty)$ and some constants { $r>1$, $p> 2$,
  $q'>q:=p/(p-1)$,} $\theta\in(0,c_1)$, $K>0$, $A\geq 1$, and $B\geq 1$, such that
\begin{itemize}
\item[(i)] for all $N\geq 0$, 
the set $\{x\in E \colon V(x)\leq N\}$ is relatively compact.
\item[(ii)]  for all $x\in E$,
\[Q_x \cdot V\leq \theta V(x) + K \quad\text{and}\quad Q_x \cdot V^{\nicefrac1q}\leq \theta V^{\nicefrac1q}(x) + K \quad (\forall x\in E).\]
\item[(iii)] for all continuous functions $f:E\rightarrow\R$ bounded by~$1$ and all $x\in E$,
  \[
|Q_x\cdot f|^{q'} \vee \E\big[\big| {R}^{\sss (i)}_{x} \cdot f - R_{x}\cdot
    f\big|^r\big]\vee \E\big[\big| {Q}^{\sss (i)}_{x} \cdot f -
    Q_{x}\cdot f\big|^p\big] \leq AV(x),
  \]
\item[(iv)] and
  \[
  |Q_x\cdot V^{\nicefrac1q}|^{q}\vee |Q_x\cdot V|\vee  \E\left[\left|Q^{\sss (i)}_x\cdot V^{\nicefrac1q}-Q_x\cdot V^{\nicefrac1q}\right|^{r}\right] \leq B V(x).
  \]
\end{itemize}
\end{itemize}

Assuming in addition that Assumptions (A1), (A3) and (A4) are satisfied, the
conclusions of Theorem~\ref{thm:unbal-with-weights} hold true. Since
the set of assumptions (T, A1, A'2, A3, A4) is actually implied (see Lemma~\ref{lem:A->A'} below) by the assumptions of Theorem~\ref{thm:unbal-with-weights}, we prove this
result in the more general situation of the present subsection. 

{
\begin{lemma}\label{lem:A->A'}
Assumptions (T${}_{>0}$, A1-4) imply Assumptions (T, A1, A'2, A3, A4).
\end{lemma}
\begin{proof}
  The fact that Assumption~(T${}_{>0}$) implies Assumption~(T) is
  straightforward.  Fix $q=p/(p-1)$; using H\"older's inequality
  ($q\geq 1$) and Assumption (A2-ii), we get, for all $x\in E$,
\[(Q_x\cdot V^{\nicefrac1q})^q \leq Q_x(E)^{\nicefrac{q}{p}}\,Q_x\cdot V
\leq \theta V(x) + K.
\]
Using the fact that, by concavity, for all $a\leq 1$ and $u\geq 0$, 
$(1+u)^a \leq 1 + u^a$, we thus get
\[Q_x\cdot V^{\nicefrac1q}
\leq \theta^{\nicefrac1q} V^{\nicefrac1q}(x) + K^{\nicefrac1q}.\]
To prove (A'2-ii), it is thus enough to show that $\theta^{\nicefrac1q} < c_1$. This follows since, by assumption on~$p$,
\[\frac1q \ln\theta = (1-\nicefrac1p) \ln\theta
< \left(1-\frac{\ln(\nicefrac\theta{c_1})}{\ln\theta}\right)\ln \theta
= \ln c_1.
\]
Now we prove (A2-iii); first note that, since $q':=p>q>1$, we have, by Jensen's inequality, for all continuous function bounded by~1,
\[|Q_x\cdot f|^{q'}\leq \mathbb E\left[|Q^{\sss (1)}_x\cdot f|^{q'}\right]
\leq  \mathbb E\left[Q^{\sss (1)}_x(E)^{q'}\right]\leq AV(x),
\]
where we have used (A2-iii).  Similarly, for all $r'\in(1,r]$, using
the convexity of $u\mapsto u^{r'}$ and Jensen's inequality,
we get that,
\[\mathbb E\left[|R_x^{\sss (1)}\cdot f - R_x\cdot f|^{r'}\right]
  \leq 2^{r'-1}\mathbb E\left[|R_x^{\sss (1)}\cdot f|^{r'}+| R_x\cdot
    f|^{r'}\right] \leq 2^{r'}\mathbb E \big[|R^{\sss (1)}_x\cdot
  f|^{r'}\big]\leq A2^{r'-1}V(x),\] and similarly for
$\mathbb E\big[|Q_x^{\sss (1)}\cdot f - Q_x\cdot f|^p\big]$.

It only remains to prove (A2-iv). We have, using H\"older's inequality, the fact that $Q_x$ is non-negative and the fact that $V(x)\geq 1$,
\[
  \left|Q_x\cdot V^{\nicefrac1q}\right|^q= \left(Q_x\cdot V^{\nicefrac1q}\right)^q\leq Q_x(E)^{\nicefrac{q}{p}}\,Q_x\cdot V\leq Q_x\cdot V \leq \theta\,V(x)+K\leq (\theta+K)\,V(x).
\]
Then, using the convexity of $u\mapsto u^{r'}$ and Jensen's inequality,
we get that
\begin{align*}
  \mathbb E\left[|Q_x^{\sss (1)}\cdot V^{\nicefrac1q} - Q_x\cdot V^{\nicefrac1q}|^{r'}\right] \leq
  2^{r'-1}\mathbb E\left[|Q_x^{\sss (1)}\cdot V^{\nicefrac1q}|^{r'}+| Q_x\cdot
    V^{\nicefrac1q}|^{r'}\right] \leq 2^{r'}\mathbb E \big[(Q^{\sss (1)}_x\cdot
  V^{\nicefrac1q})^{r'}\big]. 
\end{align*}
Now, using H\"older's inequality, we obtain
\begin{align*}
 Q^{\sss (1)}_x\cdot
  V^{\nicefrac1q}\leq (Q^{\sss (1)}_x\cdot V)^{1/q}Q^{\sss (1)}_x(E)^{1/p}.
\end{align*}
Using again H\"older's inequality, we have, setting $\varpi=\nicefrac{q}{r'}$, 
\begin{align*}
  \mathbb E \big[(Q^{\sss (1)}_x\cdot V^{\nicefrac1q})^{r'}\big]
  &\leq \mathbb E \big[(Q^{\sss (1)}_x\cdot V)^{\frac{\varpi r'}{q}}\big]^{1/\varpi}\,
    \mathbb E \big[Q^{\sss (1)}_x(E)^{\frac{r'\varpi}{p(\varpi-1)}}\big]^{(\varpi-1)/\varpi}\\
  &\leq  \mathbb E \big[Q^{\sss (1)}_x\cdot V\big]^{1/\varpi}
   \, \mathbb E \big[1+Q^{\sss (1)}_x(E)^p\big]^{(\varpi-1)/\varpi},
\end{align*}
where we used that ${\frac{r'\varpi}{p(\varpi-1)}}=\frac{r' (q-1)}{q-r'}\leq p$ for $r'$
small enough in $(1,r]$. Using Assumption~(A2-ii), we get $\mathbb E \big[Q^{\sss (1)}_x\cdot V\big]=Q_x(V)\leq (\theta+K)V(x)$ and, using Assumption~(A2-iii), $\mathbb E \big[Q^{\sss (1)}_x(E)^p\big]\leq A V(x)$. We finally deduce that
\[
  \mathbb E \big[(Q^{\sss (1)}_x\cdot V^{\nicefrac1q})^{r'}\big]\leq (\theta+K+1+A)\,V(x),
\]
where we have used that $1/\varpi+(\varpi-1)/\varpi=1$.
This concludes the proof.
\end{proof}
}

\begin{remark}
  { When $Q_x(\{x\})$ is not bounded uniformly in $x$,
    the infinitesimal generator $Q-I$ may not
    define a unique sub-Markovian transition kernel $(P_t)_{t\geq 0}$,
    and hence a unique 
    pure jump Markov process~$X$ (in distribution). The
    problem of existence and uniqueness of such a transition kernel
    has been considered in great generality by Feller
    in~\cite{Feller1940} and is also studied in details
    in~\cite[Chapter~2]{Chen2004}.  In our case, Assumption~(A2-ii) and
    Theorem~\cite[Theorem~2.25]{Chen2004} imply that $Q-I$ uniquely
    determines a sub-Markovian semi-group $(P_t)_{t\in[0,+\infty)}$ and
    hence a unique jump-process $X$ (in distribution). As a consequence,
    Assumption~(A3) remains unambiguous when Assumption~(T${}_{>0}$,A2) is
    replaced by Assumption~(T,A'2).}
\end{remark}

\bigskip {\bf Plan of the paper: }In Section~\ref{sec:examples}, we
apply Theorem~\ref{thm:unbal-with-weights} to several examples.  In
particular, in Section~\ref{sub:trees}, we look at examples that come
from studying different characteristics (degree distribution,
protected nodes) in random recursive trees or forests. In
Section~\ref{sub:samplepaths}, we detail the case when the replacement
kernels are the occupation measures of Markov processes, in discrete
and continuous time, and show how one can apply these results to the
numerical approximation of QSDs on a non-compact space (see
Section~\ref{sub:approx_QSD}).  Finally, Section~\ref{sec:proof}
contains the proof of Theorem~\ref{thm:unbal-with-weights}.

\section{Examples}\label{sec:examples}

\subsection{Markov chains}
\subsubsection{Ergodic Markov chains}
In~\cite{MaillerMarckert2016}, the following example is treated:
take $E=\mathbb N := \{0,1,2,\ldots\}$, fix $0<\lambda<\mu$, and set
\[R_x = \frac{\lambda}{x\mu+\lambda} \delta_{x+1} + \frac{x\mu}{x\mu+\lambda}\delta_{x-1},\]
for all $x\neq 0$, and $R_0 = \delta_1$.
This example is not weighted, meaning that $P_x = \delta_x$ for all $x\in E$, 
and balanced since $R_x(E) = 1$ for all $x\in E$.
Note that the Markov chain of transition kernel $R$ is the $M/M/\infty$ queue.
Theorem~\ref{th:MM} implies that this MVPP satisfies
\[n^{-1}m_n \to \gamma \quad \text{ in probability},\]
where $\gamma$ is the stationary measure of the $M/M/\infty$ queue, i.e.\
\[\gamma(x) = \left(\frac{\lambda}{\mu}\right)^x \frac{\mathrm e^{-\lambda/\mu}}{x!} \quad (\forall x\in\mathbb N).\]
Let us show how our result implies almost-sure convergence of this
MVPP.  Note that, in this example, the $R^{\sss (i)}$ are
deterministic and equal to $R$, $P_x = \delta_x$; therefore,
$Q^{\sss (i)}=Q=R$ ($\forall i\geq 1$). Since $R_x(E) = 1$ for all
$x\in\mathbb N$, then (A1) is satisfied (we can take $\mu=\delta_1$,
and thus, $c_1 = 1$).  Assumption (A2) also holds: one can take
$V(x) = \mathrm e^x$, implying that
\[R_x \cdot V = \frac{\lambda \mathrm e^{x+1}+\mu x \mathrm e^{x-1}}{\lambda+\mu x}
= \frac{\lambda \mathrm e^2 + \mu x}{\lambda+\mu x} \, \mathrm e^{x-1}
= \frac{\lambda \mathrm e^2 + \mu x}{\mathrm e(\lambda+\mu x)}V(x).\]
Note that
\[\frac{\lambda \mathrm e^2 + \mu x}{\mathrm e(\lambda+\mu x)} < \frac{2}{\mathrm e} \Leftrightarrow x > \frac{\lambda(\mathrm e^2-2)}{\mu},\]
therefore,
\[R_x\cdot V \leq \theta\, V(x) + K,\]
where $\theta=\frac{2}{\mathrm e}\in(0,c_1)$ and $K = \sup_{x\leq \lambda(\mathrm e^2-2)/\mu} R_x \cdot V$.
Also note that, for all $r,p>1$, we have
\[\mathbb E R_x^{\sss (1)}(E)^r \vee \mathbb E Q_x^{\sss (1)}(E)^p
  =R_x(E)^r \vee R_x(E)^p = 1,\] implying that (A2-iii) holds.  Since
the queue $M/M/\infty$ is ergodic with stationary distribution
$\gamma$, we can infer that the continuous-time Markov process of
generator $R-I$ is also ergodic and the domain of attraction of
$\gamma$ is $\mathcal P(\mathbb N)$. Moreover, the
same procedure as in the proof of Lemma~\ref{lem:A->A'} shows that,
for any $q>1$,
$Q_x\cdot V^{\nicefrac1q}\leq \theta^{1/q} V(x)+K^{\nicefrac1q}$,
where $\theta^{\nicefrac1q}< 1$. This and the Foster-Lypanuov type
criteria of~\cite{MeynTweedie1993} provide the uniform convergence to
$\nu$ required in Assumption (A3).  Finally, since $\mathbb N$ is
discrete, (A4) is trivially satisfied.  Thus,
Theorem~\ref{thm:unbal-with-weights} applies and we can conclude that
if $\sum_{k\geq 0} \mathrm e^k m_0(k)$ is finite, then
\[n^{-1}m_n \to \gamma\quad\text{ almost surely when }n\to\infty.\]

\subsubsection{Quasi-ergodic Markov chains}

Let us now consider the more general case where $E=\N$ and, for all $x\in E$,
\[
  R_x=\lambda_x\delta_{x+1}+\mu_x\delta_{x-1},
\]
where $(\lambda_x)_x$ and $(\mu_x)_x$ are families of positive numbers
such that  $\mu_0=0$, $\lambda_0>0$, $\inf_{x\geq 1} \mu_x>0$, $\sup_x \mu_x<\infty$ and
$\lambda_x=o(\mu_x)$ when $x\rightarrow+\infty$.  In this situation,
the MVPP is not weighted, so that $P_x=\delta_x$ and $Q_x=R_x$ for all
$x\in E$, and it is not balanced (hence Theorem~\ref{th:MM} does not apply).

We assume, without loss of generality, that
$\sup_x (\lambda_x+\mu_x)=1$, so that $Q_x(E)\leq 1$ for all $x\in
E$. Let $\mu$ be the Dirac mass at $\inf_x (\lambda_x+\mu_x)$, which
is positive. 
 Assumption
(A1) is satisfied with this choice of $\mu$, and
$c_1 = \inf_x (\lambda_x+\mu_x)$.  Let
\[
  V(x)=\mathrm e^{ax}\quad\text{with}\quad a>0\text{ such that }e^{-a}\leq c_1/4. 
\]
Assumption (A2-i) is clearly satisfied, and (A2-ii) can be checked easily: 
for all $x\in E$, 
\begin{align*}
  Q_x\cdot V&=\lambda_x\mathrm e^{a(x+1)}+\mu_x\mathrm e^{a(x-1)}
  =V(x)\left(\lambda_x \mathrm e^a+\mu_x \mathrm e^{-a}\right)\\
       &\leq V(x)\sup_y\mu_y \left(\frac{\lambda_x}{\mu_x}\mathrm e^a
       +\mathrm e^{-a}\right)\leq V(x) \left(\frac{\lambda_x}{\mu_x}\mathrm e^a+\frac{c_1}{4}\right)\\
       &\leq \theta V(x)+K,
\end{align*}
where $\theta=\frac{c_1}{2}$ and
$K=\max\left\{V(y)
  \left(\frac{\lambda_y}{\mu_y}\mathrm e^a+\frac{c_1}{4}\right),\text{ with
  }y\text{ s.t. }\frac{\lambda_y}{\mu_y}\mathrm e^a+\frac{c_1}{4}\geq
  \frac{c_1}{2}\right\}$ (note that this last set is finite by
assumption and hence that $K<\infty$). Since $R_x(E)=Q_x(E)$ is
uniformly bounded from above, (A2-iii) is trivial for any fixed $p>2\vee \frac{\ln \theta}{\ln\theta-\ln c_1}$. 
Assumption~(A4) is also clearly satisfied in this case since $E$ is discrete.

The same procedure as in the proof of Lemma~\ref{lem:A->A'} shows that $Q_x\cdot V^{\nicefrac1q}\leq \theta^{1/q} V(x)+K^{\nicefrac1q}$, where $\theta^{\nicefrac1q}< c_1$ since we fixed $p>\frac{\ln \theta}{\ln\theta-\ln c_1}$. Now, using Theorem~5.1 and Remark~11 in~\cite{ChampagnatVillemonais2017}
for the irreducible process $X$ with infinitesimal generator $Q-I$,
we deduce that there exist a quasi-stationary distribution
$\nu_{QSD}$ for $X$ and two positive constants $\mathrm{Cst},\delta>0$ such that, 
for all probability measure $\alpha\in E$,
satisfying $\alpha\cdot V^{\nicefrac1q}<+\infty$,
\begin{align*}
 \|\P_\alpha(X_t\in \cdot\mid t<\tau_\d)-\nu_{QSD}\|_{\sss TV}\leq \mathrm{Cst}\,\alpha\cdot V^{\nicefrac1q}\,\mathrm e^{-\delta t},
\end{align*}
which entails Assumption~(A3) and provides a candidate for the long
time behavior of the MVPP $m_n/m_n(E)$.

Finally, using the fact that $\nu_{QSD}(Q-I)=-\lambda_0\nu_{QSD}$ for some
$\lambda_0>0$ (this is a classical property of quasi-stationary
distributions, see for instance~\cite{DoornPollett2013}) and hence that $\nu_{QSD} R$ is
proportional to $\nu_{QSD}$,  Theorem~\ref{thm:unbal-with-weights} entails
that, if $\sum_{k\geq 0} \mathrm e^{ak} m_0(k)$ is finite, then
\begin{align*}
  \frac{m_n}{m_n(E)}\xrightarrow[n\rightarrow+\infty]{a.s.} \frac{\nu_{QSD}R}{\nu_{QSD} R(E)}=\nu_{QSD}.
\end{align*}
with respect to the topology of weak convergence.

\subsection{Random trees}\label{sub:trees}
As discussed in Janson~\cite[Examples~7.5 and~7.6]{Janson04},
infinitely-many-color urns are particularly useful for the study of some functionals of random trees;
we give below two examples where our main result applies, and gives stronger convergence results.

\subsubsection{Outdegree profiles}\label{sub:out-degrees}
\begin{definition}\label{df:profile}
We define the out-degree profile of a rooted tree $\tau$ as
\[\mathrm{Out}(\tau) = \sum_{\nu\in\tau} \delta_{\mathrm{outdeg}(\nu)},\]
where for all node $\nu$ in $\tau$, $\mathrm{outdeg}(\nu)$ is the out-degree of $\nu$ (i.e.\ its number of children).
\end{definition}

\vspace{\baselineskip}
\noindent{\bf Out-degree profile in the random recursive tree.}
The random recursive tree $(\mathrm{RRT}_n)_{n\geq 1}$ is a sequence of random rooted trees defined recursively as follows:
\begin{itemize}
\item $\mathrm{RRT}_1$ has one node (the root);
\item we build $\mathrm{RRT}_{n+1}$ from $\mathrm{RRT}_n$ by choosing a node of $\mathrm{RRT}_n$ uniformly at random, 
and adding a child to this node.
\end{itemize}
It is straightforward to see that the sequence $(\mathrm{Out}(\mathrm{RRT}_n))_{n\geq 1}$ of 
the out-degree profile of the random recursive tree is a MVPP on $\mathbb N$ of initial composition $m_1=\delta_0$, and replacement kernel
\[R_x = -\delta_x + \delta_0 + \delta_{x+1}\quad (\forall x\geq 0).\]
Note that the replacement measures $R_x$ are not positive, but the process satisfies Assumption~(T) by definition 
and thus this MVPP falls into the framework of Section~\ref{sec:remball}.
In this case, $P_x=\delta_x$, and $R^{\sss (i)} = R = Q$ almost surely for all $i\geq 1$.
Note that $Q_x(\mathbb N) = 1$ for all $x\in\mathbb N$, and, therefore, Assumption (A1) holds with $\mu = \delta_1$ and $c_1 = 1$.

Fix $\varepsilon\in(0,\nicefrac12)$ 
and let $V(x)= (2-\varepsilon)^x$ for all $x\geq 0$; Assumption (A'2-i) holds, and we have
\[Q_x\cdot V = -(2-\varepsilon)^x + 1 + (2-\varepsilon)^{x+1} = 1 + (1-\varepsilon)V(x),\]
 for all $q\in(1,2]$,
\[Q_x\cdot V^{\nicefrac1q} 
= -(2-\varepsilon)^{\nicefrac x q} + 1 + (2-\varepsilon)^{{\nicefrac {(x+1)}q}} 
= 1 + ((2-\varepsilon)^{\nicefrac1q}-1)V(x)^{\nicefrac1q}
\leq 1 + (1-\varepsilon)V(x)^{\nicefrac1q},\]
since $\nicefrac1q<1$ and $2-\varepsilon>1$.
Therefore, Assumption (A'2-ii) is satisfied with $\theta=1-\varepsilon$ and $K=1$.
Note that, for all continuous function $f : \mathbb N \to \mathbb R$ bounded by~$1$, we have, for all $q'\in(1,3]$
\[|Q_x\cdot f|^{q'} \leq |1-f(x)+f(x+1)|^{q'}\leq 3^{q'} \leq 27V(x),\]
since $1\leq V(x)$ for all $x\in\mathbb N$.
Therefore, since $Q^{\sss (i)}=R^{\sss (i)} = R = Q$ almost surely for all $i\geq 1$, Assumption (A'2-iii) holds with $A=27$. 
Using again that $V(x)\geq 1$ for all $x\in\mathbb N$, we have
\[|Q_x\cdot V| = 1 + (1-\varepsilon)V(x)\leq (2-\varepsilon)V(x),\]
and, for all $q\in(1,2]$,
\[|Q_x\cdot V^{\nicefrac1q}|^q 
\leq \big(1 + (1-\varepsilon)V(x)^{\nicefrac1q}\big)^q
\leq 2^q V(x),\]
since $2-\varepsilon<2$. 
Therefore, Assumption (A'2-iv) holds and so does (A'2); note that $p$ can be arbitrary in $(2,\infty)$, making $q$ arbitrary in $(1,2)$. Note that $q'$ is restricted to be in $(q,3]$.

One can check that the Markov chain of kernel $(R_x)_{x\in\mathbb N}$
is ergodic, with unique stationary distribution
$\nu_x = 2^{-x-1}$ ($\forall x\geq 0$). { By~\cite{MeynTweedie1993},
  we obtain the uniform convergence to $\nu$ required in Assumption
  (A3).}  Finally, (A4) holds since $E=\mathbb N$ is discrete.

Therefore, Theorem~\ref{thm:unbal-with-weights} applies and gives that
\begin{equation}\label{eq:cv_RRT}
n^{-1}\mathrm{Out}(\mathrm{RRT}_n) \to \nu \quad\text{weakly, almost surely when }n\to\infty.
\end{equation}
since $\nu R = \nu$.
Different versions of this result can be found in the literature: 
Bergeron, Flajolet \& Salvy~\cite[Corollary~4]{BFS92} 
prove it using generating functions,
Mahmoud \& Smythe~\cite{MahmoudSmythe92} prove a joint central limit theorem for the number of nodes of out-degree 0, 1 and 2,
Janson~\cite[Example~7.5]{Janson04} extends this result by considering out-degrees $0, 1, \ldots, M$ for all $M\geq 0$,
which implies~\eqref{eq:cv_RRT}. 
The approach of~\cite{MahmoudSmythe92} and~\cite{Janson04} relies on the remarkable fact that, in that particular example, 
one can reduce the problem to finitely many types.

Our main contribution for this example is to
prove the convergence in a stronger sense, and thus answer a question of Janson 
(see Remark~1.2~\cite{Janson03}).
Indeed, Theorem~\ref{thm:unbal-with-weights} also gives that, for all $q\in (1,2)$,
\[\sup_n \frac{\mathrm{Out}(\mathrm{RRT}_n)}{n} \cdot V^{\nicefrac1q} <+\infty,\]
since $P_x = \delta_x$ for all $x$, in this example.
Therefore, 

\begin{proposition}
For all $\varepsilon\in(0,\nicefrac12)$, for all $q\in(1,2)$,
for all functions $f\,:\, \mathbb N \to \mathbb R$ such that $f(x) = o\big((2-\varepsilon)^{\nicefrac x q}\big)$ 
when $x\to\infty$, we have
\[\frac1n\int f \,\mathrm d\mathrm{Out}(\mathrm{RRT}_n)
\to \sum_{x=0}^{\infty} 2^{-x-1} f(x), \text{ almost surely when }n\to\infty.\]
\end{proposition}

Our approach also has the advantage of 
providing a framework that can be easily generalized, as, for example, in
the next application to which Janson's finitely-many-types approach wouldn't apply.

\vspace{\baselineskip}
\noindent{\bf Out-degree profile in a random recursive forest with multiple children.}
Let us now consider the following generalization of the random
recursive tree studied above.  The random recursive forest
$(\mathrm{RRF}_n)_{n\geq 1}$ with multiple children is defined as a
sequence of random rooted forests defined recursively as follows:
consider a probability measure $\alpha$ on $\{-1\}\cup\{1,2,\ldots\}$
(with $0<\alpha_{-1}<1$) and a probability measure $\beta$ on
$\{1,2,\ldots\}$;
\begin{itemize}
\item $\mathrm{RRF}_1$ has one node (the root);
\item we build $\mathrm{RRF}_{n+1}$ from $\mathrm{RRF}_n$ by choosing
  a node of $\mathrm{RRF}_n$ uniformly at random, and, if this node has at least one child,
  \begin{itemize}
  \item with probability $\alpha_{-1}$, remove the edge between the node
    and one of his children (hence forming an other tree in the forest),
  \item with probability $\alpha_k$ ($k\geq 1$), add $k$ children to this node,
  \end{itemize}
  while, if this node has $0$ child, with probability $\beta_k$ ($k\geq 1$), add $k$ children to this node.
\end{itemize}
We define $\mathrm{Out}(\mathrm{RRF}_n)$ as the sum of the out-degree profiles (see Definition~\ref{df:profile}) of the trees composing the forest $\mathrm{RRF}_n$.

\begin{proposition}
  \label{prop:RRF}
  Assume that $\alpha$ and $\beta$ both admit an exponential moment of
  order $\lambda$, for some fixed $\lambda>0$. There exists a
  probability distribution $\nu_{QSD}$ such that, for all $q\in(1,2)$,
  for all $a>0$ satisfying
  \[
    \sum_{k=1}^{+\infty} \alpha_k e^{ak} < 2 \sum_{k=1}^\infty \alpha_k,
  \]
  and for all function $f\,:\,E=\mathbb N\to \mathbb R$ such that
  $f(x) = o(e^{ax/q})$ when $x\to\infty$, we have
\begin{equation}\label{eq:cv_RRT_mult}
\int f \, \frac{\mathrm d\mathrm{Out}(\mathrm{RRF}_n)}{\mathrm{Out}(\mathrm{RRF}_n)(E)} 
\to \int f \,\mathrm d\nu_{QSD}, \text{ almost surely when }n\to\infty.
\end{equation}
\end{proposition}

\begin{proof}
It is straightforward to see that the sequence
$(\mathrm{Out}(\mathrm{RRF}_n))_{n\geq 1}$ of the out-degree profile
of the random recursive forest is a MVPP on $\mathbb N$ of initial
composition $m_0=\delta_0$, and random replacement kernel given, for
all $x\geq 1$ by
\[
  R^{(i)}_x =
  \begin{cases}
    -\delta_x+\delta_{x-1}&\text{ with probability }\alpha_{-1}\\
    -\delta_x+k\,\delta_0+\delta_{x+k}&\text{ with probability }\alpha_k, \text{ for all }k\geq 1,\\
  \end{cases}
\]
and
\[R^{(i)}_0 = (k-1)\delta_0+\delta_k\text{ with probability }\beta_k, \text{ for all } k\geq 1.
\]
In particular, for all $x\geq 1$,
\[R_x=-\delta_x+\sum_{k=1}^\infty k \alpha_k\delta_0+\alpha_{-1}\delta_{x-1}+\sum_{k=1}^\infty
\alpha_k\delta_{x+k},\]
and
\[R_0=\sum_{k=1}^\infty (k-1)\beta_k \,\delta_0+\sum_{k=1}^\infty \beta_k\delta_k.
\]
We deduce that, for all $x\geq 1$,
$R_x(E)=M_\alpha:=\sum_{k\in\N\cup\{-1\}} |k| \alpha_k$ (the first absolute
moment of $\alpha$) and $R_0(E)=M_\beta:=\sum_{k\in\N} k\,\beta_k$ (the mean of
$\beta$). From now on, we consider the MVPP $m_n^M$ with replacement
kernel $\bar R^{\sss (i)}_x:=\frac{1}{M}R^{\sss (i)}_x$, 
where $M=M_\alpha\vee M_\beta$. 
Although the replacement measures $\bar R^{\sss (i)}$ are not
positive, the process satisfies Assumption~(T) by definition and thus
this MVPP falls into the framework of Section~\ref{sec:remball}, 
with weight kernel $\bar P_x=\delta_x$ and
$\bar Q^{\sss (i)}_x = \frac{1}{M}R^{\sss (i)}_x$ for all $x\geq 0$.

For any fixed $p>2$ and $q=\frac{p}{p-1}\in(1,2)$,
we have, for all $i\geq 1$, for all $x\geq 1$, 
\[\bar Q_x^{\sss (i)}(E)
= \begin{cases}
\frac k M & \text{ with probability }\alpha_k, \text{ for all }k\geq 1\\
0 &\text{ with probability }\alpha_{-1},
\end{cases}\]
and $\bar Q_0^{\sss (i)}(E) = \nicefrac{k}{M}$ with probability $\beta_k$ for all $k\geq 1$.
Thus, if we set
  \[\mu=\alpha_{-1}\,\delta_0+\left(\sum_{k=1}^\infty \alpha_k\right)\,\delta_{\nicefrac 1 M},\]
we get that $c_1 = \int x\,\mathrm d \mu(x)=\nicefrac{\sum_{k\geq 1} \alpha_k}{M} > 0$, and thus Assumption (A1) holds.

Let us
now check that (A'2) holds with $V(x)=\mathrm e^{ax}$, 
where $a\in(0,\lambda)$ satisfies
  \[
    \sum_{k=1}^{+\infty} \alpha_k e^{ak} < 2 \sum_{k=1}^\infty \alpha_k.
  \]
Assumption (A'2-i) is straightforward. 
Moreover, we
have, for all $x\geq 1$,
\begin{align*}
  \bar Q_x\cdot V
  &=\frac{-1}{M} V(x)+\frac{\sum_{k=1}^\infty k\alpha_k}{M}V(0)+\frac{1}{M}\alpha_{-1}V(x-1) +\frac{1}{M}\sum_{k=1}^{+\infty} \alpha_k V(x+k)\\
  &\leq \frac{1}{M}\left(-1+\alpha_{-1}\mathrm e^{-a}+\sum_{k=1}^{+\infty} \alpha_k \mathrm e^{ak}\right)V(x)+1\\
   &\leq \frac{\sum_{k\geq 1} \alpha_k e^{ak}-\sum_{k\geq 1} \alpha_k}{\sum_{k\geq 1}\alpha_k}c_1 V(x)+1,
\end{align*}
where
  $\frac{\sum_{k\geq 1} \alpha_k e^{ak}-\sum_{k\geq 1}
    \alpha_k}{\sum_{k\geq 1}\alpha_k}<1$ by assumption.  Similarly,
\begin{align*}
  \bar Q_x\cdot V^{\nicefrac1q}
  &\leq \frac{1}{M}\left(-1+\alpha_{-1}\mathrm e^{-\nicefrac a q}
  +\sum_{k=1}^{+\infty} \alpha_k \mathrm e^{\nicefrac{ak}q}\right)V^{\nicefrac1q}(x)+1\\
  &\leq \frac{\sum_{k\geq 1} \alpha_k e^{ak}-\sum_{k\geq 1} \alpha_k}{\sum_{k\geq 1}\alpha_k}c_1 V^{\nicefrac1q}(x)+1,
\end{align*}
so that (A'2-ii) is satisfied with $\theta=\frac{\sum_{k\geq 1} \alpha_k e^{ak}-\sum_{k\geq 1} \alpha_k}{\sum_{k\geq 1}\alpha_k}c_1\in(0,c_1)$
and $K=1\vee \bar Q_0\cdot V$. 
For all $x\geq 1$, for all $q'>1$, and
for all function $f : \mathbb N\to\mathbb R$ 
continuous and bounded by~$1$, we have
\begin{equation}\label{eq:RRF1}
|\bar Q_x\cdot f|^{q'} = \frac1{M^{q'}}
\left|-f(x)+\alpha_{-1}f(x-1) + \Big(\sum_{k\geq 1}k\alpha_k\Big)f(0) + \sum_{k\geq 1}\alpha_k f(x+k)\right|^{q'}
\leq
A_1^{q'}V(x),
\end{equation}
where $A_1 = 1\vee ({3+M_{\alpha}}/{M})$, since $V(x)\geq 1$ for all $x\geq 0$;
we also have
\begin{equation}\label{eq:RRF2}
|\bar Q_0\cdot f|^{q'} = \frac1{M^{q'}}\left|\Big(\sum_{k\geq 1}(k-1)\beta_k\Big)f(0) + \sum_{k\geq 1}\beta_k f(k)\right|\leq \left(\frac{M_\beta}{M}\right)^{q'} \leq 1 \leq A_1V(0),
\end{equation}
since $V(0)\geq 1$ and $A_1\geq 1$ by definition.
We also have that, for all $r>1$,
\begin{align*}
\E\big[\big|\bar R^{\sss (i)}_{x} \cdot f - \bar R_{x}\cdot f\big|^r\big]
&\leq \mathbb P(|\bar R^{\sss (i)}_{x} \cdot f - \bar R_{x}\cdot f |\leq 1)
+ 2^{r-1}\E\big[|\bar R^{\sss (i)}_{x} \cdot f|^r +|\bar R_{x} \cdot f|^r\big]\\
&\leq 1+ 2^{r-1}\E\big[|\bar R^{\sss (i)}_{x} \cdot f|^r\big]+2^{r-1}A_1^r V(x),
\end{align*}
because of Equations~\eqref{eq:RRF1} and~\eqref{eq:RRF2} applied to the special case~$q'=r$. Note that
\begin{align*}
\E\big[|\bar R^{\sss (i)}_{x} \cdot f|^r\big]
&=\frac{\alpha_{-1}}{M}\big|-f(x)+f(x-1)\big|^r
+\sum_{k\geq 1}\frac{\alpha_k}{M}\big|-f(x)+kf(0)+f(x+k)\big|^r\\
&\leq \frac{2^r\alpha_{-1} + \sum_{k\geq 1}(2+k)^r\alpha_k}{M}
=: A_{2,r} <+\infty,
\end{align*}
since $\alpha$ admits an exponential moment, and therefore has finite polynomial moments.
Therefore, using again that~$V$ is bounded from below by 1, we get that
\[\E\big[\big|\bar R^{\sss (i)}_{x} \cdot f - \bar R_{x}\cdot f\big|^r\big]
\leq (1+2^{r-1}A_1+2^{r-1}A_{2,r})V(x),
\]
for all $x\geq 1$. A similar reasoning, using that $\beta$ also has exponential moments, implies that
\[\E\big[\big|\bar R^{\sss (i)}_{0} \cdot f - \bar R_{0}\cdot f\big|^r\big]
\leq (1+2^{r-1}A_1+2^{r-1}A_{3,r})V(0),\]
where $A_3 = \sum_{k\geq 1}\beta_k k^r$. Since $\bar R^{\sss (i)}=\bar Q^{\sss (i)}$ almost surely, we obtain
 \[\E\big[\big|\bar Q^{\sss (i)}_{0} \cdot f - \bar Q_{0}\cdot f\big|^p\big]
\leq (1+2^{p-1}A_1+2^{p-1}A_{3,p})V(0).\]
Therefore, setting $A = 1+2^{p-1}A_1+2^{p-1}(A_{2,r} \vee A_{3,r} \vee A_{3,p})$, we can conclude that
Assumption (A'2-iii) holds.

Finally, let us check Assumption (A'2-iv): 
for all $x\geq 1$, for all $\ell>1$ and $s\leq 2$, we have
\begin{align}
  \big|\bar Q_x\cdot V^{\nicefrac1\ell}\big|^s
  &=\frac1{M^s}
  \left|-V(x)^{\nicefrac1\ell}+V(0)^{\nicefrac1\ell}
  +\alpha_{-1} V(x-1)^{\nicefrac1\ell}
  +\sum_{k=1}^{+\infty}\alpha_k V(x+k)^{\nicefrac1\ell}\right|^s\notag\\
  &\leq \frac1{M^s}
  \left(2+\alpha_{-1}+\sum_{k=1}^{+\infty}\alpha_k \mathrm e^{\lambda k/\ell}\right)^s V(x)^{\nicefrac{s}\ell}\notag\\
  &\leq \left(\frac{3+\sum_{k\geq 1} \alpha_k \mathrm e^{\lambda k}}{M}\right)^s V(x)^{\nicefrac{s}\ell} ,\label{eq:RRF3}
\end{align}
and, for all $r\in(1,q)$,
\begin{align}
  \E\left[\left|\bar Q^{\sss (i)}_x\cdot V^{\nicefrac1q}\right|^r\right]
  &=\frac1{M^r}\left(\alpha_{-1} \left|-\mathrm e^{ax/q}+1+\mathrm e^{a(x-1)/q}\right|^r
  +\sum_{k=1}^{+\infty} \alpha_k\left|-\mathrm e^{ax/q}+k+\mathrm e^{a(x+k)/q}\right|^r\right)\notag\\
   &\leq \frac{V(x)^{\nicefrac1q}}{M^r} \left(\alpha_{-1} 3^r +
   3^{r-1} \sum_{k\geq 1} 
   \alpha_k \big(1+k^q+\mathrm e^{ak}\big)\right)\notag\\
   &\leq B_1 V(x)^{\nicefrac1q} ,
   \label{eq:RRF4}
\end{align}
where
$B_1 = 3^2 \big(\alpha_{-1} + \sum_{k\geq 1} \alpha_k (1+k^2+\mathrm
e^{\lambda k})/M<+\infty$. Similar calculations hold for $x=0$; we
thus now reason as if Equations~\eqref{eq:RRF3} and~\eqref{eq:RRF4}
also hold for $x=0$.  Applying Equation~\eqref{eq:RRF3} to
$\ell = s = 1$ gives that $|Q_x\cdot V|\leq B_2 V(x)$ for all
$x\geq 0$, where $B_2 = (3+\sum_{k\geq 1} \alpha_k \mathrm e^{\lambda k})/M$. Applying
Equation~\eqref{eq:RRF3} to $\ell = s = q$ gives that
$|Q_x\cdot V^{\nicefrac1q}|^q\leq B_3 V(x)^{\nicefrac1q}$ for all
$x\geq 0$, where $B_3 = ((2+\sum_{k\geq 1} \alpha_k \mathrm e^{\lambda k})/M)^q$. Finally, applying
Equation~\eqref{eq:RRF3} to $\ell=q$ and $s=r$, and using
Equation~\eqref{eq:RRF4}, we get that
\begin{align*}
\E\left[\left|\bar Q^{\sss (i)}_x\cdot V^{\nicefrac1q}-\bar Q_x\cdot V^{\nicefrac1q}\right|^r\right]
&\leq 2^{r-1} \left(\E\big[\big|\bar Q^{\sss (i)}_x\cdot V^{\nicefrac1q}\big|^r\big]
+\E\big[\big|\bar Q_x\cdot V^{\nicefrac1q}\big|^r\big]\right)\\
&\leq 2^{r-1} \big(B_1V(x)^{\nicefrac{r}{q}} +({2+\sum_{k\geq 1} \alpha_k \mathrm e^{\lambda k}}/{M})^r V(x)^{\nicefrac{r}{q}}\big)\\
&\leq B_4 V(x),
\end{align*}
with
$B_4 = 2^{r-1} \big(B_1 +({2+\sum_{k\geq 1} \alpha_k \mathrm
  e^{\lambda k}}/{M})^r\big)$, because $\nicefrac{r}{q}<1$, and
$V(x)\geq 1$ for all $x\geq 0$.  Therefore, taking
$B = B_2\vee B_3\vee B_4$, we conclude that Assumption (A'2-iv) holds.

\medskip
The continuous-time pure jump Markov process~$X$ with
sub-Markovian jump matrix $Q-I$ is irreducible and clearly satisfies the
assumptions of Theorem~5.1 and Remark~11 in~\cite{ChampagnatVillemonais2017}.
Therefore, there exist a quasi-stationary distribution
$\nu_{QSD}$ for $X$ { and two positive constants $\textrm{Cst},\delta>0$ such that, for all probability measure $\alpha\in E$ satisfying $\alpha\cdot V^{\nicefrac1q}<+\infty$, for all $t\geq 0$,
\begin{align*}
  \|\P_\alpha(X_t\in \cdot\mid t<\tau_\d) - \nu_{QSD}\|_{\sss TV}\leq \textrm{Cst}\,\alpha\cdot V^{\nicefrac1q}\,\mathrm e^{-\delta t},
\end{align*}}
which entails Assumption~(A3). Since Assumption~(A4) is clearly satisfied,
Theorem~\ref{thm:unbal-with-weights} applies and hence 
\begin{equation}
\frac{\mathrm{Out}(\mathrm{RRF}_n)}{\mathrm{Out}(\mathrm{RRF}_n)(E)} \to \frac{\nu_{QSD} R}{\nu_{QSD} R(E)} \quad\text{weakly, almost surely when }n\to\infty.
\end{equation}
Since $\nu_{QSD}R$ is proportional to $\nu_{QSD}$, 
and since we also have, again by Theorem~\ref{thm:unbal-with-weights}, 
\[\sup_n \frac{\mathrm{Out}(\mathrm{RRF}_n)}{\mathrm{Out}(\mathrm{RRF}_n)(E)}\cdot V^{\nicefrac1q} <+\infty,\]
this concludes the proof of Proposition~\ref{prop:RRF}.
\end{proof}

\subsubsection{Protected nodes}

A node $\nu$ of a tree $\tau$ is 2-protected if the closest leaf is at distance at least $2$ from $\nu$;
in a social network, 2-protected nodes can be users who used to invite new users to the network but have not done so recently.
The proportion of such nodes in different models of random trees have been studied in the literature: 
Motzkin trees in Cheon \& Shapiro~\cite{CS08},
random binary search tree in B\'ona~\cite{Bona14},
and more recently in the $m$-ary search tree in Holmgren, Janson \& \v{S}ileikis~\cite{HJS16}.
Devroye \& Janson~\cite{DJ14} show how results of Aldous~\cite{Aldous91} about fringe trees can be used to study this question 
with a unified approach for different models of random trees, including simply generating trees and the random recursive tree.
We show here how our main result allows to get information about protected nodes in random trees.

\vspace{\baselineskip}
\noindent{\bf Protected nodes in the random recursive tree.}
For all $n\geq 1$ and $x\geq 0$, 
let us denote by $X_{n,x}$ the number of internal nodes in $\mathrm{RRT}_n$ having exactly $x$ leaf-children.
The random measure 
\[m_n = \sum_{x\in\mathbb N} X_{n,x}\delta_x\]
is a MVPP of initial composition $m_0 = \delta_1$. 
The replacement kernel of $(m_n)_{n\geq 0}$ is (for all $i\geq 1$ and $x\geq 1$)
\[
R_0^{\sss (i)} = -\delta_0 + \delta_1\quad\text{ and }\quad 
R_x^{\sss (i)} = B^{\sss (i)}_{\nicefrac1{x+1}}\delta_{x+1} + \big(1-B^{\sss (i)}_{\nicefrac1{x+1}}\big)(\delta_{x-1}+\delta_1)-\delta_x,
\]
where $\big(B^{\sss (i)}_{\nicefrac1{x+1}}\big)$ is a sequence of i.i.d.\ random 
Bernoulli-distributed variables of parameters $\nicefrac1{x+1}$ for all $x\geq 1$.
The weight kernel of $(m_n)_{n\geq 0}$ is $P_x= (x+1)\delta_x$ (for all $x\in\mathbb N$).
We therefore have
\[
R_0 = -\delta_0 + \delta_1\quad\text{ and }\quad 
R_x = \frac{1}{x+1}\,\delta_{x+1} + \frac{x}{x+1}(\delta_{x-1}+\delta_1)-\delta_x,
\]
and
\[
Q_x = \frac{x+2}{x+1}\,\delta_{x+1} + \frac{x}{x+1}(x\delta_{x-1}+2\delta_1)-(x+1)\delta_x,
\]
for all $x\geq 0$.
Note that $Q_x(\mathbb N) = 1$ for all $x\geq 0$.
Let us check the assumptions of Theorem~\ref{thm:unbal-with-weights}; (T) is satisfied by construction of the model, (A1) is satisfied with $\mu = \delta_1$ and thus $c_1 = 1$. 
Fix $\varepsilon>0$, $V(0)=V(1)=1$, and $V(x) = \prod_{i=2}^x (i-\varepsilon)$ for all $x\geq 2$; (A'2-i) is clearly satisfied, and for all $x\in \mathbb N$,
\begin{align*}
Q_x\cdot V
&= \frac{x+2}{x+1}V(x+1) + \frac{x}{x+1}\big(xV(x-1)+2V(1)\big)-(x+1)V(x)\\
&=V(x)\bigg(\frac{x+2}{x+1}(x+1-\varepsilon)+\frac{x^2}{(x+1)(x-\varepsilon)} + \frac{2x}{x+1}\frac{1}{V(x)} - x-1 \bigg).
\end{align*}
Note that, when $x\to\infty$,
\[\frac{x+2}{x+1}(x+1-\varepsilon)+\frac{x^2}{(x+1)(x-\varepsilon)} + \frac{2x}{x+1}\frac{1}{V(x)} - x-1
= 1-\varepsilon + o(1),\]
implying that there exists $x_0$ such that, for all $x\geq x_0$,
$Q_x\cdot V\leq 1-\nicefrac{\varepsilon}{2}$, and thus, for all $x\geq 0$,
\[Q_x\cdot V \leq (1-\nicefrac\varepsilon2)V(x) + \sup_{z\leq x_0} Q_z\cdot V.\]
The same reasoning gives that, for all $p> 2$, $q=p/(p-1)\in(1,2)$,
\begin{align*}
Q_x\cdot V^{\nicefrac1q}
&=V(x)^{\nicefrac1q}\bigg(\frac{x+2}{x+1}(x+1-\varepsilon)^{\nicefrac1q}+\frac{x^2}{(x+1)(x-\varepsilon)^{\nicefrac1q}} + \frac{2x}{x+1}\frac{1}{V(x)^{\nicefrac1q}} - x-1 \bigg)\\
&=V(x)^{\nicefrac1q}(x^{\nicefrac1q}+x^{1-\nicefrac1q}-x + \mathcal O(1))
=-V(x)^{\nicefrac1q}(x+o(x)),
\end{align*}
and there exists $x_1$ such that for all $z\geq x_1$, $Q_x\cdot V^{\nicefrac1q}\leq 0$.
Thus, (A'2-ii) is satisfied with $\theta = 1-\nicefrac{\varepsilon}{2}$ and 
$K = \sup_{z\leq x_0} Q_z\cdot V + \sup_{z\leq x_1} Q_z\cdot V^{\nicefrac1q}$.
Let $f$ be a function from $\{0, 1, \ldots\}$ to $\mathbb R$ continuous and bounded by~$1$, 
and $r\in(1,2)$; we have
\begin{align*}
|Q_x\cdot f|^r
&= \Big|\frac{x+2}{x+1}f(x+1)+\frac{x^2}{x+1}f(x-1) + \frac{2x}{x+1}f(1) - (x+1)f(x)\Big|^r\\
&\leq 4^{r-1}\left(\Big(\frac{x+2}{x+1}\Big)^r+\Big(\frac{x^2}{x+1}\Big)^r+\Big(\frac{2x}{x+1}\Big)^r + (x+1)^r\right).
\end{align*}
When $x\to\infty$, we have
\[\Big(\frac{x+2}{x+1}\Big)^r+\Big(\frac{x^2}{x+1}\Big)^r+\Big(\frac{2x}{x+1}\Big)^r + (x+1)^r
  = (2+o(1)) x^r.\] Note that, when $x\to\infty$,
$x^r = o(x^2) = o\big(V(x)\big)$, which implies that there exists a
constant $A$ such that, for all $x\geq 0$,
$|Q_x\cdot f|^r\leq AV(x)$. One can check that,
$R^{\sss (i)}_0 = R_0$, and, for all $i\geq 1$,
\[|R_x^{\sss (i)}\cdot f-R_x\cdot f|^r\leq 3,\]
because a Bernoulli random variable is at most at distance~$1$ from its mean, almost surely. We also have
\begin{align*}
&\mathbb E|Q_x^{\sss (i)}\cdot f-Q_x\cdot f|^p\\
&=\left|(x+2)f(x+1)\Big(B^{\sss (i)}_{\nicefrac{1}{x+1}}-\frac1{x+1}\Big) 
+ x\big(xf(x-1)+2f(1)\big)\Big(\frac{1}{x+1}-B^{\sss (i)}_{\nicefrac{1}{x+1}}\Big)\right|^p\\
&=\Big((x+2)^r + x^r(x+2)^r\Big)\leq AV(x),
\end{align*}
for $A$ large enough, since $x^{2r} = o\big(V(x)\big)$ when $x\to\infty$.
We have thus checked that (A'2-iii) holds. 
Assumption (A'2-iv) can be checked in the same way; we leave the details to the reader. Note that $p> 2$, and thus $q\in(1,2)$ are arbitrary.

Set
\[\nu_0 = \frac{\mathrm e-2}{1+2\mathrm e}, \quad \nu_1 = \frac{4(\e-2)}{1+2\e},
  \quad\text{ and }\quad \nu_i = \frac{2(i+1)}{1+2\e}\sum_{j\geq
    i+1}\frac1{j!}, \quad(\forall i\geq 2).\] One can check that the
Markov process with jump measure $Q-I$ is ergodic, that
$\nu = (\nu_i)_{i\geq 0}$ is the unique stationary distribution of
$Q-I$. { Using (A2) and~\cite{MeynTweedie1993}, we get that (A3) is satisfied.}
Therefore, our main result
applies ((A4) is immediate since $E = \mathbb N$ is discrete) and we
get that $\tilde m_n$ converges almost surely to
$\pi := \nu R/\nu R(\mathbb N)$.  Let us denote by $\hat \pi = \nu R$;
it is straightforward to check that
\[\hat\pi_0 = \frac{\e - 2}{1+2\e}, \quad \hat\pi_1 = \frac{2\e - 4}{1+2\e}, \quad\text{ and }\quad
\hat \pi_x = \frac2{1+2\e}\sum_{i\geq x+1}\frac1{i!},\]
and thus that $\nu R(\mathbb N) = \e/(1+2\e)$, implying that
\[\pi_0 = 1-\frac2{\e}, \quad \pi_1 = 2-\frac4{\e}, \quad\text{ and }\quad
\pi_x= \frac2{\e}\sum_{i\geq x+1} \frac1{i!}.\]
We have thus proved the following:
\begin{proposition}
For all $x\geq 1$, 
the proportion $p_{n,x}$ of internal nodes having exactly~$x$ 
leaf-children in the $n$-node random recursive tree
converges almost surely to
\[\frac2{\e}\sum_{i\geq x+1} \frac1{i!}.\]
The proportion $p_{n,0}$ of protected internal nodes converges almost
surely to $1-\nicefrac2\e$.  Moreover, for all $q\in(1,2)$ and all function
$f:\{0,1, \ldots\}\to\mathbb R$ such that
$f(x) = o\big(\prod_{i=2}^x (i-\varepsilon)^{\nicefrac1q}\big)$ for
some $\varepsilon>0$ when $x\to\infty$, we have
\[\sum_{i\geq 0} p_{n,i}f(i) \to (1-\nicefrac2\e) f(0)  + \frac2{\e}\sum_{i\geq 1} f(i)\sum_{j\geq i+1} \frac1{j!}\]
almost surely when $n\to\infty$.
\end{proposition}
Note that, in the proposition above, the proportions are calculated among internal nodes only.
To translate this result in terms of proportion among all nodes, we need one last calculation to take into account the leaf-nodes.
Note that the limit proportion of leaves in the random recursive tree is given by
\[\frac{\sum_{i\geq 0} i\pi_i}{1+\sum_{i\geq 0} i\pi_i} = \nicefrac12,\]
because $\sum_{i\geq 0} i\pi_i=1$ (this result is folklore and was already discussed in Section~\ref{sub:out-degrees}).
Therefore, the proportion of nodes having exactly $i$ leaf-children in the $n$-node random recursive tree converges almost surely to
$\nicefrac{\pi_i}{2}$:
We get that,
for all $i\geq 1$, the proportion of nodes having exactly $i$ leaf-children in the $n$-node random recursive tree
converges almost surely to
\[\frac1{\e}\sum_{j\geq i+1} \frac1{j!}.\]
The proportion of protected internal nodes converges almost surely to $\nicefrac12-\nicefrac1\e$.
Note that the convergence in probability of the proportion of protected nodes in the random recursive tree was already proved by Ward \& Mahmoud~\cite{MW15}; we have shown how our main result implies almost-sure convergence.

\subsection{``Sample paths'' P\'olya urns}\label{sub:samplepaths}

In this section we consider the case where the replacement measures
are the empirical occupation measures of sample paths of Markov processes. 
The section is divided into three subsections: the first one is devoted to the discrete-time setting, the second to the continuous-time setting, 
the third one to an application to stochastic-approximation
algorithms for the computation of quasi-stationary distributions.

\subsubsection{Discrete-time sample paths P\'olya urns}

Let $(X_n)_{n\in\{0,1,2,\ldots\}}$ be a Markov chain evolving in a
Polish locally-compact state space $E\cup\{\d\}$, where $\d\notin E$
is an absorbing point : $X_n=\d$ for all
$n\geq \tau_\d:=\min\{k\geq 0,\ X_k\in\d\}$ almost surely. We denote by
$\P_x$ and $\E_x$ the law of the process $X$ starting from
$x\in E\cup\d$ and its associated expectation. 
Also fix $\cal T$ a probability distribution on $\N\cup\{+\infty\}$ such that ${\cal T}(\{0\})<1$ and such that, if $(T,X)$ is distributed according to ${\cal T}\otimes \P_x$,
then $\tau_\d \wedge T$ admits an exponential moment
uniformly bounded with respect to $x\in E$; 
in other words, there
exists $\lambda>0$ such that
\[\sup_{x\in E} \E_x\left[\exp\big(\lambda (T\wedge\tau_\d)\big)\right]<\infty,
\]
(with a slight abuse of notation, 
since we also denote by $\E_x$ the expectation under ${\cal T}\otimes \P_x$).

We consider the MVPP on $E$ with random replacement measures $(R_x^{\sss (i)})_{x\in E, i\geq 1}$ being i.i.d.\ copies of
\[
  R^{\sss (1)}_x = \sum_{n=0}^{T\wedge(\tau_\d-1)} \delta_{X_n},
\]
for all $x\in E$ and all $i\geq 0$, 
where $(T,X)$ is a random variable of distribution ${\cal T}\otimes \P_x$. 
This means that, at each time, we
add to the urn the empirical measure of a sample path of length
$T\wedge (\tau_\d-1)$ of $X$. 
For simplicity, we consider the case without
weights, i.e. $P_x=\delta_x$ for all $x\in E$, so that
$Q^{\sss (i)}=R^{\sss (i)}$. 
Note that the mass of $R^{\sss (i)}_x$ is random, equal
in law to $(T+1)\wedge\tau_\d$ under ${\cal T}\otimes\P_x$, and is not
uniformly bounded in general (although its expectation is, by
assumption, uniformly bounded with respect to $x$). In particular, the
considered MVPP is unbalanced.

To ensure the convergence of this MVPP, we assume that the
following particular instance of the assumptions of Theorem~2.1
in~\cite{ChampagnatVillemonais2017} is satisfied. This abstract
criterion ensures the existence of a quasi-stationary distribution for
$X$; we will show later many examples that fall into this framework.

\medskip\noindent\textbf{Assumption (E).} There exist a positive
integer $n_1$, positive real constants $\alpha_0$, $\alpha_1$,
$\alpha_2$, $\alpha_3$, a locally bounded function with compact level
sets $V:E\rightarrow [1,+\infty)$ and a probability measure $\pi$ on a
compact subset $K\subset E$ such that
\begin{itemize}
\item[(E1)] \textit{(Local Dobrushin coefficient).} For all $x\in K$, 
  \begin{align*}
    \P_x(X_{n_1}\in~\cdot~)\geq \alpha_0 \pi(\,\cdot\,\cap K).
  \end{align*}
\item[(E2)] \textit{(Global Lyapunov criterion).} We have $\alpha_1<\alpha_2$ and, for all $x\in E$,
  \[
    \E_x V(X_1)\leq \alpha_1 V(x)+\alpha_3\1_K(x)\\
    \quad\text{ and }\quad
    \P_x(1<\tau_\d)\geq \alpha_2.
  \]
\item[(E3)] \textit{(Local Harnack inequality).} We have
  \begin{align*}
    \sup_{n\in \Z_+}\frac{\sup_{y\in K} \P_y(n<\tau_\d)}{\inf_{y\in K} \P_y(n<\tau_\d)}\leq \alpha_3
  \end{align*}
\item[(E4)] \textit{(Aperiodicity/irreducibility).} For all $x\in E$, there exists $n_4(x)$ such that, for all $n\geq n_4(x)$,
  \begin{align*}
    \P_x(X_n\in K)>0.
  \end{align*}
\end{itemize}

Under Assumption~(E), it is proved in~\cite{ChampagnatVillemonais2017}
that $X$ admits one and only one quasi-stationary distribution
$\nu_{QSD}$ such that $\nu_{QSD}\cdot V<+\infty$ and which corresponds to
the so-called minimal quasi-stationary distribution (or Yaglom limit).
{ It is also proved in~\cite{ChampagnatVillemonais2017} that there exist two positive constants  $C>0,\delta>0$ such that, for all $t\geq 0$,
\[\|\mathbb P_{\alpha}(X_t\in\cdot|X_t\notin \varnothing)-\nu_{QSD}\|_{\sss TV}
\leq C\,\alpha\cdot V\,\mathrm e^{-\delta t}.\]
}
\begin{proposition}
  \label{prop:sample-paths}
  Under Assumption~(E), if $x\mapsto \E_x f(X_1)$ is continuous on
  $E$ for all continuous bounded function $f:E\rightarrow\R$ and if
  $m_0\cdot V<\infty$, then the normalized sequence of probability
  measures $(\tilde{m}_n)_{n\in\N}$ associated to the MVPP with 
  random replacement kernel $(R^{\sss (i)})_{i\geq 1}$ 
  converges almost surely to the
  quasi-stationary distribution $\nu_{QSD}$ of $X$ in $\cP(E)$.
\end{proposition}

Before turning to the proof of Proposition~\ref{prop:sample-paths}, we
provide typical examples that satisfy Assumption (E) 
and consequently fall into the framework of Proposition~\ref{prop:sample-paths}.

\medskip
\textbf{Example 1.}
  If $E$ is finite and $X$ is irreducible in $E$ (\textit{i.e.}
  $\exists n\geq 1$ s.t. $\P_x(X_n=y)>0$ for all $x,y\in E$) and
  $\P_x(\tau_\d<+\infty)=1$ for all $x\in E$, then Assumption~(E) is
  satisfied for any probability distribution~$\cal T$ 
  (one simply chooses $K=E$ and $V=1$). 

  \medskip

\textbf{Example 2.}
 Consider the case $E=\mathbb N$ and $X$ is a discrete-time birth-and-death process with transition probabilities given by
  \begin{align*}
    \P_x(X_1=y)=
    \begin{cases}
      b_x&\text{if }y=x+1\\
      d_x&\text{if }y=x-1\\
      \kappa_x&\text{if }y=\d,
    \end{cases}
  \end{align*}
  where $(b_x)_{x\in\N},(d_x)_{x\in\N},(\kappa_x)_{x\in\N}$ are families
  of non-negative numbers such that $b_x+d_x+\kappa_x=1$ for all
  $x\in\N$, $d_0 = 0$ and $\inf_{x\geq 1} d_x>0$ for all $x\geq 1$. 
  If
  \[
    b_x\to 0 ~\text{ when }x\to+\infty,
  \]
  then Assumption~(E) is satisfied for any probability distribution $\cal T$ such
  that there exists $\lambda>0$ satisfying $\E\mathrm e^{\lambda T}<+\infty$ ( where the random variable $T$ has distribution $\cal T$). To see this, one simply chooses $K$ large
  enough and $V(x)=\mathrm e^{a x}$ with $a>0$ large enough.

  \medskip \textbf{Example 3.}  Assume that $(X_n)_{n\geq 0}$ is a
  $d$-type Galton-Watson process. We recall that such a process $X$
  evolves in $\mathbb N^d=E\cup\{\d\}$ and
  is absorbed at $\d=(0,\ldots, 0)$. Also, for all
  $n\geq 0$ and $i\in\{1,\ldots,d\}$, we have
  \begin{align*}
    X_{n+1}^i=\sum_{k=1}^{d}\sum_{\ell=1}^{X^k_n} \zeta^{\sss (n,\ell)}_{k,i},
  \end{align*}
  where
  $\big(\zeta^{\sss (n,\ell)}_{k,1},\ldots,\zeta^{\sss (n,\ell)}_{k,d}\big)_{n,\ell,k}$
  is a family of independent random variables in $\mathbb N^d$ such that, for
  all $k\in\{1,\ldots,d\}$,
  $\big(\zeta^{\sss (n,\ell)}_{k,1},\ldots,\zeta^{\sss (n,\ell)}_{k,d}\big)_{n,\ell}$ is
  an independent and identically distributed family. We assume that 
 the matrix of mean offspring denoted by $M=(M_{k,i})_{1\leq k,i\leq d}$ and defined by
  \begin{align*}
    M_{k,i}=\E \zeta^{\sss (n,\ell)}_{k,i},\quad\forall k,i\in\{1,\ldots,d\},
  \end{align*}
  is finite and that there exists $n\geq 1$ such that $M^n_{k,i}>0$
  for all $k,i\in\{1,\ldots,d\}$.
  Let $v$ be a positive right
  eigenvector of the matrix $M$ and denote by $\rho(M)$ its spectral
  radius.

  We assume that $X$ is subcritical (i.e.\ $\rho(M)<1$),
  aperiodic, and irreducible. Then, if there exists
  $\alpha>0$ such that
  $ \E[\exp(\alpha\,|X_1|)\mid X_0=(1,\ldots,1)]<\infty$, then $X$
  satisfies Assumption~(E). To check this, one simply
  observes that $\inf_{x\in E} \P_x(1<\tau_\d)>0$ and carefully checks
  that there exists $\varepsilon>0$ small enough and $K$ large enough so
  that Assumption~(E) is satisfied with $V:x\in E\mapsto\mathrm e^{\varepsilon\langle v,x\rangle}$.

  \medskip

  \textbf{Example 4.} Assume that $X$ evolves in $E=\R^d$ according to
  the following perturbed dynamical systems
  \begin{align*}
    X_{n+1}=f(X_n)+\xi_n,
  \end{align*}
  where $f:\R^d\rightarrow \R^d$ is a measurable function such that
  $|x|-|f(x)|\rightarrow+\infty$ when $|x|\rightarrow+\infty$,
  $(\xi_n)_{n\in\N}$ is an i.i.d.\ sequence of Gaussian random
  variables with positive density in $\R^d$. We assume that the
  process evolves in a measurable set $E$ of $\RR^d$: it is
  immediately sent to $\d\not\in\RR^d$ as soon as $X_n\not\in E$. If
  $E$ is such that
  \[
    \inf_{x\in E} \P\big(f(x)+\xi_1\in E\big)>0,
  \]
  then Assumption~(E) is satisfied. This result is obtained by
  observing that $\inf_{x\in E}\P_x(1<\tau_\d)>0$, by choosing $K$ a
  large enough ball and $V(x)=\mathrm e^{|x|}$
  (see~\cite[Example~9]{ChampagnatVillemonais2017} for more details).

  \begin{proof}[Proof of Proposition~\ref{prop:sample-paths}]
    For all $n\geq 0$, let $\hat m_n = m_n/\sup_{x\in E}
    R_x(E)$. First note that $\hat m_n$ is well defined since
    $\sup_{x\in E} R_x(E)\leq \sup_{x\in E}
    \E_x[T\wedge\tau_\d]<+\infty$, by assumption on the existence of a
    uniform exponential moment for $T\wedge \tau_\d$. Moreover,
    $(\hat m_n)_{n\geq 0}$ is an MVPP of replacement kernel
    $\hat R^{\sss (i)} = R^{\sss (i)}/\sup_{x\in E} R_x(E)$ and weight
    kernel $\hat P_x = \delta_x$ (for all $x\in E$). Let us check that
    Assumption~(A) is satisfied by $(\hat m_n)_{n\geq 0}$. Note that,
    for all $x\in E$ and all bounded measurable function
    $f:E\rightarrow \R$,
    \begin{align*}
    R_x \cdot f := \mathbb E [R^{\sss (i)}_x] \cdot f 
    =\E\Bigg[\sum_{n=0}^{T} G_n \cdot f(x) \Bigg],
    \end{align*}
    where $G_n \cdot f(x)=\E_x\big[f(X_n)\1_{n<\tau_\d}\big]$ is the sub-Markovian
    semi-group of the absorbed process~$X$.

    Moreover, we have that   
    \[
     \hat R_x(E)
     \geq \frac{\E\left[\sum_{n=0}^{T} \alpha_2^{n}  \right]}{\sup_{y\in E} R_y(E)}
     =\frac{1-\E\left[\alpha_2^{T+1}\right]}{(1-\alpha_2)\sup_{y\in E} R_y(E)}
     =: c_1 >0,
    \]
    so that Assumption~(A1) is satisfied (take $\mu$ the law of
    $\frac{(T+1)\wedge \Delta}{\sup_{y\in E} R_y(E)}$, where $T$ and $\Delta$ are independent and
    $\Delta$ is distributed with respect to a geometric law with
    parameter $1-\alpha_2$ on $\{1,2,\ldots\}$).  Moreover, we deduce from~(E) that, for some constant
    $C>0$,
    \begin{align*}
      G_n\cdot V(x)
      &\leq \alpha_1^n \, V(x)+C\big(G_{n-1}\cdot\1_E(x)+\alpha_1 G_{n-2}\cdot\1_E(x)+\cdots+\alpha_1^{n-1}\1_E(x)\big)\\
             &\leq \alpha_1^n \, V(x)+\frac{CG_{n}\cdot\1_E(x)}{\alpha_2}\left(1+\frac{\alpha_1}{\alpha_2}+\cdots+\frac{\alpha_1^{n-1}}{\alpha_2^{n-1}}\right)
    \end{align*}
    where we used (E2) and Markov's property for the second inequality. Since $\alpha_1<\alpha_2$ by assumption, then there exists some constant $C'$ such that    
    \begin{align*}
      R_x \cdot V 
      &= \E\left[\sum_{n=0}^{T} G_n\cdot V(x) \right]
      \leq \E\left[\sum_{n=0}^{T} \alpha_1^n  \right] V(x) 
      + C' \E\left[\sum_{n=0}^T G_n\cdot\1_E(x)\right]\\
    &= \E\left[\sum_{n=0}^{T} \alpha_1^n  \right]\,V(x) + C'\sup_{y\in E} \E_y\big[T\wedge\tau_\d\big].
    \end{align*}
    We thus get
    \[\hat R_x \cdot V
    \leq \theta V(x)
    + \frac{C'\sup_{y\in E} \E_y\big[T\wedge\tau_\d\big]}{\sup_{y\in E} R_y(E)}\]
    where
    \[\theta:=\frac{1-\E\left[\alpha_1^{T+1}\right]}{(1-\alpha_1)\sup_{y\in E} R_y(E)}
      <c_1.\] Assumption~(A2-ii) is thus satisfied by $\hat R$.
    Assumption (A2-iii) is satisfied for any
    $p>2\vee \frac{\ln \theta}{\ln\theta-\ln c_1}$ since
    $R^{\sss (1)}_x(E)\leq T\wedge\tau_\d$, which admits a uniformly
    bounded exponential moment by assumption.  Since (A2-i) is assumed
    to be true under~(E), we deduce that Assumption~(A2) is implied by
    Assumption~(E).

    To prove that (A3) holds true, it is sufficient, by Theorem~2.1
    in~\cite{ChampagnatVillemonais2017}, to prove that $\hat{R}$
    satisfies Assumption~(E) with Lyapunov function $V^{\nicefrac1q}$.
    Since $T\geq 1$ with positive probability, and since $X$ satisfies
    Assumption~(E1), we get that $\hat R$ also satisfies
    Assumption~(E1).  We have already proved that $\hat R$ satisfies
    Assumptions~(A1-2) with Lyapunov function $V$ and hence with
    Lyapunov function $V^{\nicefrac1q}$ (see the proof of
    Lemma~\ref{lem:A->A'}), which implies that $\hat R$ satisfies
    Assumption~(E2) with Lyapunov function
    $V^{\nicefrac1q}$. Moreover, for all $n\geq 0$ and all $x,y\in K$,
    we have
    \begin{align*}
      R^n_x(E)=\E\left[\sum_{\ell=1}^n\sum_{i_\ell=0}^{T_\ell} G_{i_1+\cdots+i_n}\cdot \1_{E}(x)\right]
      \leq a_3 \E\left[\sum_{\ell=1}^n\sum_{i_\ell=0}^{T_\ell} G_{i_1+\cdots+i_n}\cdot \1_{E}(y)\right]
      =a_3 R^n_y(E),
    \end{align*}
    where $T_1,\ldots,T_n$ are i.i.d random variables with distribution $\cal T$ and
    where we used Assumption~(E3) for~$X$; this implies that Assumption (E3) is satisfied by $\hat R$. 
    The fact that $\hat R$ satisfies Assumption~(E4) is an
    immediate consequence of (E4) for $X$, since $T\geq 1$ with
    positive probability.
    By Theorem~2.1 in~\cite{ChampagnatVillemonais2017}, this implies
    that the discrete-time Markov process with transition
    probabilities given by $\hat{R}$ admits a unique
    quasi-stationary distribution $\nu$ such that
    $\nu\cdot V^{\nicefrac1q}<+\infty$. More precisely, it implies that there
    exist $\alpha\in(0,1)$ and $C>0$ such that, for any probability measure
    $\mu$ on $E$ such that $\mu\cdot V^{\nicefrac1q}<+\infty$,
    \begin{align*}
      \left\|\frac{\mu \hat{R}^n}{\mu \hat{R}^n(E)}-\nu\right\|_{TV} \leq C \alpha^n \mu\cdot V^{\nicefrac1q}.
    \end{align*}
    In particular, for all measurable set $A\subset E$,
    \begin{align*}
     \big|\mu \hat{R}^n(A)-\mu\hat{R}^n(E)\,\nu(A)\big|\leq  C \mu\hat{R}^n(E) \alpha^n \mu\cdot V^{\nicefrac1q}.
    \end{align*}
and hence that for all $t\geq 0$,
    \begin{align*}
     \big|\mu \mathrm e^{t \hat{R}}(A)-\mu \mathrm e^{t \hat{R}}(E)\,\nu(A)\big|\leq C \mu \mathrm e^{t \alpha \hat{R}}(E) \mu\cdot V^{\nicefrac1q}.
    \end{align*}
    Since $\alpha\in(0,1)$,
    $\sum_{n=0}^{+\infty} \frac{t^n\alpha^n}{n!}\mu\hat{R}^n(E)$ is
    negligible in front of 
    $\sum_{n=0}^{+\infty} \frac{t^n}{n!}\mu\hat{R}^n(E)$ when $t\to+\infty$, so that
    \begin{align*}
     \left|\frac{\mu \mathrm e^{t \hat{R}}(A)}{\mu \mathrm e^{t \hat{R}}(E)}-\nu(A)\right|\leq C \frac{\mu \mathrm e^{t \alpha \hat{R}}(E)}{\mu \mathrm e^{t\hat{R}}(E)} \mu\cdot V^{\nicefrac1q}\to 0~ \text{ when }t\to+\infty.
    \end{align*}
    Note that ${\mu \mathrm e^{t \hat{R}}(A)}/{\mu \mathrm e^{t \hat{R}}(E)}$ is the
    law of the continuous-time process with sub-Markovian jump kernel
    $\hat R-\mathrm{Id}$ at time~$t$ conditioned not to be absorbed at
    time~$t$. Therefore, we can conclude that (A3) is satisfied by $\hat R$.

    Finally, the continuity of $\hat R_x$ with respect to $x$
    directly derives from the continuity of $\delta_x G_1$
    with
    respect to $x$ and from the uniform boundedness of~$\E_x\left[e^{\lambda T\wedge\tau_\d}\right]$ with respect to~$x$. Therefore, Theorem~\ref{thm:unbal-with-weights} applies and gives that $\hat m_n/\hat m_n(E) = \tilde m_n$ converge almost surely (for the topology of weak convergence) to a probability measure $\nu$. This distribution $\nu$ is the unique quasi-stationary distribution of the process of sub-Markovian jump kernel $\hat R - I$ such that $\nu\cdot V^{\nicefrac1q}<+\infty$.
    
It only remain to show that $\nu$ is indeed equal to $\nu_{QSD}$, the unique quasi-stationary distribution of X such that $\nu_{QSD}\cdot V<+\infty$. 
Since $\nu_{QSD}$ is a quasi-stationary distribution for $X$, we have
    \[\nu_{QSD}R\cdot f
    =\E\sum_{n=0}^T \nu_{QSD} G_n\cdot f
    =\E\sum_{n=0}^T \theta_0^n \nu_{QSD}\cdot f
    =\E\left[\sum_{n=0}^T \theta_0^n\right] \nu_{QSD}\cdot f,\]
    where $\theta_0:=\nu_{QSD}G_1(E)$.
This implies that $\nu_{QSD}$ is a quasi-stationary distribution of the discrete-time sub-Markov process of transitions~$\hat R$. Moreover, since $\nu_{QSD}\cdot V<+\infty$, $V\geq 1$ and $\nicefrac1q<1$, we have $\nu_{QSD}\cdot V^{\nicefrac1q}<+\infty$, implying that $\nu=\nu_{QSD}$, by uniqueness of $\nu$.
  \end{proof}
  
  \subsubsection{Continuous-time sample paths P\'olya urns}

  Let $(X_t)_{t\in[0,+\infty)}$ be the solution in $E=\R^d$ to the
  stochastic differential equation
  \begin{align*}
    \mathrm dX_t=\mathrm dB_t+b(X_t) \mathrm dt,
  \end{align*}
  where $B$ is a standard $d$-dimensional Brownian motion and
  $b:\R^d\mapsto \R^d$ is
  locally H\"older-continuous in $\R^d$. We assume that $X$ is subject to an
  additional soft killing $\kappa:\R^d\mapsto [0,+\infty)$, which is
  continuous and uniformly bounded: the process is sent to a cemetery
  point $\d\notin\R^d$ at rate $\kappa(X_t)$ and we denote by
  $\tau_\d$ the hitting time of $\d$ by $X$. As in the
  discrete-time case, we denote by $\P_x$ and $\E_x$ the law of the process
  $X$ starting from $x\in E\cup\d$ and its associated expectation, and
  we consider $\cal T$ a probability distribution on $[0,+\infty]$ such that
  $\tau_\d\wedge T$ admits under ${\cal T} \otimes \P_x$ an exponential
  moment uniformly bounded with respect to $x\in E$.

  We consider the unbalanced MVPP on $E$ without weights and with
  random replacement kernels $(R^{\sss (i)})_{i\geq 1}$ being i.i.d.\ copies of
  \[
    R^{\sss (1)}_x= \int_0^{T\wedge \tau_\d} \delta_{X_t}\,\mathrm dt, \quad (\forall x\in E),
  \]
  where $(T,X)$ is distributed
  according to ${\cal T}\otimes \P_x$.

  \begin{proposition}
    \label{prop:sample-paths-cont}
  If
  \begin{align*}
    \limsup_{|x|\rightarrow+\infty} \frac{\langle b(x),x\rangle}{|x|}
    < -\frac{3}{2} \|\kappa\|^{\nicefrac12}_\infty,
  \end{align*}
  then Theorem~\ref{thm:unbal-with-weights} applies with
  $V:x\in\R^d\mapsto \exp(\|\kappa\|^{\nicefrac12}_\infty |x|)$. In
  particular, if $m_0\cdot V<\infty$, the normalized sequence of
  probability measures $(\tilde{m}_n)_{n\in\N}$ associated to the MVPP
  with random replacement kernels $(R^{\sss (i)})_{i\geq 1}$ converges
  almost surely to the unique quasi-stationary distribution
  $\nu_{QSD}$ of $X$ such that $\nu_{QSD}\cdot V<+\infty$.
  \end{proposition}

  \begin{remark} The fact that $X$ admits a unique
  quasi-stationary distribution $\nu_{QSD}$ such that
  $\nu_{QSD}\cdot V<+\infty$ is proved
  in~\cite{ChampagnatVillemonais2017}. Proposition~\ref{prop:sample-paths-cont}
  could be generalized to diffusion processes with a non constant
  diffusion coefficient; the proof would be very similar. More
  generally, Condition~(F) of~\cite{ChampagnatVillemonais2017} can be
  used to show that Theorem~\ref{thm:unbal-with-weights} applies to
  other continuous-time processes. We do not develop these
  generalizations further, but provide two simple examples that
  fall into the framework of the proof of
  Proposition~\ref{prop:sample-paths-cont}:
  \end{remark}

  \medskip \textbf{Example 1.}  If $E$ is finite and $X$ is regular
  and irreducible in $E$ (\textit{i.e.}
  $\P_x(\exists t\geq 0,\ s.t.\ X_t=y)>0$ for all $x,y\in E$), and if $\P_x(\tau_\d<+\infty)=1$ for all
  $x\in E$, then Theorem~\ref{thm:unbal-with-weights} applies for any
  probability distribution $\cal T$. (One can take $V=1$.)
  
  \medskip \textbf{Example 2.}  Let $X$ be a continuous-time multitype
  birth and death process, taking values in $E\cup\{\d\}=\mathbb N^d$ for
  some $d\geq 1$, with transition rates
  \[
  q_{x,y}=
  \begin{cases}
    b_i(x) & \text{if }y=x+e_i,\\ 
    d_i(x) & \text{if }y=x-e_i,\\
    0 & \text{otherwise,}
  \end{cases}
  \]
  where $(e_1,\ldots,e_d)$ is the canonical basis of $\mathbb N^d$, and
  $\d=(0,\ldots,0)$.  We assume that $b_i(x)>0$ and $d_i(x)>0$ for all
  $1\leq i\leq d$ and $x\in E$.
  
  If
  \begin{align}
  \label{eq:multi-bd-1}
  \frac{1}{|x|}\sum_{i=1}^d\big(d_i(x)-b_i(x)\big) \to +\infty \quad\text{ when }|x|\to+\infty,
  \end{align}
  or if there exists $\delta>1$ such that
  \begin{align}
  \label{eq:multi-bd-2}
  \sum_{i=1}^d\big(d_i(x)-\delta\,b_i(x)\big)\to+\infty\quad\text{ when }|x|\to+\infty,
  \end{align}
  then Theorem~\ref{thm:unbal-with-weights} applies for any
  probability distribution~$\cal T$ admitting an exponential moment.
  One can choose $V(x)=|x|=x_1+\ldots+x_d$ if~\eqref{eq:multi-bd-1} is
  satisfied, and $V(x)=\exp(\varepsilon x_1+\cdots+\varepsilon x_d)$
  with $\varepsilon>0$ small enough if~\eqref{eq:multi-bd-2} is
  satisfied.  To prove this, one would simply use the same approach as
  in the proof of Proposition~\ref{prop:sample-paths-cont} together
  with the results of~\cite[Example~7]{ChampagnatVillemonais2017} and
  the fact that the killing rate is bounded by $d_1(e_1)+\cdots+d_d(e_d)$.

  If moreover the birth and death process comes back from infinity
  (see for instance~\cite{MartinezSanEtAl2013} for the one dimensional
  case), then $\tau_\d$ admits a uniformly bounded exponential moment
  and hence the conclusion of
  Proposition~\ref{prop:sample-paths-cont} applies for any
  probability distribution $\cal T$.

  \begin{proof}[Proof of Proposition~\ref{prop:sample-paths-cont}]
    For all $n\geq 0$, we let $\hat m_n = m_n/\sup_{x\in E} R_x(E)$;
    note that $(\hat m_n)_{n\geq 0}$ is well defined since
    $\sup_{x\in E} R_x(E)\leq \sup_{x\in E} \E_x[T\wedge
    \tau_\d]<+\infty$, by assumption on the existence of a uniform
    exponential moment for $T\wedge \tau_\d$.  One can check that
    $(\hat m_n)_{n\geq 0}$ is an MVPP of replacement kernel
    $\hat R^{\sss (i)} = R^{\sss (i)}/\sup_{x\in E} R_x(E)$ and weight
    kernel $\hat P_x = \delta_x$ (for all $x\in E$); note that we have
    $\hat Q=\hat R\hat P=\hat R$. Let us check that Assumption~(A) is
    satisfied by $(\hat m_n)_{n\geq 0}$. Note that, for all $x\in E$
    and all bounded measurable function $f:E\rightarrow \R$,
    \begin{align*}
      R_x \cdot f :=\mathbb E [R^{\sss (i)}]\cdot f
      = \E\left[\int_0^T G_t \cdot f(x)\,\mathrm d t \right],
    \end{align*}
    where $G_t \cdot f(x)=\E_x[f(X_t)\1_{t<\tau_\d}]$ is the sub-Markovian
    semi-group of the absorbed process~$X$.

We have
    \begin{align*}
     \hat R_x(E)\geq c_1:=\frac{\E\left[\int_0^T \exp(-\|\kappa\|_\infty t) \,\mathrm dt \right]}{\sup_{y\in E} R_y(E)},
    \end{align*}
    implying that Assumption~(A1) is satisfied (take $\mu = \delta_{c_1}$).

    Let us now check Assumption~(A2). The function $V$ clearly satisfies (A2-i). Moreover, one easily checks that
    \begin{align*}
      \frac{1}{2}\sum_{i=1}^d \frac{\d^2}{\d x_i^2}V(x)+\sum_{i=1}^d b_i(x)\,\frac{\d}{\d x_i} V(x)\leq - (\|\kappa\|_\infty+\varepsilon) V(x)+ C,
    \end{align*}
    for some positive constants $\varepsilon$ and $C$. Setting
    $V(\d)=0$, using Dynkin's formula for the killed process and a
    localization argument, we get that
    \begin{align*}
      \E_x\left[\mathrm e^{(\|\kappa\|_\infty+\varepsilon)t\wedge\tau_\d}V(X_{t\wedge\tau_\d})\right]\leq  V(x) + C\,\E_x\left[\int_0^{t\wedge\tau_\d} \mathrm e^{(\|\kappa\|_\infty+\varepsilon) s} \,\mathrm ds \right].
    \end{align*}
    and hence that
    \begin{align*}
      G_t V(x)=\E_x\left[V(X_t)\1_{t<\tau_\d}\right]
      \leq \mathrm e^{-(\|\kappa\|_\infty+\varepsilon)t} V(x)
      + C\int_0^{t} \mathrm e^{-(\|\kappa\|_\infty+\varepsilon) (t-s)}  \P_x(s<\tau_\d)\,\mathrm ds.
    \end{align*}
    As a consequence, we have
    \begin{align}
      R_x V &= \E\left[\int_0^T G_t V(x)\,\mathrm dt \right]
 \leq \E\left[\int_0^T \mathrm e^{- (\|\kappa\|_\infty+\varepsilon) t} \, \mathrm dt\right] V(x) + \frac{C}{\|\kappa\|_\infty+\varepsilon} \E_x\left[T\wedge\tau_\d\right]\nonumber\\
&\leq  \theta \sup_{y\in E} R_y(E)\,V(x) +\frac{C}{(\|\kappa\|_\infty+\varepsilon)\lambda}\sup_{y\in E} \E_y\left[\mathrm e^{\lambda(T\wedge\tau_\d)}\right]\label{eqNtruc}
    \end{align}
    where
    $\theta:=\E\big[\int_0^T \exp(-\lambda_1 t) \, \mathrm dt\big]/\sup_{y\in E} R_y(E) < c_1$, and where $\sup_{y\in E} \E_y\left[\mathrm e^{\lambda(T\wedge\tau_\d)}\right]<+\infty$ by assumption. Dividing the above inequality by $\sup_{x\in E} R_x(E)$
    entails that Assumption~(A2-ii) is satisfied.  Finally, Assumption
    (A2-iii) is implied by the fact that $R^{\sss (i)}_x(E)$ is
    stochastically dominated by $T\wedge\tau_\d$ under $\P_x$, which
    admits a uniformly bounded exponential moment by assumption.
    As a consequence, we deduce that
    Assumption~(A2) is satisfied by $\hat R$.

    To prove that (A3) holds true, we first prove that $\hat{R}$
    satisfies Assumption~(E) above. 

    Using the same approach
    as in \cite[Proposition~12.1]{ChampagnatVillemonais2017}, we
    deduce that there exist a probability measure ${\pi}$ on ${K}$ and
    two positive constants $b$ and $t_\pi$ such that
    \begin{align*}
      \P_x(X_{t_\pi}\in\cdot)\geq b\,\pi(\cdot),\ \forall x\in {K}.
    \end{align*}
    Since $X$ is an elliptic diffusion process in $\R^d$, it
    satisfies, for any $t>0$, $\inf_{x\in   K} \P_x(X_t\in   K)>0$. Using
    Markov's property, we deduce that, for any $t>t_\pi$, there
    exists a constant $b_t>0$ such that
    $\P_x(X_t\in\cdot)\geq b_t \,\pi(\cdot)$, for all $x\in   K$.  In
    particular, we obtain, for any integer $n\geq 1$ and any
    measurable set $A\subset   K$, that, for all $x\in   K$,
    \begin{align*}
      R^n_x\cdot \1_A & = \E\left[\int_0^{T_1} \cdots \int_0^{T_n} G_{t_1+\cdots+t_n} \cdot\1_A(x) \,\mathrm dt_1\cdots \mathrm dt_n\right]\\
                 & \geq \E\left[\int_0^{T_1} \cdots \int_0^{T_n} b_{t_1+\cdots+t_n} \1_{t_1+\cdots+t_n\geq t_\pi}\,\mathrm dt_1\cdots \mathrm dt_n\right]\,\, {\pi}(A),
    \end{align*}
    where $T_1,\ldots,T_n$ are i.i.d.\ random variables distributed with
    respect to $\cal T$. Since $\P(T_1>0)>0$, we deduce that there
    exists $ {n}_1$ large enough such that
    $\P(t_1+\ldots+t_{ {n}_1}\geq t_\pi)>0$ and hence such
    that
    $\E\left[\int_0^{T_1} \cdots \int_0^{T_{ {n}_1}}
      b_{t_1+\cdots+t_{ {n}_1}}\1_{t_1+\cdots+t_{ {n}_1}\geq
        t_\pi}\,\mathrm dt_1\ldots \mathrm dt_{ {n}_1}\right]>0$. In
    particular, there exists a constant
    $ {\alpha}_0>0$ such that
    \begin{align}
      \label{eqNnewnew}
      \hat{R}^{ {n}_1}_x\cdot \1_A \geq  {\alpha}_0  {\pi}(A\cap  {K}).
    \end{align}
    This entails that Condition~(E1) is satisfied.

    We already proved that $\hat{R}_x(E)\geq c_1$ for all $x\in
    E$. Now, for any fixed $\alpha_1\in(\theta^{\nicefrac1q},c_1)$ and $\rho>0$
    large enough, we deduce from~\eqref{eqNtruc} and as in the proof of Lemma~\ref{lem:A->A'}
    that
    \begin{align*}
      \hat{R}_x\cdot V^{\nicefrac1q} \leq \alpha_1 V^{\nicefrac1q}(x)+\alpha_3\,\1_{|x|\leq \rho},\ \forall x\in \R^d.
    \end{align*}
    Setting $K=\{x\in\R^d,\ |x|\leq \rho\}$, we deduce that
    Condition~(E2) holds true with $\alpha_1$, $\alpha_2=c_1$ and
    $\alpha_3$ large enough, with Lyapunov function $V^{\nicefrac1q}$.

    We also deduce from
    \cite[Proposition~12.1]{ChampagnatVillemonais2017} that
    \begin{align*}
     {\alpha}_3:=\inf_{t\geq 0} \frac{\inf_{x\in {K}} G_t\cdot\1_E(x)}{\sup_{x\in {K}} G_t\cdot\1_E(x)}=  \inf_{t\geq 0} \frac{\inf_{x\in {K}} \P_x(t<\tau_\d)}{\sup_{x\in {K}} \P_x(t<\tau_\d)}>0.
    \end{align*}
    Since $R_x(E)=\E\big[\int_0^T G_t\cdot\1_E(x)\,\mathrm dt\big]$, we get that
    \begin{align*}
     \inf_{t\geq 0} \frac{\inf_{x\in {K}} \hat{R}_x(E)}{\sup_{x\in {K}} \hat{R}_x(E)})=\inf_{t\geq 0} \frac{\inf_{x\in {K}} R_x(E)}{\sup_{x\in {K}} R_x(E)}=  {\alpha}_3>0.
    \end{align*}
    This implies that Condition~(E3)
    holds true.

    Finally, using similar calculations as in the derivation
    of~\eqref{eqNnewnew}, we deduce that Condition~(E4) also holds
    true. This concludes the proof of Condition~(E) with Lyapunov function~$V^{\nicefrac1q}$.

    By Theorem~2.1 in~\cite{ChampagnatVillemonais2017}, this implies
    that the discrete-time Markov process with transition
    probabilities given by $\hat{R}$ admits a unique
    quasi-stationary distribution $\nu_{QSD}$ such that
    $\nu_{QSD}\cdot V^{\nicefrac1q}<+\infty$. 
    Using the same argument as in the proof of (A3) in the proof of Proposition~\ref{prop:sample-paths}, we can show that this implies that (A3) is satisfied by $\hat R$.

    The continuity of $x\mapsto R_x$ (and thus of $x\mapsto \hat R_x$) is a consequence of the continuity of $x\mapsto \E_x\big[f(X_t)\1_{t<\tau_\d}\big]$ for all
    continuous bounded function $f:E\rightarrow\R$ and all $t\geq 0$
    (see, e.g.~\cite[Theorem~7.2.4]{StrookVaradhan});
    therefore, Assumption~(A4) is also satisfied.

    We have proved that Assumption~(A) holds true for the MVPP of replacement kernels $(\hat R^{\sss (i)})$; therefore, Theorem~\ref{thm:unbal-with-weights} applies. To conclude the proof, note that the continuous-time process~$X$ also admits a unique
    quasi-stationary distribution $\mu_{QSD}$ such that
    $\mu_{QSD}\cdot V^{\nicefrac1q}<+\infty$
    (see~\cite[Example~2]{ChampagnatVillemonais2017}), i.e.\ a
    probability measure such that
    $\mu_{QSD} \cdot G_t=\mu_{QSD} \cdot G_t(E)\,\mu_{QSD}$ for all $t>0$. The
    definition of $\hat{R}$ implies that $\mu_{QSD}$ is also a
    quasi-stationary distribution for $\hat{R}$; because
    $\mu_{QSD}\cdot V^{\nicefrac1q}<+\infty$ and by uniqueness, we get that $\nu_{QSD}=\mu_{QSD}$, which concludes the proof.
  \end{proof}  
  
\subsubsection{Application to stochastic-approximation algorithms for the computation of quasi-stationary distributions}\label{sub:approx_QSD}
 It is a difficult question to give an explicit formula for the quasi-stationary distribution of a sub-Markovian process, even when one can prove that this distribution exists and is unique. Stochastic approximation provides algorithms that allow to numerically approximate the quasi-stationary distribution of a given sub-Markovian process.

 The recent papers~\cite{BGZ16, BC15, BCP++} introduce such stochastic
 approximation algorithms for discrete-time sub-Markovian processes
 evolving in compact spaces and~\cite{WangRobertsEtAl2018} studies
 these algorithms for diffusion processes in compact manifolds.  Our
 results allow to extend these convergence results to discrete- and
 continuous-time processes in compact and non-compact spaces. We
 illustrate this approach with the case of the approximation of the
 quasi-stationary distribution of a diffusion process satisfying the
 conditions of Proposition~\ref{prop:sample-paths-cont} by a
 stochastic-approximation algorithm. This particular example was not
 covered by the previous literature since it is a continuous-time
 process and its state space is not compact.

As in the previous section, let $(X_t)_{t\in[0,+\infty)}$ be the
solution in $E=\R^d$ to the stochastic differential equation
\begin{align*}
  \mathrm dX_t=\mathrm dB_t+b(X_t) \mathrm dt,
\end{align*}
where $B$ is a standard $d$-dimensional Brownian motion and
$b:\R^d\mapsto \R^d$ is locally H\"older continuous in $\R^d$. We
assume that $X$ is subject to an additional soft killing
$\kappa:x\mapsto [0,+\infty)$, which is continuous, uniformly bounded
and such that $\kappa\geq 1$. Note that the quasi-stationary
distribution of $X$ with killing rate $\kappa$ is the same as the
quasi-stationary distribution of $X$ with a killing rate $\kappa-1$.

We also assume that
\begin{align*}
  \limsup_{|x|\rightarrow+\infty} \frac{\langle b(x),x\rangle}{|x|}
  < -\frac{3}{2} \|\kappa\|^{\nicefrac12}_\infty,
\end{align*}
so that the process $X$ admits a unique quasi-stationary distribution
$\nu_{QSD}$ such that $\nu_{QSD}\cdot V<+\infty$, where
$V:x\in\R^d\mapsto \exp(\|\kappa\|^{\nicefrac12}_\infty |x|)$ (see the
previous subsection for details). 

We consider the self-interacting process $(Y_t)_{t\geq 0}$ evolving
with the same dynamic of~$X$ but, at rate $\kappa$, instead of being killed, 
it jumps to a new position chosen accordingly to its
empirical  occupation measure $\frac{1}{t}\int_0^t \delta_{Y_s}\,\mathrm ds$. More formally, it evolves following the dynamic
\begin{align*}
  \mathrm dY_t=\mathrm dB_t+b(Y_t)\mathrm dt+\mathrm dN_t,\quad Y_0=y\in \R^d,
\end{align*}
where $(N_t)_{t\geq 0}$ is a time inhomogeneous pure jump process with jump measure given by
\begin{align*}
  \frac{\kappa(Y_{t-})}{t}\int_0^t \delta_{Y_s-Y_{t-}}\,\mathrm ds.
\end{align*}

\begin{proposition}
  The empirical occupation measure
  $\frac{1}{t}\int_0^t \delta_{Y_s}\,\mathrm ds$ converges almost-surely when
  $t\to+\infty$, with respect to the topology of weak convergence, to the
  quasi-stationary distribution $\nu_{QSD}$ of $X$.
\end{proposition}

\begin{proof}
  Denote by $0<\tau_1<\tau_2<\ldots$ the jump times of $Y$ and set
  $\tau_0=0$. Then, for all $n\geq 0$ and conditionally on
  $Y_{\tau_n}$,
  \begin{align*}
    \int_{\tau_n}^{\tau_{n+1}} \delta_{Y_s} \, \mathrm ds = R^{\sss (n+1)}_{Y_{\tau_n}},
  \end{align*}
  where $R^{(n+1)}$ is defined as in the proof of
  Proposition~\ref{prop:sample-paths-cont}. Moreover, $Y_{\tau_n}$ is
  distributed according to the probability measure
  $\frac{1}{\tau_n}\int_0^{\tau_n} \delta_{Y_s}\,\mathrm ds$.  As a
  consequence, setting $m_0=\int_0^{\tau_1} \delta_{Y_s}\,\mathrm ds$ (which
  satisfies $m_0\cdot V<+\infty$ almost surely) and
  $m_n:=\frac{1}{\tau_{n+1}}\int_0^{\tau_{n+1}} \delta_{Y_s}\,\mathrm ds$, the
  sequence $(m_n)_{n\in\N}$ has the law of the MVPP of
  Proposition~\ref{prop:sample-paths-cont}. Applying this proposition
  with $T=+\infty$ almost surely (note that $\kappa\geq 1$ implies
  that $\tau_\d\wedge \infty=\tau_\d$ admits a uniformly bounded
  exponential moment), we obtain that
  \begin{align}
    \label{eq:step1-stoc-alg}
    \frac{1}{\tau_n}\int_0^{\tau_n} \delta_{Y_s}\,\mathrm ds\xrightarrow[n\rightarrow+\infty]{a.s.} \nu_{QSD}
  \end{align}
  with respect to the topology of weak convergence.

  Since $\kappa\geq 1$, one can couple the sequence
  $(\tau_{n+1}-\tau_n)_{n\geq 0}$ with a sequence of i.i.d.\ random
  variables $(D_n)_{n\geq 0}$ with exponential law of parameter $1$
  such that $0\leq \tau_{n+1}-\tau_n\leq D_n$ almost surely for all
  $n\geq 0$. Moreover $\tau_n\rightarrow+\infty$ almost surely when
  $n\rightarrow+\infty$ (this is due to the fact that $\kappa$ is
  uniformly bounded). Hence, using~\eqref{eq:step1-stoc-alg}, we get
  \[
    \frac{1}{\tau_n}\int_0^{\tau_{n+1}} \delta_{Y_s}\,\mathrm ds\xrightarrow[n\rightarrow+\infty]{a.s.} \nu_{QSD}
  \]
  and
  \[
    \frac{1}{\tau_{n+1}}\int_0^{\tau_{n}} \delta_{Y_s}\,\mathrm ds\xrightarrow[n\rightarrow+\infty]{a.s.} \nu_{QSD}
  \]

  For all $t\geq 0$, we define
  $\alpha(t):= \sup\{n\geq 0,\ \tau_n\leq t\}$. In particular, for all
  $t\geq 0$, $\alpha(t)<+\infty$,
  $\tau_{\alpha(t)}\leq t <\tau_{\alpha(t)+1}$ and
  $\alpha(t)\to+\infty$ almost
  surely when $t\to+\infty$. 
  As a consequence, for all bounded continuous function
  $f:\R^d\rightarrow[0,+\infty)$,
  \begin{align*}
    \frac{1}{\tau_{\alpha(t)+1}}\int_0^{\tau_{\alpha(t)}} f(Y_s)\,\mathrm ds
    \leq \frac{1}{t}\int_0^t f(Y_s)\,\mathrm ds 
    \leq \frac{1}{\tau_{\alpha(t)}}\int_0^{\tau_{\alpha(t)+1}} f(Y_s)\,\mathrm ds.
  \end{align*}
  This and the above convergence results allow us to conclude the proof.
\end{proof}

\begin{remark}
  Since
  the submission of this paper, Bena\"im, Champagnat \&
  Villemonais~\cite{BCV} proved almost sure convergence of a similar
  stochastic approximation algorithm, where the diffusion process is
  resampled according to its empirical occupation measure when it hits
  the boundary of a bounded domain. On the one hand, their result do
  not apply to the model studied in this section where the state space
  is not bounded; on the other hand, our result do not apply to their
  situation, since Assumption (A1) would fail in that case.
\end{remark}

\section{Proof of Theorem~\ref{thm:unbal-with-weights}}\label{sec:proof}

Let us define an auxiliary sequence of random distributions: 
let $\eta_0 = 0$, and, for all $n\geq 1$,
\begin{align*}
\eta_n=\eta_{n-1}+\delta_{Y_n} = \sum_{i=1}^n\delta_{Y_i}.
\end{align*}
Recall that, by definition,
\[m_n = m_0 + \sum_{i=1}^n R_{Y_i}^{\sss (i)} = m_0 + \sum_{i=1}^n
  \delta_{Y_i}R^{\sss (i)}\] 
  and that, conditionally on the sigma-algebra
$\mathcal F_n$ generated by $\{m_i\}_{0\leq i\leq n}\cup\{Y_i\}_{1\leq i\leq n}$, 
the random variable $Y_{n+1}$ is distributed
according to $m_n P/m_n P(E)$ and $R^{\sss (n+1)}$ is chosen independently
of $\mathcal F_n$ and $Y_{n+1}$.

We set $\tilde\eta_0 = 0$, and, for all $n\geq 1$,
\[\tilde \eta_n = \frac{\eta_n}{\eta_n(E)} = \frac{\eta_n}{n}.\]
We first prove that
$\tilde\eta_n$ converges almost surely weakly to~$\nu$ when $n$ goes to infinity and
then deduce almost-sure convergence of $\tilde m_n$
to~$\nu R/\nu R(E)$:
\begin{proposition}\label{prop:eta}
  Under the Assumptions (T, A1, A'2, A3, A4),
  the sequence $(\tilde\eta_n)_{n\geq 0}$ converges weakly almost surely to~$\nu$ when $n$ goes to infinity.
  Said differently,
  \[\frac1n\sum_{i=1}^n \delta_{Y_i} \to \nu \quad\text{ almost surely when }n\to\infty.\]
\end{proposition}

\subsection{Proof of Proposition~\ref{prop:eta}}
We consider the dynamical system 
defined by
\begin{align}
\label{eq:dyn-sys}
\frac{d\mu_t\cdot f}{dt}=\mu_t Q \cdot f -\mu_t Q(E)\,\mu_t \cdot f,
\end{align}
for all bounded continuous functions $f:E\rightarrow\R$,
  where $(\mu_t)_{t\geq 0}$ shall not depend on $f$. Existence, uniqueness and
continuity properties of the flow induced by this dynamical system are
stated and proved in Lemma~\ref{lem:flow}.

To prove almost-sure convergence of $\tilde\eta_n$ to $\nu$
(i.e. Proposition~\ref{prop:eta}), we prove that a linearization of it
is a \textit{pseudo-asymptotic trajectory} (see Section~3
of~\cite{Benaim1999}) of the semi-flow induced by the dynamical
system~\eqref{eq:dyn-sys}.  To do so, we need to prove several
intermediate results: In Lemma~\ref{lem:algo_sto}, we write down the
studied stochastic algorithm. In Lemma~\ref{lem:C}, we prove that the
expectation of $V$ with respect to the measure-valued process remains
bounded. In Lemma~\ref{lem:U}, we prove almost-sure convergence of the
quantity introduced in Proposition~4.1 of~\cite{Benaim1999} to control
the error term between the dynamical system~\eqref{eq:dyn-sys} and its
linearized counterpart (the almost-sure convergence of this error to
zero is sometimes called the Kushner \& Clark's condition).  In
Lemma~\ref{lem:compacity}, we prove that the sequence
$(\tilde\eta_n)_n$ is relatively compact for the topology of weak
convergence on $\cP(E)$. All these elements allow us to conclude the
proof of Proposition~\ref{prop:eta} using standard
stochastic-approximation methods, as developed
in~\cite{BenaimLedouxEtAl2002}.

From now on, we assume that all the hypotheses of Proposition~\ref{prop:eta} hold.
\begin{lemma}\label{lem:algo_sto}
 For all $n\geq 1$, we have
\[\tilde{\eta}_{n+1}-\tilde \eta_n = \gamma_{n+1}\Big(F(\tilde \eta_n) + U_{n+1}\Big),\]
where
\[\gamma_{n+1}=\frac{1}{\eta_{n+1}(E)\tilde{\eta}_n Q(E)},\]
and
\begin{align*}
F(\tilde{\eta}_n)&=\tilde{\eta}_n Q-\tilde{\eta}_n Q(E)\,\tilde{\eta}_n,\\
U_{n+1}&=\tilde{\eta}_n Q(E)\delta_{Y_{n+1}}-\tilde{\eta}_nQ.
\end{align*}
\end{lemma}

{ The term $\gamma_{n+1}$ may be interpreted as the step size of a stochastic
Euler scheme approximation of Equation~\eqref{eq:dyn-sys} and it decreases to $0$ when $n\rightarrow+\infty$. For instance, in the simple
case where $Q(E)=1$, $\gamma_{n+1}$ equals $\nicefrac{1}{(n+1)}$.}

\begin{proof}
The result directly follows from
\[
\tilde{\eta}_{n+1}
=\left(1-\frac1{n+1}\right)\tilde \eta_n  + \frac1{n+1}\delta_{Y_{n+1}}
=\tilde \eta_n + \frac{1}{n+1}\big(\delta_{Y_{n+1}} - \tilde \eta_n\big).\qedhere
\]
\end{proof}

\begin{lemma}\label{lem:cv_sigma_k}
Fix $c'\in(\theta,c_1)$, for all $k\geq 1$, we let
\begin{equation}\label{eq:def_sigma_k}
\sigma_k = \inf\big\{n\geq k \colon m_nP(E)<c' n\big\}.
\end{equation}
We have $\mathbb P\big(\cup_{k\geq 1} \{\sigma_k = \infty\}\big) = 1$.
\end{lemma}

  \begin{proof}
    Recall that $m_nP(E) = m_0P(E) + \sum_{i=1}^n R^{\sss (i)}_{Y_i}P(E)$, and, therefore,
    \[m_nP(E) = m_0P(E) + \sum_{i=1}^n Q^{\sss (i)}_{Y_i}(E).\]
    Assumption~(A1) and, conditionally on $Y_1,\ldots, Y_n,\ldots$, the
    independence of the random variables $Q^{\sss (i)}_{Y_i}(E)$
    entails (by coupling) that there exists a sequence of independent random
    variables $Z_1,\ldots,Z_n,\ldots$ with law $\mu$ such that,
    conditionally on $Y_1,\ldots, Y_n,\ldots$, we have
    $Q^{\sss (i)}_{Y_i}(E)\geq Z_i$ for all $i\geq 1$. 
    The law of
    large numbers hence implies that
    \[\liminf_{n\rightarrow+\infty} \frac{m_n P(E)}{n} \geq c_1\quad\text{almost surely},\]
    which concludes the proof.
  \end{proof}

We claimed that Assumption (A1) can be replaced by Equation~\eqref{eq:alt-A1} in Theorem~\ref{thm:unbal-with-weights}, to prove this claim, we need to prove Lemma~\ref{lem:cv_sigma_k} in this alternative setting:
\begin{proof}[Proof of Lemma~\ref{lem:cv_sigma_k} with Assumption~(A1) replaced by~\eqref{eq:alt-A1}]
Recall that \[m_nP(E) = m_0P(E) + \sum_{i=1}^n R^{\sss (i)}_{Y_i}P(E),\] and, therefore,
\[m_nP(E) = m_0P(E) + \sum_{i=1}^n \mathbb E_{i-1}Q^{\sss (i)}_{Y_i}(E)
+ \sum_{i=1}^n \big(Q^{\sss (i)}_{Y_i}(E) - \mathbb E_{i-1}Q^{\sss (i)}_{Y_i}(E)\big),\]
where $\mathbb E_{i-1}$ denotes the expectation conditionally on $(m_1, \ldots, m_{i-1})$.
Note that, since $Q^{\sss (i)}$ is independent from $\cF_{i-1}$ and $Y_i$, we have
\begin{equation}\label{eq:lower_bound_mP}
\sum_{i=1}^n \mathbb E_{i-1}Q^{\sss (i)}_{Y_i}(E)
= \sum_{i=1}^n \mathbb E_{i-1} Q_{Y_i}(E) \geq c_1 n,
\end{equation}
by Assumption (A1).
Also note that 
\[M_n:=\sum_{i=1}^n \big(Q^{\sss (i)}_{Y_i}(E) - \mathbb
  E_{i-1}Q^{\sss (i)}_{Y_i}(E)\big)\] is a martingale. Using Lemma~1
in~\cite{Chatterji1969} (without loss of generality, we assume that
$\beta\in(1,2]$), one deduces from Assumption~\eqref{eq:alt-A1} that
\begin{align*}
  \E|M_n|^\beta
  &\leq  2\,\sum_{i=1}^n \mathbb E_{i-1}\big|Q^{\sss (i)}_{Y_i}(E) - \mathbb E_{i-1}Q^{\sss (i)}_{Y_i}(E)\big|^\beta\\
  &\leq 2 n \sup_{x\in E}\, \E\left|Q_x^{(i)}(E)-Q_x(E)\right|^{\beta}.
\end{align*}
Hence, using~\eqref{eq:alt-A1}, we get that the sequence
$(n^{-1} \mathbb E \left|M_n\right|^\beta)_{n\geq 1}$ is bounded. This
implies, by an immediate adaptation of Theorem~1.3.17 in~\cite{Duflo}
(the main point is to use Doob's inequality instead of Kolmogorov's
inequality), that $n^{-1}M_n$ goes almost surely to zero when $n$ goes
to infinity.

Therefore, we have that, almost surely when $n\to+\infty$,
\[m_nP(E) = \sum_{i=1}^n \mathbb E_{i-1}Q^{\sss (i)}_{Y_i}(E)+o(n),\]
and, using Equation~\eqref{eq:lower_bound_mP}, we get
\[m_nP(E) \geq c_1 n +o(n) \text{ almost surely},\]
which concludes the proof because $c'<c_1$.
\end{proof}

\begin{lemma}\label{lem:C}
For all $k\geq 1$, there exists a constant $C_k>0$  such that
\[\sup_{n\geq 1} \mathbb E\left[\frac{\eta_{n\wedge\sigma_k} \cdot V}{n}\right] \,\vee\, 
\sup_{n\geq 1}\E\left[\frac{m_{n\wedge\sigma_k}P\cdot V}{n}\right]\,\vee\,
\sup_{n\geq 1}\E\big[V(Y_{n+1})\1_{n<\sigma_k}\big]\leq C_k.\]
\end{lemma}

\begin{proof} Fix $n\geq k+1$,
We have
\begin{equation}\label{eq:exp-cont}
\E\left[\frac{{\eta}_{(n+1)\wedge \sigma_k}\cdot V}{n+1}\right]
=\left(1-\frac1{n+1}\right)\mathbb E\left[\frac{\eta_{n\wedge \sigma_k}\cdot V}{n}\right]
+\frac{\E \big[V(Y_{n+1})\1_{n<\sigma_k}\big]
}{n+1}
\end{equation}
Note that, by definition of $\sigma_k$ (see Equation~\eqref{eq:def_sigma_k}),
we have, almost surely and for all $n\in\{k+1,\ldots,\sigma_k-1\}$,
\[
  m_nP(E)  \geq c'n.
\]
Hence, by definition of $Y_{n+1}$, we have (recall that $m_n$, and thus $m_nP$, is
assumed to be a positive measure almost surely), for all $n\geq k+1$,
\begin{align*}
\E\left[V(Y_{n+1})\1 _{n<\sigma_k}\right]
  &=\E\left[\frac{m_n P\cdot V}{m_n P(E)}\,\1_{n<\sigma_k} \right]\leq \frac{1}{c'n}\E\left[m_{n\wedge\sigma_k} P\cdot V \right]\\
  &= \frac{1}{c'n}\,\E\left[m_0P\cdot V+\sum_{i=1}^{n\wedge\sigma_k} Q_{Y_i}^{\sss (i)}\cdot V\right]\\
  &\leq \frac{1}{c'n}\,\E\left[m_0P\cdot V+\sum_{i=1}^{n} Q_{Y_i}^{\sss (i)}\cdot V\,\1_{i\leq \sigma_k}\right]\\
  &=\frac{1}{c'n}\,\E\left[m_0P\cdot V+\sum_{i=1}^{n} Q_{Y_i}\cdot V\,\1_{i\leq \sigma_k}\right]
\end{align*}
where the last equality is obtained by conditioning on
$\mathcal F_{i-1}$ and $Y_i$, and using the fact that
$\1_{i\leq \sigma_k}$ is measurable with respect to
 $\mathcal F_i \cup \sigma(Y_i)$ and that $Q^{\sss (i)}$ is independent of
$\mathcal F_i \cup \sigma(Y_i)$.  We thus get, using the
Lyapunov assumption (A'2-i) in the second inequality,
\begin{align}
  \E\left[V(Y_{n+1})\1_{n<\sigma_k}\right]
  &\leq \frac{1}{c'n}\E\left[m_{n\wedge\sigma_k} P\cdot V \right]\label{eq:unif_bound} \\
  &\leq \frac{m_0 P\cdot V+\E[\eta_{n\wedge\sigma_k}Q\cdot V]}{c'n} \nonumber   \\
  &\leq \frac{m_0 P\cdot V + nK + \theta \E\left[\eta_{n\wedge\sigma_k}\cdot V\right]}{c'n}\nonumber\\
  &\leq  \frac{m_0 P\cdot V+nK}{c'n} + \frac{\theta}{c'} \E\left[\frac{\eta_{n\wedge\sigma_k}\cdot V}{n}\right] .\label{eq:unif_bound_bis}
\end{align}
Thus, using Equation~\eqref{eq:exp-cont}, we get, for all $n\geq k+1$,
\[
  \E\left[\frac{{\eta}_{(n+1)\wedge\sigma_k}\cdot V}{n+1}\right] \leq
  \left(1-\frac{1-\nicefrac{\theta}{c'}}{n+1}\right)\E\left[\frac{{\eta}_{n\wedge\sigma_k}\cdot
    V}{n}\right] + \frac{m_0P\cdot V + nK}{c'n(n+1)}.\] One easily checks
that $\E[\tilde{\eta}_{n\wedge\sigma_k}\cdot V]<+\infty$ for all
$n\leq k$ and, since we assumed that $m_0 P\cdot V<+\infty$ and since
$\theta<c'<1$, we can infer that
$\E[\eta_{n\wedge\sigma_k}\cdot V/n]$ is uniformly bounded in $n$.
Finally, the inequality between~\eqref{eq:unif_bound} and
  \eqref{eq:unif_bound_bis} implies that both
  $\E\left[{m_{n\wedge\sigma_k}P\cdot V}/{n}\right]$ and
  $ \E\left[V(Y_{n+1})\1_{n<\sigma_k}\right]$ are also uniformly
  bounded in $n$.
  \end{proof}

\begin{lemma}[Kushner \& Clark's condition]\label{lem:U}
Set $W = V^{\nicefrac1q}$.
Almost surely~~$\lim_{n\to+\infty} \sum_{\ell=1}^{n} \gamma_{\ell} U_{\ell}\cdot W$~~exists and is finite.
\end{lemma}

\begin{proof}
Fix $k\geq 1$.  Following~\cite[Lemma~1]{Renlund2010}, we let
  $Z_\ell = \gamma_\ell U_\ell\cdot W$ and
  $M_n = \sum_{\ell=1}^{n\wedge\sigma_k} (Z_\ell - \E_{\ell-1}Z_\ell)$, where $\E_{\ell-1}$ denotes the conditional expectation conditionally on ${\mathcal F}_{\ell-1}$.
  The rest of the proof is done into two steps: first, we prove that the martingale $(M_n)_{n\geq 0}$ is uniformly bounded in $L^r$, implying that it converges almost surely, second, we prove that $\sum_{\ell=1}^{n\wedge\sigma_k}\E_{\ell-1}Z_\ell$ converges almost surely when $n$ tends to infinity.
  
  {\it Step 1:} { Using Jensen's inequality, we get that the constant~$r$ can be assumed to be arbitrarily small as long as it is larger than~1; in particular, we can assume that $r<2$.}
  Using this together with Lemma~1 in~\cite{Chatterji1969}, we get 
  \begin{equation}
    \label{eq:M_n}
    \mathbb E |M_n|^r
    \leq  2\sum_{\ell=1}^n \mathbb E\left[\big|Z_\ell - \E_{\ell-1}Z_\ell\big|^r\1_{\ell\leq \sigma_k}\right]
    \leq 8\sum_{\ell=1}^n \E\left[|Z_\ell|^r\1_{\ell\leq \sigma_k}\right].
  \end{equation}
  Recall that, by definition, 
  $U_\ell=\tilde \eta_{\ell-1}Q(E)\delta_{Y_\ell}-\tilde \eta_{\ell-1}Q$ and
  $\gamma_\ell = \big(\eta_\ell(E)\tilde\eta_{\ell-1}Q(E)\big)^{-1}$ (see Lemma~\ref{lem:algo_sto});
  therefore, we have
  \begin{align*}
    \E\left[|Z_\ell|^r\1_{\ell\leq \sigma_k}\right]
    &=\E\left[\frac{\left|\tilde{\eta}_{\ell-1}Q(E) W(Y_\ell) - \tilde{\eta}_{\ell-1}Q\cdot W\right|^r}
      {\left|\eta_{\ell}(E)\tilde{\eta}_{\ell-1}Q(E)\right|^r}\1_{\ell\leq \sigma_k}\right]\\
    &\leq 2\mathbb E\left[\frac{V(Y_\ell)}{\eta_\ell(E)^r}\1_{\ell\leq \sigma_k}
      + \frac{1}{\eta_\ell(E)^r}\left|\frac{\tilde \eta_{\ell-1}Q}{\tilde\eta_{\ell-1}Q(E)}\cdot W\right|^r\1_{\ell\leq \sigma_k}
      \right],
  \end{align*}
  where we recall that $W=V^{\nicefrac1q}$.
  Using Assumption~(A'2-iv) and the fact that
  $\eta_\ell(E)=\ell$, $\tilde\eta_{\ell-1}Q(E)\geq c_1$ (see  Assumption~(A1)) and
  $\mathbb E[V(Y_\ell)\1_{\ell\leq \sigma_k}] \leq C_k$ (see Lemma~\ref{lem:C}), we get
  \[
    \E \left[|Z_\ell|^r\1_{\ell\leq \sigma_k}\right] \leq
    \frac{2C_k}{\ell^r} + \frac{2\mathbb E\big[\tilde\eta_{\ell-1}\big|Q\cdot
      W\big|^r\1_{\ell\leq \sigma_k}\big]}{c_1^r\ell^r} \leq
    \frac{2}{\ell^r}\left(C_k+\frac{BC_k}{c_1^r}\right),
  \]
  where we used Lemma~\ref{lem:C} and Assumption~(A'2-iv) for the last inequality { (recall that, by Jensen's inequality, $r$ can be assumed to be arbitrarily close to one, and thus smaller than~$q$, in particular)}.
  Using Equation~\eqref{eq:M_n}, this implies that the martingale
  $(M_n)_{n\geq 0}$ is uniformly bounded in $L^r$ and hence that it
  converges almost surely. 

\medskip
  {\it Step 2:} Using the fact that $\eta_\ell(E)=\ell$, we also have
  \begin{align*}
    \E\big|\E_{\ell-1}[Z_\ell]\1_{\ell\leq\sigma_k}\big|
    &= \E\left|\E_{\ell-1}\left[
      \frac{\tilde{\eta}_{\ell-1}Q(E) W(Y_\ell)-\tilde{\eta}_{\ell-1}Q\cdot W}
      {\eta_{\ell}(E)\tilde{\eta}_{\ell-1}Q(E)}\right]\1_{\ell\leq\sigma_k}\right|\\
    &{=} \frac{1}{\ell}\,
      \E\left|\E_{\ell-1}\left[W(Y_\ell)-\frac{\tilde{\eta}_{\ell-1}Q\cdot W}
      {\tilde{\eta}_{\ell-1}Q(E)}\right]\1_{\ell\leq\sigma_k}\right|\\
    &=\frac{1}{\ell}\,\E\left|\frac{m_{\ell-1}P\cdot W}{m_{\ell-1}P(E)}\1_{\ell\leq\sigma_k}-\frac{\eta_{\ell-1}Q\cdot W}{\eta_{\ell-1}Q(E)}\1_{\ell\leq\sigma_k}\right|,
  \end{align*}
  {  where we used for the last equality that the conditional distribution of $Y_\ell$ given $\cF_{\ell-1}$ is $m_{\ell-1}P/m_{\ell-1}P(E)$.}
  By the triangular inequality, and using the fact that
  $\eta_{\ell-1}Q(E) \geq c_1(\ell-1)$ almost surely (see
  Assumption~(A1)), we get
  \begin{align}
    \E\big|\E_{\ell-1}[Z_\ell]\1_{\ell\leq\sigma_k}\big|
    &\leq 
      \frac{1}{c_1\ell(\ell-1)}\,\E\Big[\left|m_{\ell-1} P\cdot W - \eta_{\ell-1}Q\cdot W\right|\1_{\ell\leq\sigma_k}\Big]\nonumber\\
    &\quad\quad\quad\quad+\frac{1}{\ell}\,\E\left[\left|\frac{1}{m_{\ell-1}P(E)}-\frac{1}{\eta_{\ell-1}Q(E)}\right| m_{\ell-1}P \cdot W\1_{\ell\leq\sigma_k}\right]\label{eq:substep}.
  \end{align}
Let us first bound the first term of the above sum. Using Jensen's inequality and Lemma~1 in~\cite{Chatterji1969} (note that $(m_{\ell\wedge\sigma_k} P-\eta_{\ell\wedge\sigma_k} Q)_{\ell\geq 0}$ is a martingale), we get
\begin{align*}
  \E\Big[\big|m_{\ell-1}P\cdot W - \eta_{\ell-1}&Q\cdot W\big|\1_{\ell\leq\sigma_k}\Big]^{r}
  {\leq \E\Big[\big|m_{\ell-1}P\cdot W - \eta_{\ell-1}Q\cdot W\big|\1_{\ell-1\leq\sigma_k}\Big]^{r}}\\
&\leq \E\big|m_{(\ell-1)\wedge\sigma_k}P\cdot W - \eta_{(\ell-1)\wedge\sigma_k}Q\cdot W\big|^{r}\\
                                              &\leq 2(m_0P\cdot W)^{r} + 2\sum_{i=1}^{\ell-1}\E\left[\left|Q^{\sss (i)}\cdot W(Y_i)-Q\cdot W(Y_i)\right|^{r}\1_{i\leq \sigma_k}\right]\\
                                              & \leq 2\big(m_0 P\cdot W\big)^{r}+2\sum_{i=1}^{\ell-1}  B\,\E\left[V(Y_i)\1_{i\leq\sigma_k}\right],
\end{align*}
where we used the fact that $\1_{\ell\leq \sigma_k}\leq\1_{\ell-1\leq \sigma_k}$ almost surely, 
that $\1_{i\leq\sigma_k}$ is measurable with respect to
$\mathcal F_{i-1}\cup\sigma(Y_i)$,
and  Assumption~(A'2-iv).
Finally, Lemma~\ref{lem:C} implies that there exists a constant $C'_k>0$ such that
\begin{align}
\label{eq:substep1}
&\E\left[\left|m_{\ell-1}P\cdot W - \eta_{\ell-1}Q\cdot W\right|\1_{\ell\leq\sigma_k}\right]
\leq C'_k\,\left((m_0P\cdot W)^r + \ell-1\right)^{\nicefrac1r}.
\end{align}
Let us now look at the second term in the right-hand side of Equation~\eqref{eq:substep};
using Assumption~(A1), we have that
\begin{align*}
  &\E\left[\left|\frac{1}{m_{\ell-1}P(E)}-\frac{1}{\eta_{\ell-1}Q(E)}\right|
    m_{\ell-1}P\cdot W\1_{\ell\leq\sigma_k}\right]\\
    &\hspace{.5cm}= \E\left[\frac{|\eta_{\ell-1}Q(E)-m_{\ell-1}P(E)|}{\eta_{\ell-1}Q(E)}\,
      \frac{m_{\ell-1}P\cdot W}{m_{\ell-1}P(E)}\1_{\ell\leq\sigma_k}\right]\\
    &\hspace{.5cm}\leq \frac{1}{c_1(\ell-1)}\E\left[\left|\eta_{\ell-1}Q(E)-m_{\ell-1}P(E)\right|
      \frac{m_{\ell-1}P\cdot W}{m_{\ell-1}P(E)}\1_{\ell\leq\sigma_k}\right]\\
    &\hspace{.5cm}\leq \frac{1}{c_1(\ell-1)}\E\Big[\left|\eta_{\ell-1}Q(E)-m_{\ell-1}P(E)\right|^{p}\1_{\ell\leq\sigma_k}\Big]^{\nicefrac1p}
    \E\left[\left(\frac{m_{\ell-1}P\cdot W}{m_{\ell-1}P(E)}\right)^q\1_{\ell\leq\sigma_k}\right]^{\nicefrac1q}\\
    &\hspace{.5cm}\leq   \frac{C_k^{\nicefrac1q}}{c_1(\ell-1)} \E\Big[\left|\eta_{\ell-1}Q(E)-m_{\ell-1}P(E)\right|^{p}\1_{\ell\leq\sigma_k}\Big]^{\nicefrac1p},
    \end{align*}
where we used H\"older's inequality (in the second inequality),
Jensen's inequality and Lemma~\ref{lem:C} (in the last
inequality). Now, using the main result
of~\cite{DharmadhikariFabianEtAl1968}, we obtain that, for some
constant $d_p>0$,
\begin{multline*}
  \E\left[\left|\eta_{\ell-1}Q(E)-m_{\ell-1}P(E)\right|^{p}\1_{\ell\leq\sigma_k}\right]
  \leq \E\left[\left|\eta_{\ell\wedge\sigma_k-1}Q(E)-m_{\ell\wedge\sigma_k-1}P(E)\right|^{p}\right]\\
  \begin{aligned}
  &\leq 2^{p-1}\bigg[m_0P(E)^p+d_p(\ell-1)^{p/2-1}\sum_{i=1}^{\ell-1} \E\left[\big|Q_{Y_i}(E)-Q^{(i)}_{Y_i}(E)\big|^p\1_{i<\sigma_k}\right]\bigg]\\
  &\leq 2^{p-1}\bigg[m_0P(E)^p+d_p(\ell-1)^{p/2-1}\sum_{i=1}^{\ell-1} A\,\E\left[V(Y_i)\1_{i\leq\sigma_k}\right]\bigg],
  \end{aligned}
\end{multline*}
where we used Assumption~(A'2-iii). Hence, using Lemma~\ref{lem:C}, we deduce that
\begin{equation}
  \label{eq:substep2}
  \E\left[\left|\frac{1}{m_{\ell-1}P(E)}-\frac{1}{\eta_{\ell-1}Q(E)}\right| m_{\ell-1}P\cdot W\right]^p
  \leq \frac{C_k^{\nicefrac pq}2^{p-1}}{c_1^p(\ell-1)^p}\left(m_0P(E)^p+d_p (\ell-1)^{\nicefrac p2} A C_k\right).
\end{equation}

Finally, from inequalities \eqref{eq:substep},\eqref{eq:substep1} and~\eqref{eq:substep2}, 
we deduce that $\sum_{\ell=1}^\infty\E\left|\E_{\ell-1}Z_\ell\1_{\ell\leq\sigma_k}\right|<\infty$. 
As a consequence, $\sum_{\ell=1}^{\sigma_k}|\E_{\ell-1}Z_\ell|<\infty$ almost surely, 
implying that $\sum_{\ell=1}^{n\wedge\sigma_k}\E_{\ell-1}Z_\ell$ converges almost surely when $n\rightarrow\infty$. 
Recall that we have proved that $M_n = \sum_{\ell=1}^{n\wedge\sigma_k} (Z_\ell-\mathbb E_{\ell-1}Z_\ell)$ converges almost surely when $n$ goes to infinity
(we showed earlier that it was uniformly bounded in $L^r$).
Therefore, we can imply that $\sum_{\ell=1}^{n\wedge\sigma_k} Z_k$ converges almost surely. Since $\P(\cup_{k\geq 1}\{\sigma_k=+\infty\})=1$ (see Lemma~\ref{lem:cv_sigma_k}), 
we get that $\sum_{\ell=1}^{n} Z_k$ converges almost surely, which concludes the proof.
\end{proof}

From now on, for all $C>0$, we set
\[
  \cP_C(E):=\big\{\mu\text{ : $\mu$ is a probability on $E$ such that }\mu\cdot W\leq
  C\big\},
\]
where we recall that $W= V^{\nicefrac1q}$.
Note that $\cP_C(E)$ is a compact subset of $\cP(E)$ (the set of
Borel probability measures on $E$) with respect to the topology of weak convergence.

\begin{lemma}
\label{lem:compacity}
The sequence $(\tilde{\eta}_n)_{n\geq 0}$ is almost surely relatively
compact in $\cP(E)$ with respect to the topology of weak
convergence. More precisely, there exists a random value $C>0$ such
that, almost surely, $\tilde{\eta}_n\in \cP_C(E)$ for all $n\in\N$.
\end{lemma}

\begin{proof}
Using Lemma~\ref{lem:algo_sto}, we have that, for all $n\geq 0$ (recall that $W=V^{\nicefrac1q}$),
\begin{align*}
\tilde\eta_{n+1}\cdot W
=\tilde\eta_n \cdot W
+\gamma_{n+1}\left(U_{n+1}\cdot W + F(\tilde\eta_n)\cdot W\right),
\end{align*}
where
\begin{align*}
F(\tilde\eta_n)\cdot W
&=\tilde\eta_n Q \cdot W - \tilde\eta_n Q(E) \tilde\eta_n\cdot W
\leq \theta \tilde\eta_n\cdot W +K - c_1\,\tilde\eta_n\cdot W,
\end{align*}
where we have used Assumption~(A1) and~(A'2-ii).
Therefore, we get
\begin{align}
\label{eq:recur-rel-weight}
\tilde\eta_{n+1}\cdot W
\leq \tilde\eta_n\cdot W 
+ \gamma_{n+1}\left(U_{n+1}\cdot W + K+ (\theta-c_1)\tilde\eta_n\cdot W\right).
\end{align}
We define the random variable
\begin{align*}
M=\sup_{m\geq n\geq 1}\left|\sum_{k=n}^m \gamma_{n+1}U_{k+1}\cdot W\right|
\end{align*}
which is finite almost surely (by Lemma~\ref{lem:U}).
Let us prove by induction that
\begin{align}
\label{eq:bound-of-eta_n-V-weight}
\tilde\eta_n\cdot W
\leq 2M + \frac{1+c_1-\theta}{c_1-\theta}\hat K,
\end{align}
where $\hat K = \nicefrac{K}{c_1}\vee (\tilde \eta_1\cdot W)$ (note that $\hat K$ is random and that $\hat K\geq \nicefrac K{c_1}\geq K$).
The result is immediate for $n=1$.
Assume now that the result holds true for $n\geq 1$. 
If $\tilde\eta_n\cdot W\leq \frac{\hat K}{c_1-\theta}$, 
then~\eqref{eq:recur-rel-weight} entails that
\begin{align*}
\tilde\eta_{n+1}\cdot W
\leq \frac{\hat K}{c_1-\theta} + M + \hat K
\leq M + \frac{1+c_1-\theta}{c_1-\theta}\hat K,
\end{align*}
because $\gamma_{n+1}\leq \nicefrac1{c_1}$ almost surely by Assumption~(A1).
If $\tilde\eta_n\cdot W>\frac{\hat K}{c_1-\theta}$, 
then we define the (random) integer $n_0$ by
\begin{align*}
n_0=\sup\left\{k\in\{1,\ldots,n\}\text{ such that }\tilde\eta_k\cdot W
>\frac{\hat K}{c_1-\theta}\text{ and }\tilde\eta_{k-1}\cdot W\leq \frac{\hat K}{c_1-\theta}\right\},
\end{align*}
which is well defined since $\tilde\eta_1\cdot W\leq \hat K$ by definition of $\hat K$. 
We can thus deduce as above that
$\tilde\eta_{n_0}\cdot W\leq M+\frac{1+c_1-\theta}{c_1-\theta}\hat K$ 
and hence
\begin{align*}
\tilde\eta_{n+1}\cdot W
\leq \tilde\eta_{n_0}\cdot W
+\sum_{k=n_0}^n \gamma_{k+1}U_{k+1}\cdot W
\leq  M+\frac{1+c_1-\theta}{c_1-\theta}\hat K+M.
\end{align*}
Finally, we deduce by induction that~\eqref{eq:bound-of-eta_n-V-weight} holds true for all $n\geq 1$.

Since the right-hand side of~\eqref{eq:bound-of-eta_n-V-weight} does not depend on $n$ 
and since $W=V^{\nicefrac1q}$ has relatively compact level sets by Assumption~(A'2), 
we deduce that $(\tilde\eta_n)_{n\in\N}$ is almost surely relatively compact 
for the topology of weak convergence on $\cP(E)$ (see for instance \cite[Theorem~6.7,~Chapter~II]{Parthasarathy1967}).
\end{proof}

\begin{lemma}
  \label{lem:flow}
  For any $C\geq \frac{K}{c_1-\theta}$ and
  any $\mu_0\in\cP_C(E)$,
  $t\mapsto \nu_t:=\P_{\mu_0}\left(X_t\in\cdot\mid X_t\neq\d\right)$
  is the unique solution to the dynamical system~\eqref{eq:dyn-sys}
  with values in $\cP_{C}(E)$ and it is continuous with
  respect to $(\mu_0,t)\in \cP_C(E)\times [0,+\infty)$.
\end{lemma}

\begin{proof}
  \noindent\textit{Step 1. Existence.}
  Fix $C>0$ and $\mu_0\in\cP_C(E)$.  We consider the weak
  forward-Kolmogorov equation defined as
  \begin{equation}
    \label{eq:FK-classical}
    \frac{\mathrm d\mu_t\cdot f}{\mathrm dt}=\mu_t(Q-I)\cdot f,
  \end{equation}
  for all bounded continuous functions $f:E\rightarrow\R$. If $\mu_0$
  is a Dirac measure $\delta_x$, then,
  by~\cite[Theorem~2.21]{Chen2004},
  $t\mapsto \P_x\left(X_t\in\cdot\right)$ is a solution of this
  equation. Recall that $W=V^{\nicefrac1q}$; Equation~(2.29)
  in~\cite{Chen2004} states that if there exists a constant $c>0$ such
  that $(Q-I)W\leq cW$, then, for all $x\in E$, for all $s\geq 0$,
  $\mathbb E_x[W(X_s)]\leq W(x)\mathrm e^{cs}$ (here and below, we
  always assume that the considered functions vanish on $\d$, so that
  $\mathbb E_x[W(X_s)]=\mathbb E_x[W(X_s)\1_{X_s\neq \d}]$).  Using
  Assumption~(A'2-iv), we get that $|Q_x W|\leq B^{\nicefrac1q} W$,
  which thus implies that
  \begin{equation}\label{eq:latter_estimate}
  \E_x W(X_s)\leq e^{(B^{1/q}+1)s} W(x)\quad\text{ for all }s\geq 0.
  \end{equation}
  If $\mu_0$ is not a Dirac mass, we get, 
  from Equation~\eqref{eq:latter_estimate} and from Assumption~(A'2-iii), that
  $(s,x)\mapsto \E_x[(Q-I)f(X_s)]$ is integrable with respect to
  $\mathrm d s \,\mu(\mathrm dx)$  on $[0,t]\times E$. Therefore, we can use Fubini's theorem and get that, for all $t\geq 0$,
  \begin{align}
    \label{eqFK}
    \E_{\mu_0}f(X_t)=\mu_0\cdot f + \int_0^t \E_{\mu_0}[(Q-I)f(X_s)]\,\mathrm ds,
  \end{align}
  which means that $t\mapsto \P_{\mu_0}\left(X_t\in\cdot\right)$ is a
  solution of~\eqref{eq:FK-classical}. 
  
  In both cases ($\mu_0$ being a Dirac mass or not), $t\mapsto \P_{\mu_0}\left(X_t\in\cdot\right)$ is a solution of~\eqref{eq:FK-classical}, 
  and, thus, $\nu_t$ is a solution of~\eqref{eq:dyn-sys}. 
  Since, by Assumption~(A1), $\P_{\mu_0}(X_t\in E)\geq e^{-(1-c_1)t}$ for all $t\geq 0$,
  we get that 
  \begin{equation}\label{eq:ineq_nuV}
  \nu_t \cdot W\leq \mathrm e^{(B^{1/q}+2-c_1)t}\nu_0\cdot W\quad\text{ for all }t\geq 0.
  \end{equation}

  \medskip
 \noindent\textit{Step 2. Compactness.} Let us now prove that
 $\nu_t\in\cP_C(E)$ for all $t\geq 0$. { We denote by
   $T_N$  the first hitting time of $\{W\geq N\}$, i.e.
   \[T_N=\inf\{t\geq 0,\ W(X_t)\geq N\}.\] Note that $T_N$ is a
   stopping time for the natural filtration of the process (see for
   instance Theorem~2.4 in~\cite{Bass2010}). Using the fact that
 $(Q-I)\cdot W\leq (\theta-1) W +K$ and Dynkin's formula, we obtain
 that, for all $x\in E$ and all $0{\leq}s<t$,
 \begin{multline}
   \label{eqDynkinLocal}
   \E_x\big[\mathrm e^{(1-c_1) { [}(t-s)\wedge T_N{]}} W(X_{(t-s)\wedge
     T_N})\1_{(t-s)\wedge T_N<\tau_\d}\big]\\
   = W(x)+ \E_x\left[\int_s^{t\wedge (s+T_N)} \mathrm
   e^{(1-c_1) (u-s)}\left((\theta-c_1)
    W(X_{u-s})\1_{u-s<\tau_\d}+K\right)\,\mathrm du\right]
\end{multline}
The same computation with $c_1$ replaced by $\theta$ {and $s=0$} shows that, for
any fixed $t\geq 0$, $\E_x[W(X_{t\wedge T_N})\1_{t\wedge T_N<\tau_\d}]$
is uniformly bounded over $N\geq 1$, so that,
\[
\P_x(T_N\leq t)\leq \E_x\left[\frac{W{(X_{t\wedge T_N})}}{N}\1_{t\wedge T_N<\tau_\d}\right]\xrightarrow[N\rightarrow+\infty]{} 0,
\]
{where we have used Markov's inequality.}
This implies in particular that the almost surely non-decreasing
sequence $(T_N)_{N\geq 0}$ converges to $+\infty$ almost surely.
Using in addition Fatou's Lemma in the left-hand side
of~\eqref{eqDynkinLocal} and the monotone convergence theorem in the right-hand side (separating the $W$ term and the $K$ term and using the fact that $\theta<c_1$ and that $T_N$ is almost surely non-decreasing), we obtain}
 \[
   \E_x\big[\mathrm e^{(1-c_1) (t-s)} W(X_{t-s})\1_{t-s<\tau_\d}\big]
   \leq W(x)+\int_s^t \mathrm e^{(1-c_1) (u-s)}\left((\theta-c_1) \E_x\big[W(X_{u-s})\1_{u-s<\tau_\d}\big]+K\right)\,\mathrm du.
 \]
 Integrating with respect to the law of $X_s$ under $\P_{\mu_0}$ and using Fubini's theorem, we thus get that
 \begin{align*}
&\E_{\mu_0}\big[\mathrm e^{(1-c_1) t} W(X_t)\1_{t<\tau_\d}\big]\\
&\hspace{1cm}
\leq \E_{\mu_0}\big[\mathrm e^{(1-c_1) s}W(X_s)\1_{s<\tau_\d}\big]
+\int_s^t \mathrm e^{(1-c_1) u}\left((\theta-c_1) \E_{\mu_0}\big[W(X_u)\1_{u<\tau_\d}\big]+K\right)\,\mathrm du.
 \end{align*}
 This implies that
 $\E_{\mu_0}\big[\mathrm e^{(1-c_1) t} W(X_t)\1_{t<\tau_\d}\big]
 \leq \mu_0\cdot W \vee
 \frac{K}{c_1-\theta}$ (we detail the proof of this implication in Lemma~\ref{lem:tech} below) and, since $\P_{\mu_0}(t<\tau_\d)\geq e^{-(1-c_1)t}$, that
 $\nu_t\cdot W\leq \nu_0\cdot W\vee \frac{K}{c_1-\theta}$, for all
 $t\geq 0$, i.e.\ that $\nu_t\in\cP_{C\vee \frac{K}{c_1-\theta}}$ for
 all $t\geq 0$.

  \medskip\noindent \textit{Step 3. Weak continuity of the
    semi-group.}  Our aim is to prove the continuity of
  $(\mu_0,t)\mapsto \E_{\mu_0} f(X_t)$ for any bounded continuous
  functions $f:E\rightarrow\R$. We prove first the continuity of the
  application
  \[(x,t)\in E\times [0,+\infty)\mapsto \E_x f(X_t).\] Recall that
  $T_N$ is the first hitting time of $\{W\geq N\}$ and is a stopping
  time for the natural filtration of the process.
  We have, for all $x\in E$ and $t\geq 0$,
  \begin{align*}
    \left|\E_x f(X_t)-\E_x\left[f(X_{t\wedge T_N})\right]\right|
    &\leq 2\|f\|_\infty\, \P_x(T_N<t) 
    \leq {2} \|f\|_\infty\,\E_x\left[\frac{W(X_{t\wedge T_N})}{N}\right]\\
    &\leq {2}\|f\|_\infty\,e^{(B^{1/q}+1)\,t}W(x)/N,
  \end{align*}
  where the last inequality is a consequence of Assumption~(A'2-iv)
  and~\eqref{eq:latter_estimate}.  In particular, since $V$ is locally bounded,
  $(x,t)\mapsto \E_x f(X_t)$ is the locally-uniform limit (when
  $N\rightarrow+\infty$) of
  $(x,t)\mapsto \E_x\big[f(X_{t\wedge T_N})\big]$, which is continuous
  with respect to $(x,t)$ since it is the expectation of a pure jump
  Markov process with uniformly-bounded continuous jump measure. As a
  consequence, the application $(x,t)\mapsto \E_x f(X_t)$ is
  continuous (and bounded).

  Let us now prove that, for any bounded continuous function
  $f:E\rightarrow\R$, the function
  \[(\mu_0,t)\mapsto \E_{\mu_0} f(X_t) \]
  is continuous on $\cP_C(E)\times [0,+\infty)$, for all $C\geq 0$. Let
  $\mu_n\in \cP_C(E)\rightarrow \mu$ and $t_n\rightarrow t$
  when $n\rightarrow+\infty$ (note that $\mu\in \cP_C(E)$ since
  this set is closed for the topology of weak convergence). Then, we have
  \begin{align*}
    \big|\E_{\mu_n} f(X_{t_n}) -\E_\mu f(X_t)\big|
    &\leq  \big|\E_{\mu_n}\big[f(X_{t_n})-f(X_t)\big]\big|
    +\big|\E_{\mu_n} f(X_{t}) - \E_\mu f(X_t)\big|\\
                                            &\to 0 \quad \text{ when }n\to+\infty,
  \end{align*}
  where we used (for the first term in the right-hand side) the almost-sure
  continuity of $s\mapsto X_s$ at time $t$ and the dominated
  convergence theorem, and (for the second term in the right-hand side) 
  the continuity of $x\mapsto \E_x f(X_t)$ and the weak convergence
  of $\mu_n$ toward $\mu$.

  \medskip
  \noindent\textit{Step 4. Uniqueness.} Let $t\mapsto \mu_t$ be a solution
  to~\eqref{eq:dyn-sys} in $\cP_C(E)$ for some $C\geq 0$ and let us
  consider
  \[
    \theta_t:=\exp\left(\int_0^t \mu_s(Q-I)(E)\,\mathrm ds\right)\,\mu_t.
  \]
  By Assumption~(A'2-iii), $|\mu_s(Q-I)(E)|\leq A^{\nicefrac1q}\,\mu_s \cdot W+1\leq A^{\nicefrac1q}C+1$,
  so that $\theta_t$ is well defined for all $t\geq 0$. Moreover, for
  all bounded continuous functions $f:E\rightarrow \R$, $\theta_t\cdot f$ is
  differentiable and we have
  \[
    \frac{\d \theta_t\cdot f}{\d t}=\mu_t(Q-I)(E)\theta_t\cdot f+\theta_t
    Q\cdot f-\mu_tQ(E)\theta_t\cdot f=\theta_t(Q-I)\cdot f.
  \]
  Said differently, $\theta_t$ is solution
  to~\eqref{eq:FK-classical}.
{  Hence, for any continuous function $f$, we have
  \[
    \frac{\mathrm d\E_{\theta_s}f(X_{t-s})}{\mathrm ds}
    =\theta_s(Q-I)\cdot \E_\cdot f(X_{t-s})
    - \theta_s (Q-I)\cdot \E_\cdot f(X_{t-s})=0,
  \]
  where we used~\eqref{eq:FK-classical} for $(\theta_t)_t$ to handle the
  first right-hand-side term (recall that $x\mapsto \E_x f(X_{t-s})$
  is bounded continuous) and the backward Kolmogorov equation for the
  second right-hand-side term (see for instance Theorem~2.21
  in~\cite{Chen2004}). This implies that
  $\theta_t\cdot f=\E_{\theta_0}f(X_t)$ and hence that
  \[
    \mu_0\cdot f=\frac{\theta_t\cdot f}{\theta_t(E)}=\frac{\E_{\mu_0}f(X_t)}{\P_{\mu_0}(X_t\neq \d)}=\nu_t\cdot f
  \]
  for all $t\geq 0$ and all bounded continuous functions
  $f:E\rightarrow\R$.  This implies that $\mu=\nu$, which is thus the
  unique  solution of~\eqref{eq:dyn-sys}.}
\end{proof}

In Step 2 of the proof above, we used the following technical lemma:
\begin{lemma}\label{lem:tech}
  Let $g:[0,+\infty)\rightarrow \R$ and $f:\R\times \R \rightarrow \R$ be two measurable functions such that $t\in[0,+\infty)\mapsto f(t,g(t))\in \R$ is locally integrable. If
    \[g(t)-g(s)\leq \int_s^t f(u,g(u))\,\mathrm du,\quad\forall 0\leq
      s\leq t\] and if there exists $M\in\R$ such that
    $f(u,g(u))\leq 0$ for all $u\in [0,+\infty)$ such that $g(u)\geq M$. Then
    \[g(t)\leq g(0)\vee M,\quad \forall t\geq 0.\]
\end{lemma}

\begin{proof}
  We assume without loss of generality that $M\geq g(0)$ and proceed by
  contradiction: assume that there exist $\varepsilon>0$ and
  $t\geq 0$ such that $g(t)\geq M+\varepsilon$ and let
  $t_0=\inf\{t\geq 0\text{ s.t. }g(t)\geq M+\varepsilon\}$.
  Note that, for all $t\geq t_0$,
  \[g(t)\leq g(t_0)+\int_{t_0}^t f(u,g(u))\,\mathrm du\xrightarrow[t\downarrow t_0]{} g(t_0),\]
  and hence $g(t_0)\geq M+\varepsilon$. Now, let $s_0=\sup\{s\leq t_0\text{ s.t. }g(s)\leq M\}$, and note that
    \[g(s_0)\leq \liminf_{s\uparrow s_0} \,\left\{g(s) + \int_s^{s_0} f(u,g(u))\,\mathrm du\right\}=\liminf_{s\uparrow s_0} g(s),\]
implying that $g(s_0)\leq M$.  Finally, since $g(s)\in [M,M+\varepsilon]$ for all $s\in [s_0,t_0]$,  we have
    \[
      M+\varepsilon\leq g(t_0)\leq g(s_0) +\int_{s_0}^{t_0} f(u,g(u))\,\mathrm du\leq g(s_0)\leq M.\qedhere
    \]
\end{proof}

We are now ready to prove Proposition~\ref{prop:eta}:
\begin{proof}[Proof of Proposition~\ref{prop:eta}]

  Our approach is based on \cite{Benaim1999} (see also
  \cite{BenaimLedouxEtAl2002} for an application of this theorem on a
  set of probability measures on a compact space).  In view
  of~\cite[Lemma~3.1]{Varadarajan1958}, since $E$ is separable by
  assumption, there exists a metrization of the topology of~$E$ such
  that~$E$ is totally bounded (this distance is imposed on~$E$ from
  now on). Also, still by \cite[Lemma~3.1]{Varadarajan1958}, there
  exists a family of bounded uniformly continuous
  functions~$(g_k)_{k\geq 1}$ that is dense in
  $U(E,\R)$, the set of all bounded uniformly-continuous 
  functions from~$E$ to~$\R$.
  Finally, \cite[Lemma~3.1]{Varadarajan1958} also states that a sequence $(\mu_n)_{n\in\N}$ of non-negative measures converges weakly to~$\mu$ if and only if
  $\mu_n\cdot g_k\rightarrow \mu\cdot g_k$ when $n\rightarrow+\infty$,
  for all $k\in\N$. 
  We also consider the function
  $g_0:x\in E \mapsto Q_x(E)$, which is continuous by Assumption~(A4)
  and bounded by Assumption~(A1), and the family of functions indexed
  by $k,M\in\N$ defined by
  \begin{align*}
    g_k^M(x)=-M\vee \left(Q\cdot g_k(x) \wedge M\right)
  \end{align*}
  and which are continuous (by Assumption~(A4)) and bounded. In
  particular, the distance
  \begin{align*}
    d(\mu_1,\mu_2)&=|\mu_1 Q(E)-\mu_2 Q(E)|+\sum_{k=1}^\infty \frac{|\mu_1\cdot g_k-\mu_2\cdot g_k|\wedge 1}{2^k(1+\|g_k\|_\infty)} + \sum_{k=1, M=1}^\infty \frac{|\mu_1\cdot g^M_k-\mu_2\cdot g^M_k|\wedge 1}{2^{k+M}(1+\|g^M_k\|_\infty)}
  \end{align*}
  is a metric for the weak convergence in the set of non-negative
  measures on $E$.

  We introduce the increasing sequence $(\tau_n)_{n\geq 1}$ defined as
\begin{align*}
\tau_n=\gamma_1+\gamma_2+\cdots+\gamma_n,
\end{align*}
(see Lemma~\ref{lem:algo_sto} for the definition of $\gamma_n$) 
and we consider the time-changed and linearized versions
$(\bar\mu_t)_{t\in[1,+\infty)}$ and $(\mu_t)_{t\in[1,+\infty)}$ of
$(\tilde \eta_n)_{n\in\N}$ defined, for all $n\geq 1$ and all
$t\in \left[\tau_n,\tau_{n+1}\right]$, by
\begin{align*}
    \bar\mu_t =\tilde\eta_n\quad\text{and}\quad \mu_t=\tilde \eta_n+\frac{t-\tau_n}{\tau_{n+1}-\tau_n}(\tilde \eta_{n+1}-\tilde \eta_n).
\end{align*}
Similarly, we define $\bar{U}_t=U_{n+1}$ for all $t\in\left[\tau_n,\tau_{n+1}\right]$ (see Lemma~\ref{lem:algo_sto} for the definition of $U_n$).

To prove that $(\mu_t)_{t\geq 0}$ is an asymptotic
pseudo-trajectory of the semi-flow induced by~\eqref{eq:dyn-sys},
we apply~\cite[Theorem~3.2]{Benaim1999} (and refer the reader
to~\cite{Benaim1999} for the definition of an asymptotic
pseudo-trajectory).

Note that $\mu_t\in\cP_C(E)$ for all $t\geq 0$, and hence
$(\mu_t)_{t\geq 0}$ has compact closure in $\cP_C(E)$ (since this set
is itself compact). Also, by construction, $t\mapsto \mu_t$ is
uniformly continuous (and even Lipschitz) with respect to the distance $d$ on
$\cP_C(E)$. { Indeed, for all $s,t\in[\tau_n,\tau_{n+1}]$,
  \begin{align*}
    d(\mu_s,\mu_t)&=\frac{t-s}{\tau_{n+1}-\tau_n}d(\tilde\eta_{n+1},\tilde\eta_n)=\frac{t-s}{\gamma_{n+1}}d(\tilde\eta_{n+1},\tilde\eta_n)\leq (t-s)\,(2\|Q(E)\|_\infty+4),
  \end{align*}
  where we have used the fact (see Lemma~\ref{lem:algo_sto}) that, for all
  bounded measurable function $g:E\rightarrow\R_+$,
  \[
    \left|\frac{\tilde\eta_{n+1}\cdot g-\tilde\eta_n\cdot
        g}{\gamma_{n+1}}\right|
    =\left|\tilde\eta_n
      Q(E)(g(Y_{n+1})-\tilde\eta_n\cdot g)\right|\leq 2\|g\|_\infty.
      \]
} Therefore, to apply \cite[Theorem~3.2]{Benaim1999}, it only remains
to prove that all limit points of $(\Theta_t(\mu))_{t\geq 0}$ in
$C(\R_+,\cP_C(E))$ endowed with the topology of uniform convergence on
compact sets are solutions of~\eqref{eq:dyn-sys}, where
$ \Theta_t(\mu):=(\mu_{t+s})_{s\geq 0}$.  Let
$\mu^{\infty}\in C(\R_+,\cP_C(E))$ be such a limit point: in other
words, we assume that there exists an increasing sequence of positive
numbers $(t_n)_{n\geq 0}$ converging to $+\infty$ such that
$(\Theta_{t_n}(\mu))_{n\geq 0}$ converges to $\mu^\infty$ in
$C(\R_+,\cP_C(E))$.

For all $t\in[\tau_n,\tau_{n+1})$ and all $s\geq 0$ such that $t+s\in[\tau_m,\tau_{m+1})$, we deduce from Lemma~\ref{lem:algo_sto} that
\begin{align}
&\int_t^{t+s} F(\bar\mu_u)+\bar U_u\,\mathrm du \\
&\hspace{1cm}= (\tau_{n+1}-t)(F(\tilde\eta_n)+U_{n+1})+\sum_{k=n+1}^{m-1} \gamma_{k+1} (F(\tilde\eta_k) +U_{k+1})+(t+s-\tau_m)(F(\tilde\eta_m)+U_{m+1})\nonumber\\
&\hspace{1cm}= \frac{\tau_{n+1}-t}{\tau_{n+1}-\tau_n}(\tilde\eta_{n+1}-\tilde\eta_n)+\tilde\eta_m-\tilde\eta_{n+1}+\frac{t+s-\tau_m}{\tau_{m+1}-\tau_m}(\tilde\eta_{m+1}-\tilde\eta_m)\nonumber\\
&\hspace{1cm}=-\mu_t+\mu_{t+s}\label{eq:usefulBenaim}
\end{align}
For all $k\in\N$, we define
$L_F^k:C(\R_+,\cP_C(E))\rightarrow \R^{[0,+\infty)}$ by
\begin{align*}
  L_F^k(\nu)(t)=\nu_0+\int_0^t F(\nu_s)\cdot g_k\,\mathrm ds,
\end{align*}
for any $\nu\in C(\R_+,\cP_C(E))$ (see Lemma~\ref{lem:algo_sto} for the definition of the function~$F$), so that, by Equation~\eqref{eq:usefulBenaim},
\begin{equation}\label{eq:Theta_t}
 \Theta_t(\mu)\cdot g_k=L_F^k\big(\Theta_t(\mu)\big)+A^k_t+B^k_t,
\end{equation}
where, for all $s\geq 0$,
\begin{align*}
A^k_t(s)=\int_t^{t+s} F(\bar\mu_u)\cdot g_k-F(\mu_u)\cdot g_k\,\mathrm du\quad\text{and}\quad B^k_t(s)=\int_t^{t+s} \bar{U}_u\cdot g_k\,\mathrm du.
\end{align*}

The rest of the proof is divided into four steps: 
The first two steps are devoted to prove that $A^k_t$ and, respectively,~$B^k_t$ converge uniformly to~$0$ on compact sets when $t\rightarrow+\infty$.
In the third step, we prove that $L_F^k(\Theta_{t_n}(\mu))$ converges to $L_F^k(\mu^\infty)$ for all subsequence $t_n\to+\infty$ such that
$(\Theta_{t_n}(\mu))_{n\geq 0}$ converges to $\mu^\infty$ in
$C(\R_+,\cP_C(E))$. 
Finally, in the fourth step, we conclude the proof of Proposition~\ref{prop:eta}.

\medskip
\textit{Step 1: $A_t^k$ converges to $0$.}
 For all
$u\in [\tau_n,\tau_{n+1})$, we have
\begin{align*}
  &\big|F(\bar\mu_u)\cdot g_k-F(\mu_u)\cdot g_k\big|\\
  &\hspace{.5cm}\leq\big |\bar\mu_u Q\cdot g_k-\mu_uQ\cdot g_k\big|
    +\big|\bar\mu_u Q(E)\bar\mu_u\cdot g_k-\mu_uQ(E)\mu_u\cdot g_k\big|\\
  &\hspace{.5cm}\leq  \big|\tilde\eta_{n+1} Q\cdot g_k - \tilde\eta_n Q\cdot g_k\big|
    +\|g_k\|_\infty \big|\bar\mu_u Q(E)-\mu_u Q(E)\big|
    + \big|\bar\mu_u\cdot g_k-\mu_u\cdot g_k\big|\\
  &\hspace{.5cm}\leq  \big|\tilde\eta_{n+1} Q\cdot g_k - \tilde\eta_n Q\cdot g_k\big|
    +\|g_k\|_\infty   \big|\tilde\eta_{n+1} Q(E) -\tilde\eta_n Q(E)\big|
    +\big|\tilde\eta_{n+1}\cdot g_k-\tilde\eta_n\cdot g_k\big|\\
  &\hspace{.5cm}\leq \frac{1}{n+1}\left|Q_{Y_{n+1}}\cdot g_k-\tilde\eta_n Q\cdot g_k\right|+\frac{\|g_k\|_\infty}{n+1}\big|Q_{Y_{n+1}}(E)-\tilde\eta_n Q(E)\big|+\frac{1}{n+1}\big|g_k(Y_{n+1})-\tilde\eta_n\cdot g_k\big|\\
  &\hspace{.5cm}\leq \frac{\|g_k\|_\infty}{n+1}(B^{\nicefrac1q}V(Y_{n+1})^{\nicefrac1q}+B^{\nicefrac1q}C+1+2) 
\end{align*}
where we used Assumptions~(A'2-iii) and (A1) and the fact that,
almost surely, $\eta_m\in\cP_C(E)$ for all $n\geq 0$. Hence, if we denote by 
$n_t$ the unique integer such that
$t\in\ [\tau_{n_t},\tau_{n_t+1})$,
for any $t\geq 0$ (such an integer exists since
$\tau_n\rightarrow+\infty$ when $n\rightarrow+\infty$), 
we have, for all $s\geq 0$,
\begin{align*}
  A_t^k(s)&\leq \frac{\|g\|_k}{n_t+1} \sum_{k=n_t}^{n_{t+s}} \gamma_{k+1} B^{\nicefrac1q}V(Y_{k+1})^{\nicefrac1q}+ \frac{\|g\|_k(B^{\nicefrac1q}C+3)s}{n_t+1}\\
          &\leq \frac{\|g\|_k}{n_t+1} \frac{B^{\nicefrac1q}}{c_1} \tilde\eta_{n_{t+s}+1}\cdot V^{\nicefrac1q}+ \frac{\|g\|_k(B^{\nicefrac1q}C+3)s}{n_t+1},
\end{align*}
where we used that  $\gamma_n \leq 1/(c_1 n)$, for all
$n\geq 1$, by Assumption~(A1). Finally, for all $T\geq 0$, we have
\begin{align*}
 \sup_{s\in[0,T]} |A^k_t(s)|\leq \frac{T\,\|g_k\|_\infty \big(B^{\nicefrac1q} C+B^{\nicefrac1q} C/c_1 + 3\big)}{n_t+1}\to 0 \quad\text{ when }t\to+\infty.
\end{align*}

\medskip
\textit{Step 2: $B_t^k$ converges to $0$.}
We have, for all $t\in[\tau_n,\tau_{n+1})$ and $t+s\in [\tau_{n+m},\tau_{n+m+1})$, 
\begin{align*}
  |B^k_t(s)|&\leq (\tau_{n+1}-t) |U_{n+1}\cdot g_k| + \left|\sum_{\ell=n+1}^{n+m-1} \gamma_{\ell+1} U_{\ell+1} \cdot g_k\right| + (s-\tau_{n+m}) \big|U_{n+m+1}\cdot g_k\big|\\
  &\leq \gamma_{n+1} \big|U_{n+1}\cdot g_k\big| + \left|\sum_{\ell=n+1}^{n+m-1} \gamma_{\ell+1} U_{\ell+1} \cdot g_k\right| + \gamma_{n+m+1} \big|U_{n+m+1}\cdot g_k\big|.
\end{align*}
Using a similar approach as in the proof of Lemma~\ref{lem:U}, one
easily obtains that, for any bounded continuous function
$f:E\rightarrow \R$, $\sum_{\ell=0}^n \gamma_{\ell+1} U_{\ell+1}\cdot f$
converges almost surely when $n\rightarrow +\infty$. Hence, we have that, almost surely,
\begin{align*}
\lim_{n\rightarrow+\infty} \sup_{m\geq 1} \left\{\gamma_{n+1} |U_{n+1}\cdot g_k|+\left|\sum_{\ell=n+1}^{n+m-1} \gamma_{\ell+1}U_{\ell+1}\cdot g_k\right|+\gamma_{n+m+1} |U_{n+m+1}\cdot g_k|\right\}=0.
\end{align*}
In particular, we have that, for all $T\geq 0$,
\begin{align*}
 \sup_{s\in[0,T]} \big|B^k_t(s)\big|\to 0 \quad\text{ when }t\to+\infty.
\end{align*}

\medskip \textit{Step 3: $L_F^k(\Theta_{t_n}(\mu))$ converges to
  $L_F^k(\mu^\infty)$ for all subsequence $t_n\to+\infty$ such that
$(\Theta_{t_n}(\mu))_{n\geq 0}$ converges to $\mu^\infty$ in
$C(\R_+,\cP_C(E))$.} To prove this, it is enough to show that $L_F^k$ is sequentially continuous in $C\big(\R_+,\cP_C(E)\big)$. 
  Let $(\nu^n)_{n\geq 0}$ be a sequence of elements of $C\big(\R_+,\cP_C(E)\big)$
which converges to $\nu\in C\big(\R_+,\cP_C(E)\big)$. For all $n\geq 0$ and all
$t\geq 0$, we have
\begin{align}
  \label{eq:continuity-lfk}
  \left|L_F^k(\nu^n)(t)-L_F^k(\nu)(t)\right|
  \leq \left|\nu^n_0\cdot g_k - \nu_0\cdot g_k\right|
  + \int_0^t |F(\nu^n_s)\cdot g_k -F(\nu_s)\cdot g_k|\,\mathrm ds.
\end{align}
The first term of the right-hand side converges to~$0$ because of the
weak convergence of $(\nu^n_0)_{n\geq 0}$ to $\nu$. Let us now focus
on the second term of the right-hand side; we have
\begin{align*}
  \big|F(\nu^n_s)\cdot g_k -F(\nu_s)\cdot g_k\big|
  \leq \big|\nu^n_s Q\cdot g_k-\nu_s Q\cdot g_k\big|
  +\big|\nu^n_s Q(E)\nu^n_s\cdot g_k-\nu_sQ(E)\nu_s g_k\big|.
\end{align*}
Since $\nu^n$ converges uniformly on compact sets toward $\nu$, we
deduce that the term
$s\mapsto \big|\nu^n_s Q(E)\nu^n_s\cdot g_k-\nu_sQ(E)\nu_s g_k\big|$ converges
uniformly to~$0$ on compact sets when $n\to+\infty$ (we use
here the fact that $g_0=Q_\cdot(E)$ appears in the distance
$d$). Moreover, since $\nu^n_s\in \cP_C(E)$ and since
$|Q\cdot g_k|\leq B^{\nicefrac1q'}\|g_k\|_\infty W^{\nicefrac q{q'}}$ by
Assumption~(A2-iii) (recall that $W:=V^{\nicefrac1q}$), 
we deduce that, for all $M\geq 1$,
\begin{align*}
  \big|\nu^n_s Q\cdot g_k-\nu_s Q\cdot g_k\big|
  & \leq \big|(\nu^n_s-\nu_s) g^M_k\big| 
  + (\nu^n_s+\nu_s)\big|Q\cdot g_k - g^M_k\big|\\
  & \leq \big|(\nu^n_s-\nu_s) g^M_k\big| 
  + (\nu^n_s+\nu_s)\big|Q\cdot g_k \1_{|Q\cdot g_k|>M}\big|\\
  & \leq \big|(\nu^n_s-\nu_s) g^M_k\big| 
  + B^{1/q'} \|g_k\|_\infty (\nu^n_s+\nu_s)\big|W^{q/q'} \1_{B^{1/q}\|g_k\|_\infty^{q'/q} W > M^{q'/q}}\big|\\
  & \leq \big|(\nu^n_s-\nu_s) g^M_k\big| 
  + \frac{B^{1/q} \|g_k\|^{q'/q}_\infty}{M^{q'/q-1}} (\nu^n_s+\nu_s)(W)\\
  & \leq |(\nu^n_s-\nu_s) g^M_k| + \frac{B^{1/q} \|g_k\|^{q'/q}_\infty\,2C}{M^{q'/q-1}},
\end{align*}
where we have used the fact that $\nu^n_s\in\cP_C(E)$ for all
$n\in \N$ and all $s\geq 0$. The term
$\frac{B^{1/q} \|g_k\|^{q'/q}_\infty\,2C}{M^{q'/q-1}}$ goes to $0$ when
$M\rightarrow+\infty$ uniformly in $s\geq 0$ and the term
$|(\nu^n_s-\nu_s) g^M_k|$ converges to $0$ uniformly in $s$ in compact
sets. As a consequence, we deduce that
$|\nu^n_s Q\cdot g_k-\nu_s Q\cdot g_k|$ converges to $0$ uniformly in
$s$ in compact sets. This allows us to conclude that the second term
of the right hand side of~\eqref{eq:continuity-lfk} converges to $0$
when $n\rightarrow+\infty$, which was the aim of Step~3.

\medskip\textit{Step 4: conclusion.} Steps 1 to 3 above entail that
any limit point $\mu^\infty$ of $(\Theta_t(\mu))_{t\geq 0}$ satisfies
\[\mu^\infty_t\cdot g_k = \mu^\infty_0\cdot g_k 
+ \int_0^t F(\mu^\infty_s)\cdot g_k\,\mathrm ds\quad (\forall k\geq 1).\]
Since $(g_k)_{k\geq 1}$ is dense in the set $U(E,\R)$, we conclude
(see for instance \cite[Lemma~2.3]{Varadarajan1958}) that
\[\mu^\infty_t = \mu^\infty_0 + \int_0^t F(\mu^\infty_s)\,\mathrm ds.\]
As a consequence, $\mu^\infty$ is solution to the dynamical
system~\eqref{eq:dyn-sys}. Using \cite[Theorem~3.2]{Benaim1999}, we
deduce that $(\mu_t)_{t\geq 0}$ is a pseudo asymptotic trajectory in
$\cP_C(E)$ for the semi-flow induced by the well-posed dynamical
system~\eqref{eq:dyn-sys} in $\cP_C(E)$. Therefore,
Assumption~(A3)
entails that the set of limit points of $(\mu_t)_{t\geq 0}$ is
included in the uniformly attracting set $\{\nu\}$ of the semi-flow generated
by~\eqref{eq:dyn-sys}. In particular, the only limit point of the
compact sequence $(\tilde \eta_n)_{n\geq 1}$ is $\nu$. This concludes
the proof of Proposition~\ref{prop:eta}.
\end{proof}

\begin{remark}
  Without Assumption~(A3), we still get that $(\mu_t)_{t\geq 0}$ is a
  pseudo asymptotic trajectory in $\cP_C(E)$ for the semi-flow induced
  by the well-posed dynamical system~\eqref{eq:dyn-sys} in
  $\cP_C(E)$. In particular, the set of limit points of
  $(\mu_t)_{t\geq 0}$ is included in the limit sets of the flow
  (see~\cite[Section~5.2]{Benaim1999}).
\end{remark}

\subsection{Proof of Theorem~\ref{thm:unbal-with-weights} from
  Proposition~\ref{prop:eta}}
  Fix $c'\in(\theta,c_1)$. For all $k\geq 1$, we define
  \[
    \sigma_{k}:=\inf\big\{n\geq k,\ m_nP(E)< c'n\big\}
  \]

  For all $n\geq 1$ and any bounded continuous function $f:E\rightarrow\R$, we set
  $\Psi_n=m_{n\wedge\sigma_k} \cdot f -\eta_{n\wedge\sigma_k} R \cdot f$, so
  that $(\Psi_n)_{n\geq 1}$ is a martingale and
  \[
    \Psi_{n} = m_0 \cdot f + \sum_{i=1}^{n\wedge\sigma_k} \big( {R}^{\sss (i)}_{Y_i}\cdot f - R_{Y_i}\cdot f\big).
  \]
  An immediate adaptation of Theorem~1.3.17 in~\cite{Duflo} tells us
  that if the sequence
  $(n^{-1} \mathbb E \left[\left|\Psi_n\right|^r\right])_{n\geq
    1}$ is bounded, then $n^{-1}\Psi_n$ goes
  almost surely to zero when $n$ goes to infinity.  We have, using
  Lemma~1 in~\cite{Chatterji1969},
  \begin{align*}
    \frac{\E\left[|\Psi_n|^r\right]}{n}
    &\leq \frac{2(m_0\cdot f)^r}{n} 
      + \frac{2}{n} \sum_{i=1}^n \mathbb E\big[\big| {R}^{\sss (i)}_{ {Y_i}} \cdot f - R_{  {Y_i}}\cdot f\big|^r\1_{i\leq \sigma_k}\big]\\
    &\leq \frac{2(m_0\cdot f)^r}{n} 
      + \frac{2\|f\|_\infty^r}{n} \sum_{i=1}^n A\,\mathbb E\big[V(Y_i)\1_{i\leq\sigma_k}\big]
  \end{align*}
  where we used the fact that $\1_{i\leq\sigma_k}$ is $\mathcal F_{i-1}$-measurable and independent of $Y_i$ and  Assumption~(A'2-iii).
  
  Using Lemma~\ref{lem:C}, we deduce that the sequence $\big(n^{-1}\E\left[|\Psi_n|^r\right]\big)_n$ 
  is uniformly bounded and hence that $n^{-1}\Psi_n$ goes almost surely to zero when $n$ goes to infinity 
  (since we have assumed, in particular, that $m_0\cdot V < +\infty$, which entails $m_0(E)<\infty$).

  Since this is true for any $k\geq 1$ and since
  $\P(\cup_{k=1}^\infty \{\sigma_k=+\infty\})=1$ (see
  Lemma~\ref{lem:cv_sigma_k}), we deduce that, almost surely,
  $ {m}_n(f) = {\eta}_n R (f) + o(n)$ when $n$ goes to infinity. In
  view of Proposition~\ref{prop:eta}, and by Assumption (A4) (namely
  continuity of $R$), we get that~$( {\eta}_n R\cdot f/n)_{n\geq 1}$
  and~$( {\eta}_n R(E)/n)_{n\geq 1}$ converge almost surely
  to~$\nu R\cdot f$ and~$\nu R(E)$ respectively, which concludes the
  proof of the first part and the last part of
  Theorem~\ref{thm:unbal-with-weights}.

\bigskip
   To get the almost-sure boundedness of $m_n P\cdot V^{\nicefrac1q}/n$, 
  recall that, by definition, $m_n = m_0 + \sum_{i=1}^n R^{\sss (i)}_{Y_i}$, 
  implying that, for all $n\geq 0$,
  \[m_nP\cdot V^{\nicefrac1q} 
  = m_0P\cdot V^{\nicefrac1q} + \sum_{i=1}^n Q^{\sss (i)}_{Y_i}\cdot V^{\nicefrac1q}.\]
  As above, we let
  \[\Phi_n = m_0P\cdot V^{\nicefrac1q} + \sum_{i=1}^{n\wedge \sigma_k} 
  \big(Q^{\sss (i)}_{Y_i}\cdot V^{\nicefrac1q}-Q_{Y_i}\cdot V^{\nicefrac1q}\big).\]
  The sequence $(\Phi_n)_{n\geq 0}$ is a martingale, and, similarly as above, we get that
  \begin{align*}
  \frac{\mathbb E |\Phi_n|^r}{n}
  &\leq \frac{2|m_0P\cdot V^{\nicefrac1q}|^r}{n} + \frac2n\sum_{i=1}^n 
  \mathbb E\big[\big|Q^{\sss (i)}_{Y_i}\cdot V^{\nicefrac1q}-Q_{Y_i}\cdot V^{\nicefrac1q}\big|^r\1_{i\leq \sigma_k}\big]\\
  &\leq \frac{2|m_0P\cdot V^{\nicefrac1q}|^r}{n} + \frac{2B}n\sum_{i=1}^n 
  \mathbb E\big[V(Y_i)\1_{i\leq \sigma_k}\big].
  \end{align*}
  Using Lemma~\ref{lem:C}, we imply that $(\mathbb E |\Phi_n|^r/n)_{n\geq 0}$ is uniformly bounded, and thus that $\Phi_n/n$ converges almost surely to~$0$ when $n\to\infty$. Therefore, we have that, almost surely when $n\to\infty$,
  \[\frac{m_nP\cdot V^{\nicefrac1q}}{n} = \frac1n\sum_{i=1}^n Q_{Y_i}\cdot V^{\nicefrac1q} + o(1)= \tilde \eta_n Q\cdot V^{\nicefrac1q} + o(1).\]
  Note that, by Assumption~(A'2-iv), we have
  \[|\tilde \eta_n Q\cdot V^{\nicefrac1q}| \leq B^{\nicefrac1q}\tilde \eta_n \cdot V^{\nicefrac1q},\]
  and recall that, by Equation~\eqref{eq:bound-of-eta_n-V-weight}, $\tilde \eta_n \cdot V^{\nicefrac1q}$ is almost surely uniformly bounded. We can thus conclude that $m_n P\cdot V^{\nicefrac1q}/n$ is almost surely uniformly bounded, as claimed.

  \section*{Acknowledgment}
  The authors would like to thank Michel Bena\"im, Pascal Maillard and Andi Q.\ Wang for their useful
  comments and suggestions on this paper. CM is grateful to EPSRC for support through the fellowship EP/R022186/1.

\bibliographystyle{abbrv}
\bibliography{biblio-PU}

\begin{thebibliography}{10}

\bibitem{Aldous91}
D.~Aldous.
\newblock Asymptotic fringe distributions for general families of random trees.
\newblock {\em The Annals of Applied Probability}, pages 228--266, 1991.

\bibitem{AldousFlanneryEtAl1988}
D.~Aldous, B.~Flannery, and J.~L. Palacios.
\newblock Two applications of urn processes the fringe analysis of search trees
  and the simulation of quasi-stationary distributions of markov chains.
\newblock {\em Probability in the engineering and informational sciences},
  2(3):293--307, 1988.

\bibitem{AK}
K.~B. Athreya and S.~Karlin.
\newblock Embedding of urn schemes into continuous time {M}arkov branching
  processes and related limit theorems.
\newblock {\em Annals of Mathematical Statistics}, 39:1801--1817, 1968.

\bibitem{BJT}
A.~Bandyopadhyay, S.~Janson, and D.~Thacker.
\newblock Strong convergence of infinite color balanced urns under uniform
  ergodicity.
\newblock {\em ArXiv:1904.06144}, 2019.

\bibitem{BT++}
A.~{Bandyopadhyay} and D.~{Thacker}.
\newblock {A New Approach to {P}\'olya Urn Schemes and Its Infinite Color
  Generalization}.
\newblock {\em Arxiv:1606.05317}, June 2016.

\bibitem{Bass2010}
R.~Bass.
\newblock The measurability of hitting times.
\newblock {\em Electron. Commun. Probab.}, 15:99--105, 2010.

\bibitem{Benaim1999}
M.~Bena\"im.
\newblock Dynamics of stochastic approximation algorithms.
\newblock In {\em S\'eminaire de {P}robabilit\'es, {XXXIII}}, volume 1709 of
  {\em Lecture Notes in Mathematics}, pages 1--68. Springer, Berlin, 1999.

\bibitem{BCV}
M.~Bena{\"\i}m, N.~Champagnat, and D.~Villemonais.
\newblock Stochastic approximation of quasi-stationary distributions for
  diffusion processes in a bounded domain.
\newblock {\em ArXiv:1904.08620}, 2019.

\bibitem{BC15}
M.~Bena{\"\i}m and B.~Cloez.
\newblock A stochastic approximation approach to quasi-stationary distributions
  on finite spaces.
\newblock {\em Electronic Communications in Probability}, 20, 2015.

\bibitem{BCP++}
M.~Benaim, B.~Cloez, and F.~Panloup.
\newblock Stochastic approximation of quasi-stationary distributions on compact
  spaces and applications.
\newblock {\em The Annals of Applied Probability}, 28(4):2370--2416, 2018.

\bibitem{BenaimLedouxEtAl2002}
M.~Bena\"im, M.~Ledoux, and O.~Raimond.
\newblock Self-interacting diffusions.
\newblock {\em Probability Theory and Related Fields}, 122(1):1--41, 2002.

\bibitem{BFS92}
F.~Bergeron, P.~Flajolet, and B.~Salvy.
\newblock Varieties of increasing trees.
\newblock In {\em Colloquium on Trees in Algebra and Programming}, pages
  24--48. Springer, 1992.

\bibitem{BlanchetGlynnEtAl2016}
J.~Blanchet, P.~Glynn, and S.~Zheng.
\newblock Analysis of a stochastic approximation algorithm for computing
  quasi-stationary distributions.
\newblock {\em Advances in Applied Probability}, 48(3):792--811, 2016.

\bibitem{BGZ16}
J.~Blanchet, P.~Glynn, and S.~Zheng.
\newblock Analysis of a stochastic approximation algorithm for computing
  quasi-stationary distributions.
\newblock {\em Advances in Applied Probability}, 48(3):792--811, 2016.

\bibitem{Bona14}
M.~B{\'o}na.
\newblock k-protected vertices in binary search trees.
\newblock {\em Advances in Applied Mathematics}, 53:1--11, 2014.

\bibitem{CattiauxColletEtAl2009}
P.~Cattiaux, P.~Collet, A.~Lambert, S.~Mart{\'{\i}}nez, S.~M{\'e}l{\'e}ard, and
  J.~San~Mart{\'{\i}}n.
\newblock Quasi-stationary distributions and diffusion models in population
  dynamics.
\newblock {\em Ann. Probab.}, 37(5):1926--1969, 2009.

\bibitem{ChampagnatVillemonais2017}
N.~{Champagnat} and D.~{Villemonais}.
\newblock {General criteria for the study of quasi-stationarity}.
\newblock {\em Arxiv:1712.08092}, Dec. 2017.

\bibitem{ChampagnatVillemonais2016}
N.~Champagnat, D.~Villemonais, et~al.
\newblock Uniform convergence to the $ q $-process.
\newblock {\em Electronic Communications in Probability}, 22, 2017.

\bibitem{Chatterji1969}
S.~D. Chatterji.
\newblock An {$L^{p}$}-convergence theorem.
\newblock {\em Annals of Mathematical Statistics}, 40:1068--1070, 1969.

\bibitem{Chen2004}
M.-F. Chen.
\newblock {\em From {M}arkov chains to non-equilibrium particle systems}.
\newblock World Scientific Publishing Co., Inc., River Edge, NJ, second
  edition, 2004.

\bibitem{CS08}
G.-S. Cheon and L.~W. Shapiro.
\newblock Protected points in ordered trees.
\newblock {\em Applied Mathematics Letters}, 21(5):516--520, 2008.

\bibitem{ColletMartinezEtAl2013}
P.~Collet, S.~Mart\'inez, and J.~San~Mart\'in.
\newblock {\em Quasi-stationary distributions}.
\newblock Probability and its Applications (New York). Springer, Heidelberg,
  2013.
\newblock Markov chains, diffusions and dynamical systems.

\bibitem{DarrochSeneta1967}
J.~N. Darroch and E.~Seneta.
\newblock On quasi-stationary distributions in absorbing continuous-time finite
  {M}arkov chains.
\newblock {\em J. Appl. Probab.}, 4:192--196, 1967.

\bibitem{DeshayesRoll2017}
A.~Deshayes and L.~T. Rolla.
\newblock Scaling limit of subcritical contact process.
\newblock {\em Stochastic Processes and their Applications}, 127(8):2630 --
  2649, 2017.

\bibitem{DJ14}
L.~Devroye and S.~Janson.
\newblock Protected nodes and fringe subtrees in some random trees.
\newblock {\em Electronic Communications in Probability}, 19, 2014.

\bibitem{DharmadhikariFabianEtAl1968}
S.~W. Dharmadhikari, V.~Fabian, and K.~Jogdeo.
\newblock Bounds on the moments of martingales.
\newblock {\em Annals of Mathematical Statistics}, 39:1719--1723, 1968.

\bibitem{DiLelievreEtAl2016}
G.~{Di Ges{\`u}}, T.~{Leli{\`e}vre}, D.~{Le Peutrec}, and B.~{Nectoux}.
\newblock {Jump Markov models and transition state theory: the quasi-stationary
  distribution approach}.
\newblock {\em Faraday Discussions}, 195:469--495, 2016.

\bibitem{Duflo}
M.~Duflo.
\newblock {\em Random iterative models}, volume~34 of {\em Applications of
  Mathematics (New York)}.
\newblock Springer-Verlag, Berlin, 1997.
\newblock Translated from the 1990 French original by Stephen S. Wilson and
  revised by the author.

\bibitem{EP23}
F.~Eggenberger and G.~P{\'o}lya.
\newblock {\"U}ber die statistik verketetter vorg{\"a}ge.
\newblock {\em Zeitschrift f{\"u}r Angewandte Mathematik und Mechanik},
  1:279--289, 1923.

\bibitem{Feller1940}
W.~Feller.
\newblock On the integro-differential equations of purely discontinuous
  {M}arkoff processes.
\newblock {\em Trans. Amer. Math. Soc.}, 48:488--515, 1940.

\bibitem{FerrariKestenEtAl1996}
P.~A. Ferrari, H.~Kesten, and S.~Mart{\'{\i}}nez.
\newblock {$R$}-positivity, quasi-stationary distributions and ratio limit
  theorems for a class of probabilistic automata.
\newblock {\em The Annals of Applied Probability}, 6(2):577--616, 1996.

\bibitem{FerrariMaric2007}
P.~A. Ferrari and N.~Mari{\'c}.
\newblock Quasi stationary distributions and {F}leming-{V}iot processes in
  countable spaces.
\newblock {\em Electron. J. Probab.}, 12:no. 24, 684--702 (electronic), 2007.

\bibitem{Gosselin2001}
F.~Gosselin.
\newblock Asymptotic behavior of absorbing {M}arkov chains conditional on
  nonabsorption for applications in conservation biology.
\newblock {\em The Annals of Applied Probability}, 11(1):261--284, 2001.

\bibitem{GrigorescuKang2004}
I.~Grigorescu and M.~Kang.
\newblock Hydrodynamic limit for a {F}leming-{V}iot type system.
\newblock {\em Stochastic Process. Appl.}, 110(1):111--143, 2004.

\bibitem{GroismanJonckheere2013}
P.~Groisman and M.~Jonckheere.
\newblock Simulation of quasi-stationary distributions on countable spaces.
\newblock {\em Markov Process. Related Fields}, 19(3):521--542, 2013.

\bibitem{HJS16}
C.~Holmgren, S.~Janson, and M.~{\v{S}}ileikis.
\newblock Multivariate normal limit laws for the numbers of fringe subtrees in
  $m$-ary search trees and preferential attachment trees.
\newblock {\em ArXiv:1603.08125}, 2016.

\bibitem{Janson04}
S.~Janson.
\newblock Functional limit theorems for multitype branching processes and
  generalized {P}\'olya urns.
\newblock {\em Stochastic Processes and Applications}, 110(2):177--245, 2004.

\bibitem{Janson03}
S.~Janson.
\newblock Asymptotic degree distribution in random recursive trees.
\newblock {\em Random Structures \& Algorithms}, 26(1-2):69--83, 2005.

\bibitem{Janson19}
S.~Janson.
\newblock Random replacements in {P}{\'o}lya urns with infinitely many colours.
\newblock {\em Electronic Communications in Probability}, 24, 2019.

\bibitem{LaruellePages2013a}
S.~Laruelle and G.~Pag\`es.
\newblock Nonlinear randomized urn models: a stochastic approximation
  viewpoint.
\newblock {\em Arxiv:1311.7367}, 2013.

\bibitem{LaruellePages2013}
S.~Laruelle and G.~Pag\`es.
\newblock Randomized urn models revisited using stochastic approximation.
\newblock {\em The Annals of Applied Probability}, 23(4):1409--1436, 2013.

\bibitem{MahmoudSmythe92}
H.~M. Mahmoud and R.~T. Smythe.
\newblock Asymptotic joint normality of outdegrees of nodes in random recursive
  trees.
\newblock {\em Random Structures \& Algorithms}, 3(3):255--266, 1992.

\bibitem{MW15}
H.~M. Mahmoud and M.~D. Ward.
\newblock Asymptotic properties of protected nodes in random recursive trees.
\newblock {\em Journal of Applied Probability}, 52(1):290--297, 2015.

\bibitem{MP16}
P.~Maillard and E.~Paquette.
\newblock Choices and intervals.
\newblock {\em Israel Journal of Mathematics}, 212(1):337--384, 2016.

\bibitem{MaillerMarckert2016}
C.~{Mailler} and J.-F. {Marckert}.
\newblock {Measure-valued P\'olya processes}.
\newblock {\em Electronic Journal of Probability}, 22, 2017.

\bibitem{Maric2015}
N.~Mari{\'{c}}.
\newblock Fleming--viot particle system driven by a random walk on
  $\mathbb{N}$.
\newblock {\em Journal of Statistical Physics}, 160(3):548--560, Aug 2015.

\bibitem{MartinezSanEtAl2013}
S.~{Mart\'inez}, J.~{San Mart\'in}, and D.~{Villemonais}.
\newblock {Existence and uniqueness of a quasi-stationary distribution for
  Markov processes with fast return from infinity}.
\newblock {\em Journal of Applied Probability}, 51(3), 2014.

\bibitem{MeleardVillemonais2012}
S.~M{\'e}l{\'e}ard and D.~Villemonais.
\newblock Quasi-stationary distributions and population processes.
\newblock {\em Probab. Surv.}, 9:340--410, 2012.

\bibitem{MeynTweedie1993}
S.~P. Meyn and R.~L. Tweedie.
\newblock Stability of {M}arkovian processes. {III}. {F}oster-{L}yapunov
  criteria for continuous-time processes.
\newblock {\em Advances in Applied Probability}, 25(3):518--548, 1993.

\bibitem{OcafrainVillemonais2017}
W.~O\c{c}afrain and D.~Villemonais.
\newblock Convergence of a non-failable mean-field particle system.
\newblock {\em Stoch. Anal. Appl.}, 35(4):587--603, 2017.

\bibitem{OliveiraDickman2006}
M.~M.~d. Oliveira and R.~Dickman.
\newblock {Quasi-stationary simulation: the subcritical contact process}.
\newblock {\em {Brazilian Journal of Physics}}, 36:685 -- 689, 09 2006.

\bibitem{Parthasarathy1967}
K.~R. Parthasarathy.
\newblock {\em Probability measures on metric spaces}.
\newblock Probability and Mathematical Statistics, No. 3. Academic Press, Inc.,
  New York-London, 1967.

\bibitem{Pemantle2007}
R.~Pemantle.
\newblock A survey of random processes with reinforcement.
\newblock {\em Probability Surveys}, 4:1--79, 2007.

\bibitem{Pouyanne2008}
N.~Pouyanne.
\newblock An algebraic approach to pólya processes.
\newblock {\em Ann. Inst. H. Poincaré Probab. Statist.}, 44(2):293--323, 04
  2008.

\bibitem{Renlund2010}
H.~{Renlund}.
\newblock {Generalized P\'olya urns via stochastic approximation}.
\newblock {\em Arxiv:1002.3716}, Feb. 2010.

\bibitem{StrookVaradhan}
D.~W. Stroock and S.~S. Varadhan.
\newblock {\em Multidimensional diffusion processes}.
\newblock Springer, 2007.

\bibitem{DoornPollett2013}
E.~A. van Doorn and P.~K. Pollett.
\newblock Quasi-stationary distributions for discrete-state models.
\newblock {\em European Journal of Operational Research}, 230(1):1--14, 2013.

\bibitem{Varadarajan1958}
V.~S. Varadarajan.
\newblock Weak convergence of measures on separable metric spaces.
\newblock {\em Sankhy\=a}, 19:15--22, 1958.

\bibitem{VerboomLankesterEtAl}
J.~Verboom, K.~Lankester, and J.~A.~J. Metz.
\newblock Linking local and regional dynamics in stochastic metapopulation
  models.
\newblock {\em Biological Journal of the Linnean Society}, 42(1‐2):39--55,
  1991.

\bibitem{Villemonais2015}
D.~Villemonais.
\newblock Minimal quasi-stationary distribution approximation for a birth and
  death process.
\newblock {\em Electronic Journal of Probability}, 20, 2015.

\bibitem{WangKolbEtAl2017}
A.~Q. {Wang}, M.~{Kolb}, D.~{Steinsaltz}, and G.~O. {Roberts}.
\newblock {Theoretical Properties of Quasistationary Monte Carlo Methods}.
\newblock {\em ArXiv e-prints}, July 2017.

\bibitem{WangRobertsEtAl2018}
A.~Q. {Wang}, G.~O. {Roberts}, and D.~{Steinsaltz}.
\newblock {An Approximation Scheme for Quasistationary Distributions of Killed
  Diffusions}.
\newblock {\em ArXiv e-prints}, Aug. 2018.

\bibitem{Zhang2016}
L.-X. Zhang.
\newblock Central limit theorems of a recursive stochastic algorithm with
  applications to adaptive designs.
\newblock {\em The Annals of Applied Probability}, 26(6):3630--3658, 2016.

\end{thebibliography}

\end{document}